\documentclass[12pt]{article}
\usepackage[utf8x]{inputenc}
\usepackage{amsmath}
\usepackage{amsfonts}
\usepackage{latexsym}
\usepackage{amsthm}
\usepackage{xargs}
\usepackage{graphicx}
\usepackage{stmaryrd}
\usepackage{color}
\usepackage{caption}
\usepackage{subcaption}
\usepackage{euscript}
\usepackage{url}
\usepackage{emptypage}

\newtheorem{theorem}{Theorem}[section]

\newtheorem{corollary}[theorem]{Corollary}

\newtheorem{lemma}[theorem]{Lemma}

\newtheorem{proposition}[theorem]{Proposition}
\newtheorem{remark}[theorem]{Remark}

\newcommand{\dist}{\kappa}
\newcommand{\coef}{\textup{d}}
\newcommand{\param}{\sigma}
\newcommand{\ost}{\omega}

\newcommand{\re}{\mathrm{Re}\,}
\newcommand{\im}{\mathrm{Im}\,}
\newcommand{\uns}{\mathrm{u}}
\newcommand{\sta}{\mathrm{s}}
\newcommand{\us}{*}

\newcommand{\vu}{u}
\newcommand{\vv}{v}
\newcommand{\vw}{w}

\newcommand{\xb}{\bar{x}}
\newcommand{\yb}{\bar{y}}
\newcommand{\zb}{\bar{z}}
\newcommand{\tb}{\bar{t}}

\newcommand{\hetz}{Z_0}
\newcommand{\hetr}{R_0}
\newcommand{\hetth}{\Theta_0}
\renewcommand{\r}{r}
\newcommand{\Fb}{\mathbf{F}}
\newcommand{\Gb}{\mathbf{G}}
\newcommand{\Hb}{\mathbf{H}}
\newcommand{\pu}{\eta}
\newcommand{\vt}{\vartheta}
\newcommand{\Diffw}{\mathcal{E}}

\newcommand{\Fout}{\mathcal{F}}
\newcommand{\Lout}{\mathcal{L}}
\newcommand{\Lpart}{\hat{\mathcal{L}}}
\newcommand{\Gout}{\mathcal{G}}
\newcommand{\Foutfp}{\tilde{\Fout}}

\newcommand{\Gpart}{\hat{\mathcal{G}}}
\newcommand{\Apart}{\mathcal{B}}

\newcommand{\Nconstnew}{\mathcal{A}_1}

\newcommandx{\Dout}[3][1=\dist, 2=\beta]{D_{#1,#2}^{#3}}
\newcommandx{\Doutinf}[1]{D_{\dist,\beta,\infty}^{#1}}
\newcommandx{\DoutT}[4][1=\dist, 2=T,3=\beta]{D_{#1,#3,#2}^{#4}}
\newcommand{\Doutinter}{D_{\dist,\beta}}
\newcommand{\Doutintertilde}{\tilde{D}_{\dist,\beta}}
\newcommandx{\Tout}[1][1=\ost]{\mathbb{T}_{#1}}

\newcommandx{\Bsout}[3][3=\ost]{\mathcal{X}^{#1}_{#2,#3}}
\newcommandx{\Bsoutdiff}[3][3=\ost]{\tilde{\mathcal{X}}^{#1}_{#2,#3}}
\newcommandx{\Bsoutceil}[3][3=\ost]{\bar{\mathcal{X}}^{#1}_{#2,#3}}

\title{Breakdown of a 2D heteroclinic connection in the Hopf-zero singularity (I)}
\author{I. Baldom\'a, O. Castej\'on, T. M. Seara}

\begin{document}
\maketitle

\section*{Summary}

In this paper we study a beyond all order phenomenon which appears in the analytic unfoldings of the Hopf-zero singularity.
It consists in the breakdown of a two-dimensional heteroclinic surface which exists in the truncated normal form of this singularity at any order.
The results in this paper are twofold: on the one hand we give results for generic unfoldings which
lead to sharp exponentially small upper bounds of the difference between these manifolds. On the other hand, we provide asymptotic formulas for
this difference by means of the Melnikov function for some non-generic unfoldings.

{\sl Keywords:} Exponentially small splitting, Hopf-zero bifurcation,
Melnikov function, Borel transform.

\newpage
\section{Introduction}
The Hopf-zero singularity (also called central singularity) is a vector field $X^*:\mathbb{R}^3\rightarrow\mathbb{R}^3$, having the origin as a critical point,
and such that the eigenvalues of the linear part at this point are $0$, $\pm i\alpha^*$, for some $\alpha^*\neq0$
so that $DX^*(0)$ has zero trace. This singularity has codimension two in the sense that
it can be met by a generic family of vector fields depending on at least two parameters. However, since $DX^*(0)$ has zero trace,
when one considers $X^*$ in the context of conservative vector fields,
it has codimension one. The generic families which meet the singularity $X^*$ for some value of the parameter, which we assume is $(0,0)$,
are called the versal unfoldings of the Hopf-zero singularity. That is, they are families
$X_{\mu,\nu}$ of vector fields on $\mathbb{R}^3$ depending on two parameters $(\mu,\nu)\in\mathbb{R}^2$, such that $X_{0,0}=X^*$,
the vector field described above. The conservative setting can be seen as a particular one taking $\nu=0$.
The generic case (with two parameters) is usually called the dissipative case.

The versal unfoldings of the Hopf-zero singularity have been widely studied in the past, see for example~\cite{BV84,G78,GR83,G85,Guc81,GH90,Tak73,Tak74}.
{In these works, }for generic singularities, depending on the region in the parameter space where $(\mu,\nu)$ belongs to,
the qualitative behavior of $X_{\mu,\nu}$ is studied.
However, there is one open region in the parameter space (see~\eqref{defU}) in  which the behavior of $X_{\mu,\nu}$
is not completely understood. These unfoldings  $X_{\mu,\nu}$ are the candidates to possess chaotic behavior.
In this work we study these unfoldings and prove, as a direct consequence of our results,
the existence of heteroclinic transversal curves between two equilibrium points of the phase space.

Let us to explain how this phenomenon is encountered in generic families $X_{\mu,\nu}$. Assume generic conditions on both, the
singularity $X^*$ and the unfolding $X_{\mu,\nu}$.
We follow the works~\cite{Guc81},~\cite{GH90} (see also~\cite{BCS13}) and we perform the normal form procedure
up to order two, some scalings and redefinitions of the parameters, and we obtain that
the normal form terms $X_{\mu,\nu}^2$ up to order two of these unfoldings in cylindric coordinates $(r,\theta,z)$ have the form:
\begin{equation}\label{FNO2polar}
\frac{d\bar{r}}{d\tb} =\bar{r}(\nu-\beta_1 \zb),\qquad \frac{d\theta}{d\tb}= \alpha_0+\alpha_1 \mu+\alpha_2 \nu + \alpha_3 z,
\qquad \frac{d\zb}{d\tb} = -\mu+\zb^2 +\gamma_2 \bar{r}^2.
\end{equation}
Assuming the generic condition on $X^*$ that $\beta_1,\gamma_2\neq0$, and rescaling variables if necessary one can assume that $\beta_1,\gamma_2>0$.
When $(\mu,\nu) \in U$, being
\begin{equation}\label{defU}
U=\{ (\mu,\nu)\in \mathbb{R}^2 : \mu>0, \; |\nu| <\beta_1 \sqrt{\mu}\},
\end{equation}
the truncation vector field $X_{\mu,\nu}^2$ has $(\xb,\yb,\zb)=(0,0,\pm \sqrt{\mu})$ as critical points of saddle-focus type connected by a heteroclinic orbit
\begin{equation}\label{W1}
W_1=\{ \xb=\yb=0,\; |\zb|\leq \sqrt{\mu}\}
\end{equation}
and, only when $\nu=0$, also by a two dimensional heteroclinic surface (see Figure~\ref{figNF}):
\begin{equation}\label{W1W2}
\qquad
W_2=\left \{ \zb^2 + \frac{\gamma_2}{\beta_1 +1} \bar{r}^2=\mu \right \}.
\end{equation}
\begin{figure}
	\centering
	\includegraphics[width=7cm]{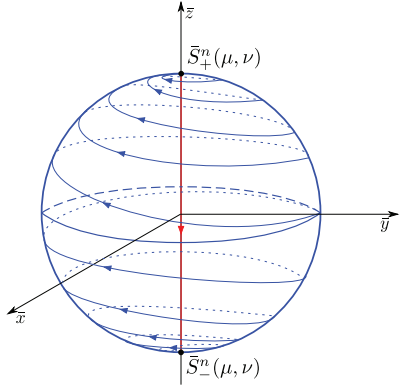}
       \caption[Phase portrait of the vector field $X_{\mu,\nu}^n$ for $(\mu,\nu)\in\Gamma_n$, for any $n\in\mathbb{N}$.]{Phase portrait of the vector field $X_{\mu,\nu}^n$ for $(\mu,\nu)\in\Gamma_n$, for any $n\in\mathbb{N}$. In red and blue, the one- and two dimensional heteroclinic connections respectively. The domain bounded by the two dimensional heteroclinic connection has size $\mathcal{O}(\sqrt{\mu})$.}\label{figNF}
\end{figure}

Concerning the whole vector field $X_{\mu,\nu}=X_{\mu,\nu}^2+F_{\mu,\nu}^2$, it also has two critical points
close to the ones of $X_{\mu,\nu}^2$, which  also are of saddle-focus type. However it is reasonable to expect that the heteroclinic
connections will no longer persist in $X_{\mu,\nu}$, see Figure~\ref{figdist1d2d}.
Performing classical perturbation theory, namely, considering the vector field $X_{\mu,\nu}$ as a perturbation of $X_{\mu,\nu}^2$
and following Poincar\'e-Melnikov method~\cite{Mel63},
one could try to measure the size of the breakdown of these heteroclinic connections. Unfortunately, the Melnikov function turns out to be
exponentially small in $\sqrt{\mu}$ when $\mu, \nu$ are small and therefore we can not directly use this technique in our case.
Let us explain the heuristic reason for this exponentially small behavior.

\begin{figure}
       \centering
	    \begin{subfigure}[b]{0.47\textwidth}
     	\centering
      	\includegraphics[width=6cm]{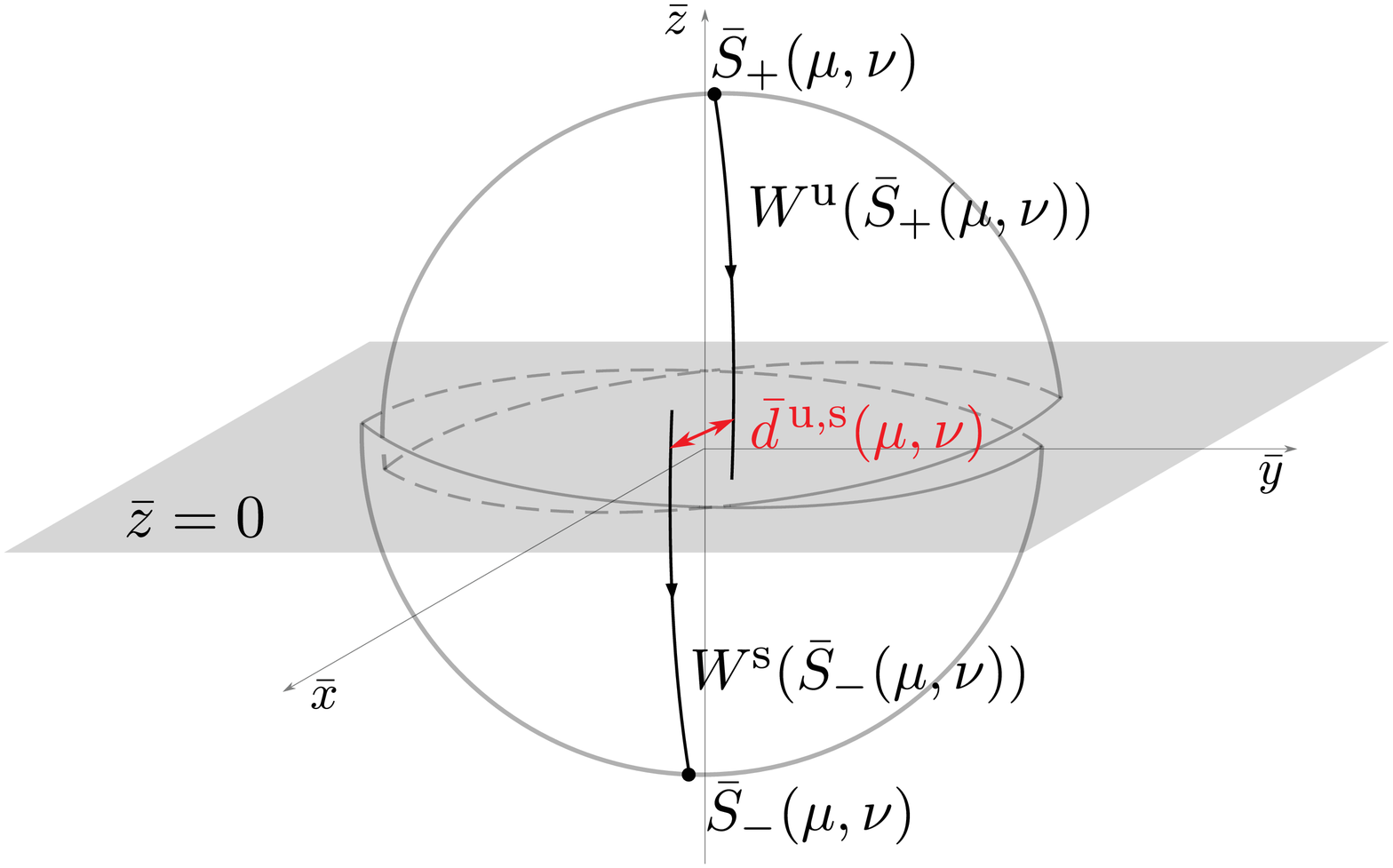}
        \caption{The distance $\bar d^{\uns,\sta}(\mu,\nu)$ between the one-dimensional invariant manifolds of $\bar{S}_+(\mu,\nu)$ and $\bar{S}_-(\mu,\nu)$.}
				\label{figdist1d}
        \end{subfigure}
			   \quad
        \begin{subfigure}[b]{0.47\textwidth}
	      \centering
		     \includegraphics[width=6cm]{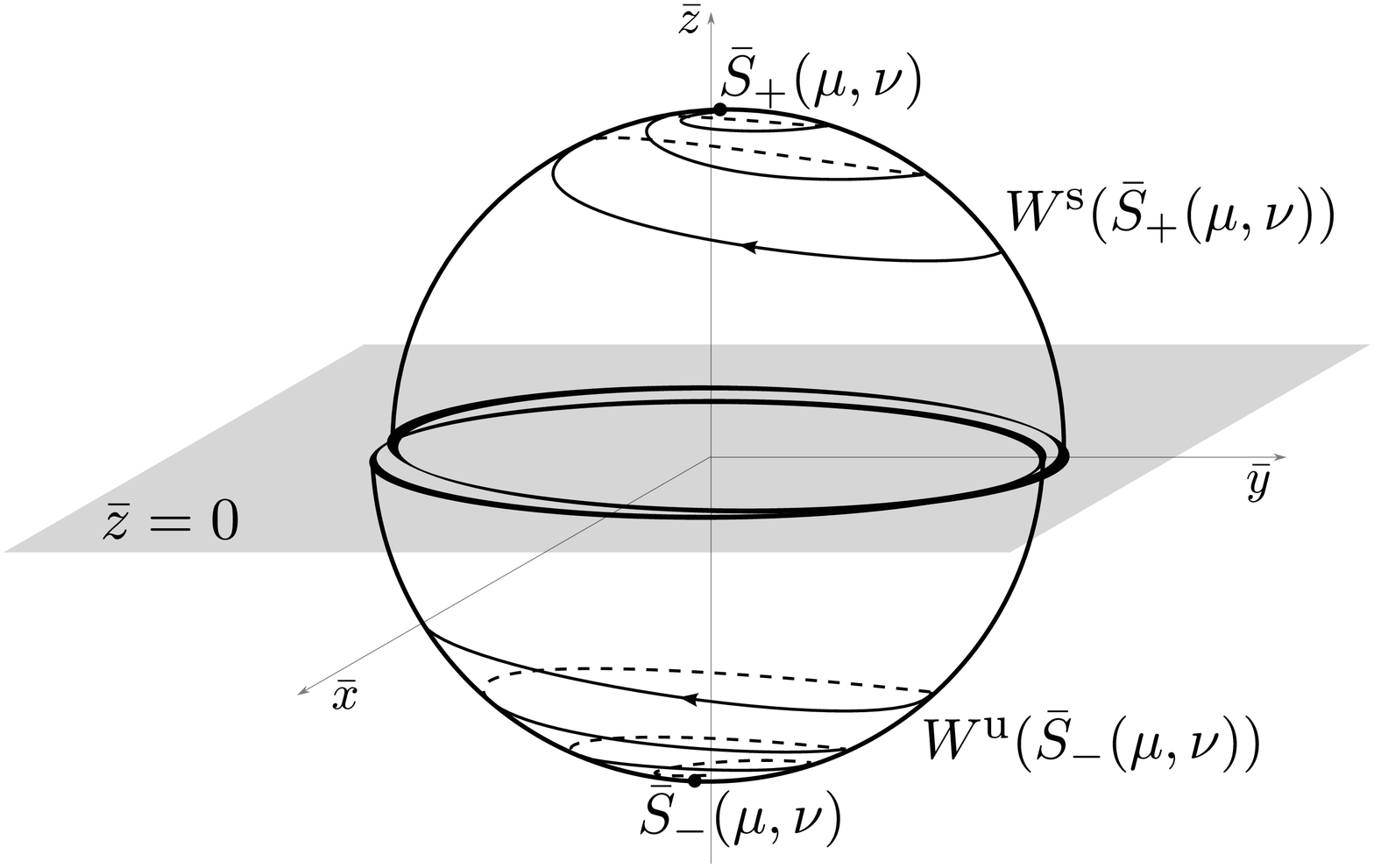}
	     	\caption{The two dimensional invariant manifolds of $\bar{S}_+(\mu,\nu)$ and $\bar{S}_-(\mu,\nu)$ until they reach the plane $\zb=0$.}\label{figdist2d-xyz}
        \end{subfigure}%
				\caption{The distance between the invariant manifolds}\label{figdist1d2d}
\end{figure}

If we denote by $X_{\mu,\nu}^n$, the truncation of the normal form of order $n$, which is a polynomial of degree $n$, then
\begin{equation}\label{VectorField}
X_{\mu,\nu}=X_{\mu,\nu}^n+F_{\mu,\nu}^n, \quad \mbox{where } \ F_{\mu,\nu}^n(\xb,\yb,\zb)=\mathcal{O}_{n+1}(\xb,\yb,\zb,\mu,\nu).
\end{equation}
If $(\mu,\nu)\in U$ the truncation of the normal
form $X^n_{\mu,\nu}$ has again two saddle-focus critical points $\bar{S}_{\pm}^n=(0,0,\mathcal{O}(\sqrt{\mu}))$
connected by an heteroclinic orbit and a two dimensional heteroclinic surface
(see again Figure~\ref{figNF}) for \emph{any} finite order $n$. The existence of the heteroclinic surface is only guaranteed
if the parameters $(\nu,\mu)$ belong to a curve $\Gamma_n$ of the form $\nu=m\mu+\mathcal{O}(\mu^{3/2})$ in the non conservative case whereas
it always exists in the conservative case.

Therefore, when $(\mu,\nu)$ are close to $\Gamma_n$ (or $\nu=0$ in the conservative case),
classical perturbation theory assures that the vector field $X_{\mu,\nu}=X_{\mu,\nu}^n+F_{\mu,\nu}^n$
has two critical points $\bar{S}_\pm(\mu,\nu) =\mathcal{O}(\sqrt{\mu})$ close to the ones of $X_{\mu,\nu}^n$, of saddle-focus type but the heteroclinic
connections will be generically destroyed in $X_{\mu,\nu}$. Obviously, the breakdown of these heteroclinic connections cannot be detected
in the truncation of the normal form at any finite order and therefore, as it is usually called, is a phenomenon \emph{beyond all orders}.

Since $X_{\mu,\nu}=X_{\mu,\nu}^n+F_{\mu,\nu}^n$, the breakdown of the heteroclinic connections, when $\nu=\mathcal{O}(\mu)$
close to some appropriate curve,
must be caused by the remainder $F_{\mu,\nu}^n$, which is of order
$\mathcal{O}_{n+1}(\xb,\yb,\zb,\mu,\nu)$. On the one hand, the heteroclinic connections are inside a domain in $\mathbb{R}^3$ of size $\mathcal{O}(\sqrt{\mu})$,
so that in this region we have: $F_{\mu,\nu}^n(\xb,\yb,\zb)=\mathcal{O}_{n+1}(\sqrt{\mu},\nu)=\mathcal{O}_{n+1}(\sqrt{\mu})$.
On the other hand, since this is valid for all $n$, the distance
between the invariant manifolds should be smaller than any finite power of the perturbation parameter $\sqrt{\mu}$. For this reason, one expects this
distance to be exponentially small in one of the perturbation parameters ($\sqrt{\mu}$ in fact) when the analytic case is considered.
Note that, in the dissipative case (that is, when two parameters are considered) one expects that the distance between the two dimensional
invariant manifolds is exponentially small only when $(\mu,\nu)$ is close to a certain curve.

The one-dimensional heteroclinic connection was studied (in the conservative setting) in~\cite{BaSe06} for some non-generic unfoldings
and the generic case in both the conservative and  the dissipative setting, in~\cite{BCS13}. In both cases, an asymptotic formula for the distance
between the one-dimensional invariant manifolds of the critical points $\bar{S}_+(\mu,\nu)$ and $\bar{S}_-(\mu,\nu)$ (see Figure~\ref{figdist1d}), is derived.

Let $\bar D^{\uns,\sta}(\theta,\mu,\nu)$, be the distance between the two dimensional invariant manifolds of the critical points
$\bar{S}_{\pm}(\mu,\nu)$ at the plane $\zb=0$ (see Figures~\ref{figdist2d-xyz} and~\ref{figdist2d-z0}).
Our final goal is to provide asymptotic formulas for this quantity.
However, due to the technical complications to deal with this exponentially small phenomenon, we have
split the whole proof in two papers, the present work and~\cite{BCS16b}. In the former, we provide asymptotic formulas for $\bar D^{\uns,\sta}(\theta,\mu,\nu)$,
when non-generic analytic unfoldings are considered whereas for generic unfoldings we provide
sharp upper bounds. In the later we give the asymptotic formula in the generic case.
It is worth mentioning that all the proofs in this work are also true for the generic case thus, in~\cite{BCS16b}, some results
derived in this work will be used.

\begin{figure}
	\centering
	\includegraphics[width=6cm]{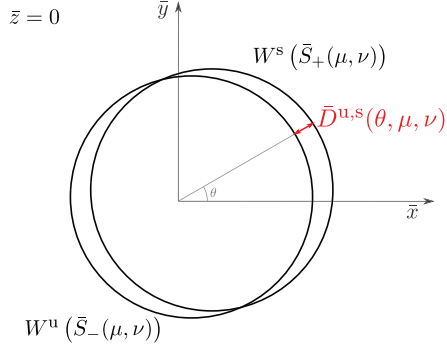}
	\caption{The intersection between the invariant manifolds and the plane $\zb=0$, and the distance
	$\bar D^{\uns,\sta}(\theta,\mu,\nu)$ between them.}\label{figdist2d-z0}
\end{figure}

\subsection{The regular vs. the singular case. Main result}
As it is well known by experts in the field, in order to obtain  asymptotic formulas for the breakdown of the two dimensional invariant manifolds
of $\bar{S}_+(\mu,\nu)$ and $\bar{S}_-(\mu,\nu)$, one needs to obtain suitable parameterizations of these invariant manifolds not only on real domains,
but also over complex ones.
These domains need to be $\mathcal{O}(\sqrt{\mu})$--close to the singularities of the corresponding heteroclinic connection of the unperturbed
system $X^2_{\mu,\nu}$. We recall that the one dimensional heteroclinic connection $W_1$ in~\eqref{W1}, exists for any $(\mu,\nu)$ whereas
the two dimensional one $W_2$, given in~\eqref{W1W2}, only exists when $\nu=0$.
Let us notice that the heteroclinic surface $W_2$ can be parameterized by
$(t,\theta)$ (in cylindric coordinates),
by:
\begin{equation}\label{firstparameterization}
\bar{r}=\bar{r}(t,\theta)=\frac{\sqrt{\mu}}{\cosh^2 \big(\beta_1 t \big)},\quad \theta=\theta, \quad \zb = \zb(t,\theta)=\tanh \big(\beta_1 t\big),
\qquad t\in \mathbb{R},\;\theta \in [0,2\pi].
\end{equation}
Observe that, $\bar{r}(\bar{t}\sqrt{\mu},\theta),\bar{z}(\bar{t}\sqrt{\mu},\theta)$ are solutions of~\eqref{FNO2polar}.

Clearly this parameterization has singularities at $\beta_1 t = \pm i \pi/2$.
Far from these singularities, we will find parameterizations of the invariant manifolds for $X_{\mu,\nu}$
which are well approximated by the unperturbed
heteroclinic connection~\eqref{firstparameterization}, but this will not be the case close to the singularities. This yields some technical difficulties.

A good way to start the study of the invariant manifolds in these complex domains is considering smaller perturbations of the vector field
$X_{\mu,\nu}^2$.
That is, one introduces a new parameter $q\geq 0$ and considers the following (artificial) vector field,
$X_{\mu,\nu}^2 + (\sqrt\mu)^{q} F_{\mu,\nu}$. Unfortunately, this is not
the good choice.
Indeed, we need to consider as  perturbation terms also the homogeneous terms of order two depending only on the
parameters $\mu,\nu$.
We call these terms $P_{\mu,\nu}^2$ and we write $X_{\mu,\nu}^2 = \widetilde{X}_{\mu,\nu}^2 + P_{\mu,\nu}^2$.
Then we define
\begin{equation}\label{Xreg}
X_{\mu,\nu}^{\rm reg}:=\widetilde{X}_{\mu,\nu}^2 + (\sqrt\mu)^q (P_{\mu,\nu}^2+ F_{\mu,\nu}^2).
\end{equation}
Clearly, for $q=0$ we recover~\eqref{VectorField}, while for $q>0$ the perturbation terms are smaller than those
in~\eqref{VectorField}. We call the case $q>0$ the \emph{regular case}, while $q=0$ is the \emph{singular} one. The first case represents
just a special subset of unfoldings of $X^*$, while the latter one represents a generic family of unfoldings of $X^*$.

Imposing the
condition $q>0$, one can see that the heteroclinic connections of the unperturbed system $\widetilde{X}_{\mu,\nu}^2$
give good approximations of the invariant manifolds, even close to their singularities.
The asymptotic formulas measuring the breakdown of the heteroclinic surface in this case consist on suitable versions of the so-called Melnikov
integrals (see~\cite{GH90,Mel63}). Thus, one can start studying the regular case to gain some intuition without getting lost with technical problems and,
after that, one can proceed with the singular case. This is what we have done in the present paper.

More precisely we have proven:
\begin{theorem}\label{mainthm-intro}
Consider unfoldings $X_{\mu,\nu}$ of the form~\eqref{Xreg} with $q\geq 0$ and $(\mu,\nu) \in U$ defined in~\eqref{defU}.
Then $X_{\mu,\nu}$ has two critical points $\bar S_\pm(\mu,\nu)$ of saddle-focus type of the form
$$
\bar{S}_{\pm}(\mu,\nu) = (0,0,\pm \sqrt{\mu})+ \mathcal{O}(\mu^2 + \nu^2)^{\frac{q+1}{2}}.
$$
In addition, $\bar{S}_{+}$ has a two dimensional stable manifold and $\bar{S}_-$ has a two dimensional unstable manifold.

For any $\vu \in \mathbb{R}$, let $\bar D^{\uns,\sta}(\vu,\theta,\mu,\nu)$ ($\bar D^{\uns,\sta}(\vu,\theta,\mu)$ in the conservative case)
be the distance between the two dimensional unstable manifold of $\bar S_-(\mu,\nu)$ and the two dimensional stable manifold of
$\bar S_+(\mu,\nu)$ when they meet the plane $\zb=\sqrt{\mu} \tanh(\beta_1\vu)$ (see~\eqref{firstparameterization}).

Then, there exist constants $\mathcal{C}_1$, $\mathcal{C}_2$ (see their formula in Theorem~\ref{thmasyformulaMelnikov}) and $L_0$
(see its formula in Remark~\ref{rmkL0L}) in such a way that, given $T_0>0$, for all $\vu\in[-T_0,T_0]$ and $\theta\in\mathbb{S}^1$, introducing
the function:
$$
\bar\vt(\vu,\mu)=\frac{\alpha_0\vu}{\sqrt{\mu}}+\frac{1}{\beta_1}\left[\alpha_3 \log\cosh(\beta_1\vu)-
\left (\alpha_3+\alpha_0 L_0 (\sqrt \mu )^{q} \right )\log \sqrt \mu\right],
$$
the following holds:
\begin{enumerate}
\item In the conservative case, which corresponds to $\beta_1=1$ and $\nu=0$, as $\mu\to 0^+$,
\begin{eqnarray*}
 \bar D^{\uns,\sta}(\vu,\theta,\mu)&=&\sqrt{\frac{\gamma_2}{2}}  \frac{e^{-\frac{\alpha_0\pi}{2\sqrt\mu}}}{(\sqrt \mu)^{3-q}}\cosh^{3}(\vu)
\Bigg[\mathcal{C}_1\cos\Big(\theta+\bar\vt(\vu,\mu)\Big)\\
&&+\mathcal{C}_2\sin\Big(\theta+\bar\vt(\vu,\mu)\Big)+\mathcal{O}\left( (\sqrt \mu)^{q}+(\sqrt \mu)^{3}\right)\Bigg].
\end{eqnarray*}

\item In the dissipative case, there exists a function $\nu=\nu_0(\mu)=\mathcal{O}((\sqrt{\mu})^{q+2})$, such that,
as $\mu\to 0^+$,
\begin{align*}
 \bar D^{\uns,\sta}(\vu,\theta,\,\mu,\nu_0(\mu))=&\sqrt{\frac{\gamma_2}{\beta_1+1}}
\frac{e^{-\frac{\alpha_0\pi}{2\beta_1\sqrt\mu}}}{(\sqrt\mu)^{\frac{2}{\beta_1}+1-q}}
\cosh^{1+\frac{2}{\beta_1}}(\beta_1\vu)
\Bigg[\mathcal{C}_1\cos\Big(\theta+\bar\vt(\vu,\mu)\Big)\\
&+\mathcal{C}_2\sin\Big(\theta+\bar\vt(\vu,\mu)\Big)+\mathcal{O}\left((\sqrt \mu)^{q}+(\sqrt \mu)^{3}\right)\Bigg].
\end{align*}
\end{enumerate}
\end{theorem}

\begin{remark}
Notice that, if we take $q=0$ in Theorem~\ref{mainthm-intro}, the error terms are not small but $\mathcal{O}(1)$ so our result provides sharp upper bounds
even in this case. Later, in Section~\ref{subsec-upperbound} these upper bounds are proven in a different (and easier) way.
To obtain asymptotic formulas when $q=0$, one needs to deal with the so called \emph{inner equation} as it is done in~\cite{BCS16b}.
\end{remark}

\begin{remark}\label{rmkmthm}
In the dissipative case, a more general result is indeed proven: for given $a_1,a_2\in\mathbb{R}$ and $a_3>0$, there exists a function $\nu=\nu(\mu)$
(depending on $a_1,a_2$ and $a_3$) satisfying $\nu(\mu)-\nu_0(\mu) = \mathcal{O}\big (\mu^{a_2}
e^{-\frac{a_3\pi}{2\beta_1\sqrt\mu}}\big)$, such that,
as $\mu\to 0^+$,
\begin{align*}
 \bar D^{\uns,\sta}(\vu,\theta,\,\mu,\nu(\mu))= &\bar D^{\uns,\sta}(\vu,\theta,\mu,\nu_0(\mu)) \\ &+ a_1\cosh^{1+\frac{2}{\beta_1}}(\beta_1\vu)\mu^{a_2}
e^{-\frac{a_3\pi}{2\beta_1\sqrt\mu}}\left(1+\mathcal{O}\left( (\sqrt \mu )^{q+1}\right)\right).
\end{align*}
The result in Theorem~\ref{mainthm-intro} corresponds to $a_1=0$.
\end{remark}
\subsection{Exponentially small splitting of invariant manifolds}
The formulas given in Theorem~\ref{mainthm-intro} prove that the breakdown of the heteroclinic connection is exponentially small
in the perturbation parameter $\sqrt \mu$ when $\nu = \nu_0(\mu)$, therefore this work deals with
the so-called exponentially small splitting of separatrices problem.

This problem was already considered the
\emph{fundamental problem of mechanics} by Poincar\'e in his famous work~\cite{Poincare}. There he studied Hamiltonian systems with
two and a half degrees of freedom and realized that this phenomenon was responsible for the creation of chaotic behavior.
He considered a model which, after reduction, became the perturbed pendulum:
$$\ddot y=2\mu\sin y+2\mu\varepsilon\cos y\cos t.$$
Using what later has become known as the Melnikov method (although Poincar\'e was actually the first one to use it, being rediscovered
by Melnikov more than 70 years later) he proved that the splitting of the separatrices is exponentially small in $\mu$, provided that
$\varepsilon$ is smaller than some exponentially small quantity. Of course, this latter assumption is enormously restrictive, but many
years had to go by until it could be removed.

When studying exponentially small phenomena, one cannot use classical perturbation methods. Over the last decades more
sophisticated techniques have been developed mainly for Hamiltonian systems and area preserving maps.
Indeed, this problem was not studied from a rigorous point of view until the end of the 80s and during the 90s.
Neishtadt~\cite{Neishtadt} gave upper bounds for the splitting in Hamiltonian systems of one and a half and two degrees of freedom
and Lazutkin~\cite{Lazutkin} (see
also~\cite{Gel99}) was the first
to give an asymptotic expression of the splitting angle between the stable and unstable manifolds for the standard map.

After Lazutkin's paper, some works gave bounds for the splitting
for rapidly forced systems, \cite{HMS88}, and \cite{FS90-2,FS90}, for area-preserving diffeomorphisms close to the
identity.

Later on,
asymptotic formulas for several examples were obtained being the first ones~\cite{KS91, DS92,Gel94}.
After these pioneering works, partial results for general Hamiltonian systems were given
in~\cite{DS97,Gel97,BF04,BF05}. A new approach that has had much influence in posterior studies of exponentially small splitting was introduced
in~\cite{Sau01,LMS03}. It is important to note that, besides~\cite{Lazutkin} and~\cite{KS91},
all the examples cited above deal with the so-called regular case, in which some artificial
condition about the smallness of the perturbation is required. In this case the Melnikov method
gives the correct size of the splitting.

In the singular case one often has to study a certain equation independent of parameters, usually called the \emph{inner equation}.
There are a few works dealing with this kind of equations in different settings, see~\cite{Gel97-2, GS01, OSS03, BaldomaInner, BaSe08, BM12}, but,
besides the works of Lazutkin and Kruskal and Segur, there are very few works with rigorous proofs in the singular case for Hamiltonian systems
(see for instance~\cite{Tre97, Gel00, GOS10, BFGS12,Guar13}) or conservative maps, \cite{GB08, MSS11}.
Numerical results about the splitting in the Hamiltonian setting can be found in~\cite{BO93,Gel97-2}, for two dimensional symplectic maps
in~\cite{DR-R99,GS08,SV09,MSV13} and in~\cite{GSV13} the splitting is computing for two-dimensional manifolds in four dimensional symplectic maps. 
In~\cite{GG11}, the authors study the Hamiltonian-Hopf bifurcation (a Hamiltonian version of the singularity studied in this paper) combining numerical and analytical techniques.

All these works deal with either Hamiltonian systems or symplectic maps. To the best of our knowledge,
the papers~\cite{La03} and~\cite{Lom00} are
the only works concerning  exponentially small splitting of separatrices in a non-Hamiltonian setting, where reversible systems are considered
and~\cite{Fo95} gives results for dissipative perturbations of Hamiltonian systems. Also in~\cite{BCS13}
an asymptotic formula to measure the breakdown of the one-dimensional heteroclinic connection ($W_1$ in~\eqref{W1}) is proven in the singular case
for any value of $(\mu,\nu)$ small enough.

It is worth mentioning that the setting of this paper is not similar to any of the works computing exponentially small splitting of invariant manifolds.
Indeed, here we do not deal with a Hamiltonian system, the flow of the vector field might not be volume preserving (since we consider not only the
conservative setting but also the dissipative one) and it is not a reversible system. For this reason,
some new ideas had to be used in order to prove the results found in this paper.

\subsection{The Shilnikov bifurcation and the Hopf-zero singularity}
To finish this introduction, let us to mention that our results can lead to prove the existence of Shilnikov bifurcations, \cite{Shil65},
in suitable unfoldings $X_{\mu,\nu}$.
Indeed, the existence of such Shilnikov bifurcations for $\mathcal{C}^\infty$ unfoldings of the Hopf-zero singularity is studied in~\cite{BV84}.
Doing the normal form procedure up to order infinity and using Borel-Ritt's theorem, the vector
field $X_{\mu,\nu}$ can be decomposed as $X_{\mu,\nu}=X_{\mu,\nu}^\infty+F_{\mu,\nu}^\infty$,
where $X_{\mu,\nu}^\infty$ has the same phase portrait as the vector field $X_{\mu,\nu}^n$ described above
(Figure~\ref{figNF}) and
$F_{\mu,\nu}^\infty=F_{\mu,\nu}^\infty(x,y,z)$ is a flat function at the origin.
Their strategy consists in constructing suitable perturbations $p_{\mu,\nu}^\infty$,
which are also flat functions,
such that the heteroclinic connections of the family $X_{\mu,\nu}^\infty$ are destroyed
and some homoclinic ones appear giving rise to the so-called Shilnikov bifurcation.
Therefore, an existence theorem is given, but the results do not provide conditions to check whether a
concrete family $X_{\mu,\nu}$ possesses or not a Shilnikov bifurcation.

The case of real analytic unfoldings of the singularity $X^*$ has been open since then.
It is possible that the strategy of Broer and Vegter can be adapted to the analytic case.
Of course one cannot consider flat perturbations, but suitable perturbations could be constructed (although not straightforwardly)
following~\cite{BT86} and~\cite{BT89}.
However, another strategy must be followed if given \textit{any} unfolding $X_{\mu,\nu}$ one wants to determine
whether it will or not possess a sequence of Shilnikov bifurcations. The key point, as in the $\mathcal{C}^{\infty}$ case, is to check if the given
unfolding $X_{\mu,\nu}$ does not have the aforementioned heteroclinic connections.

Progress was made recently in~\cite{diks}, where the authors prove the equivalent result as~\cite{BV84}
in the real analytic context assuming some upper and lower bounds of the distance between the invariant manifolds of $\bar S_+(\mu,\nu)$ and $\bar S_-(\mu,\nu)$.
In particular, the authors assume that the heteroclinic connections are destroyed and quantitative information about the splitting is required among other assumptions.
Our work computes asymptotic formulas of the splitting of these invariant manifolds which, as a consequence, allow to check if the corresponding assumptions
in~\cite{diks} are satisfied. We leave the complete proof of the existence of Shilnikov bifurcations for a future work.

The paper is organized as follows.
In Section~\ref{ChapterHopfZero2D-outer}, we expose the strategy we will follow to prove Theorem~\ref{mainthm-intro} by
enunciating, without proving, the main results we will need, namely,
i) existence of suitable parameterizations of the invariant manifolds in complex domains,
ii) derivation and computation of the Melnikov function,
iii) expression, in complex domains, for the difference between the invariant manifolds and
iv) the exponentially small formulas for the difference in real domains. After that, still in Section~\ref{ChapterHopfZero2D-outer}, we prove Theorem~\ref{mainthm-intro}
as a consequence of Theorem~\ref{mainthm}. We postpone all the technical demonstrations to Sections~\ref{sec:exponentially}-\ref{secproofDif}.

\section{The regular case. Heuristic of the proof}\label{ChapterHopfZero2D-outer}
Following the same strategy as the one presented in~\cite{BCS13}, which involves normal form changes, scalings and redefinitions
of parameters we can write $X^{\rm reg}_{\mu,\nu}$, in~\eqref{Xreg}, in its normal form of order three, namely:
\begin{align}\label{sistema-NForder3}
 \frac{d\xb}{d\tb}&=\xb\left(\nu-\beta_1\zb\right)+\yb\left(\alpha_0+\alpha_1\nu+\alpha_2\mu+\alpha_3\zb\right)+ (\sqrt{\mu})^q\bar{f}(\xb,\yb,\zb,\mu,\nu),\medskip\nonumber\\
\frac{d\yb}{d\tb}&=-\xb\left(\alpha_0+\alpha_1\nu+\alpha_2\mu+\alpha_3\zb\right)+\yb\left(\nu-\beta_1\zb\right)+(\sqrt{\mu})^q\bar{g}(\xb,\yb,\zb,\mu,\nu),\medskip\\
\frac{d\zb}{dt}&=-\mu+\zb^2+\gamma_2(\xb^2+\yb^2)+(\sqrt{\mu})^q\bar{h} (\xb,\yb,\zb,\mu,\nu)\nonumber
\end{align}
with $\bar{f}, \bar{g}$ and $\bar{h}$ analytic functions in $B(\bar r_0)^3 \times B(\bar{\mu}_0) \times B(\bar{\nu}_0)\subset \mathbb{C}^3 \times \mathbb{C}^2$,
\begin{align*}
\bar{f}(\xb,\yb,\zb,\mu,\nu)
&=\xb \bar{A}(\xb^2+\yb^2,\zb,\mu,\nu) + \yb \bar{B}(\xb^2+\yb^2,\zb,\mu,\nu)
+ \mathcal{O}_4(\xb,\yb,\zb,\mu,\nu)
\\
\bar{g}(\xb,\yb,\zb,\mu,\nu) &= \yb \bar{A}(\xb^2+\yb^2,\zb,\mu,\nu)-\xb \bar{B}(\xb^2+\yb^2,\zb,\mu,\nu)
+ \mathcal{O}_4(\xb,\yb,\zb,\mu,\nu)
\\
\bar{h}(\xb,\yb,\zb,\mu,\nu)&=\gamma_3\mu^2+\gamma_4\nu^2+\gamma_5\mu\nu+ \bar{C}(\xb^2+\yb^2,\zb,\mu,\nu) + \mathcal{O}_4(\xb,\yb,\zb,\mu,\nu)
\end{align*}
and $\bar{A}$, $\bar{B}$ and $\bar{C}$ some functions satisfying:
\begin{equation*}
 \xb \bar{A},\,\xb \bar{B},\, \yb \bar{A},\,\yb \bar{B},\,\bar{C}=\mathcal{O}_3(\xb,\yb,\zb,\mu,\nu),
\end{equation*}
when they are evaluated in their arguments.

\begin{remark}\label{remarkFormanormal3}
In~\cite{BCS13} when the breakdown of the one-dimensional heteroclinic connection was considered, we performed the normal form up to order two.
In this work we need to perform an additional step of the normal procedure for technical reasons, which will be totally
understood later on, see Section~\ref{subsubsec:settingparam}, even when the terms of order three do not appear explicitly nor in our hypotheses neither in our results.
\end{remark}

In the remaining part of this section we give the main ideas of the proof of Theorem~\ref{mainthm-intro}.
The rest of the paper is devoted to prove the results stated in this section.
We now summarize the subsections that can be found
in this Section, each one consisting in one step of the proof of Theorem~\ref{mainthm-intro}.
The first step, explained in detail in Section~\ref{subsecpreliminary},
consists on scale variables and introduce the new parameters  $\delta=\sqrt{\mu}$, $\param = \delta^{-1} \nu$ and call $p=q-2$ as in~\cite{BCS13}.
Still in this preliminary section, we give a parameterization of the heteroclinic connection of the unperturbed system which corresponds to the normal
form of order two.
In Section~\ref{secthmoutloc} we give parameterizations of the 2-dimensional invariant manifolds adequate to our purposes.
In Section~\ref{subsecintromelnikov}, we introduce and study the Melnikov function adapted to this problem.
This Melnikov function will be the dominant term in the difference between the invariant manifolds.
After that, in Section~\ref{secDifference}, we give some properties of this difference which allows us, in Section \ref{subsec-upperbound}, to give a sharp
upper bound of this difference.
Finally, in Section~\ref{subsec:firstorderdiff-outer}, we state and prove Theorem~\ref{mainthm} which is equivalent to Theorem~\ref{mainthm-intro}.

\subsection{Preliminary considerations}\label{subsecpreliminary}
This subsection is mainly devoted to fix notation and perform straightforward changes of variables
to put the vector field $X^{\rm reg}_{\mu,\nu}$, in~\eqref{Xreg}, in a suitable way to work with.
Moreover, we also study what we call the unperturbed system.
\subsubsection{Notation, scalings and set up}
We scale system~\eqref{sistema-NForder3} as in~\cite{BCS13}, namely
we define the new parameters $p=q-2$, $\delta=\sqrt{\mu}$, $\param=\delta^{-1}\nu$ and
we rename the coefficients $b=\gamma_2$, $c=\alpha_3$ and $\coef=\beta_1$.
We also introduce the constant $h_3$ from $\bar{h}$ given by
$$\bar{h}(0,0,\zb,0,0)= h_3 \zb^3 + \mathcal{O}(\zb^4).$$
In the new variables
$x=\delta^{-1}\xb$, $y=\delta^{-1}\yb$, $z=\delta^{-1}\zb+\delta^{p+3}h_3 /2$ and $t=\delta\tb$, system~\eqref{sistema-NForder3} becomes:
\begin{equation}\label{initsys-outer2D}
 \begin{array}{rcl}
  \displaystyle\frac{dx}{dt}&=&\displaystyle x\left(\param-\coef z\right)+\left(\frac{\alpha(\delta^2,\delta\param)}{\delta}+cz\right)y+\delta^{p}f (\delta x,\delta y, \delta z, \delta,\delta\param),\medskip\\
  \displaystyle\frac{dy}{dt}&=&\displaystyle-\left(\frac{\alpha(\delta^2,\delta\param)}{\delta}+cz\right)x+y\left(\param-\coef z\right)+\delta^{p}g (\delta x,\delta y, \delta z, \delta,\delta\param),\medskip\\
  \displaystyle\frac{dz}{dt}&=&-1+b(x^2+y^2)+z^2+\delta^{p}h (\delta x, \delta y, \delta z, \delta,\delta\param),
 \end{array}
\end{equation}
where
$\alpha(\delta^2,\delta\param)=\alpha_0+\alpha_1\delta\param+\alpha_2\delta^2$ with $\alpha_0\neq 0$ and $f,g$ and $h$ are
the corresponding ones to $\bar{f},\bar{g}$ and $\bar{h}$.
To shorten the notation we write system~\eqref{initsys-outer2D} as
\begin{equation}\label{initsys-outer2DX}
\frac{d \zeta }{dt} = X(\zeta,\delta,\param)= X_{0}(\zeta,\delta,\param) + \delta^p X_1(\delta \zeta, \delta, \delta \param),\qquad \zeta=(x,y,z).
\end{equation}

From now on, we will omit the dependence of $\alpha$ with respect to $\delta$ and $\param$.
\begin{remark}
Recall that $b>0$, $\coef>0$, $\delta>0$ is a small parameter and $|\param|<\coef$. Without loss of generality, we assume that $\alpha_0$ and $c$ are positive constants.
In particular, for $\delta$ small enough, $\alpha(\delta^2,\delta\param)$ will be also positive.
\end{remark}

Since the functions $\bar{f},\bar{g}$ and $\bar{h}$ are real analytic the same happens for $X_1$. We
call $B^3(r_0)\times B(\delta_0)\times B(\param_0) \subset \mathbb{C}^3 \times \mathbb{C}^2$ its analyticity domain.

\subsubsection{Unperturbed system: $\param=0$, $X_1\equiv 0$}\label{sec:unperturbed}
Consider system~\eqref{initsys-outer2D} with $\param=0$, $f=g=h=0$. It is clear that it has
rotational symmetry. For our purposes it will be very useful to consider ``symplectic'' cylindric coordinates:
\begin{equation}\label{cylindriccoordsunpertub}
x=\sqrt{2r}\cos\theta,\qquad y=\sqrt{2r}\sin\theta,\qquad z=z.
\end{equation}
The main reason is that this change of variables is divergence free, therefore in the conservative case, after this change of variables
the new vector field will be conservative too.
The unperturbed system  writes out as:
\begin{equation}\label{syspolarunperturb}
\frac{dr}{dt}=-2\coef rz,\qquad
\frac{d\theta}{dt}=-\frac{\alpha}{\delta}-cz, \qquad
\frac{dz}{dt}=-1+2br+z^2.
\end{equation}
Since $b>0$, the unperturbed system~\eqref{syspolarunperturb}, has a 2-dimensional heteroclinic manifold $\Gamma$ connecting
$S_+(\delta,0)=(0,0,1)$ and $S_-(\delta,0)=(0,0,-1)$ given by:
$$
\Gamma:=\left\{(r,z)\in\mathbb{R}^2\,:\,-1+\frac{2br}{\coef+1}+z^2=0\right\}.
$$
This manifold can be parameterized with $t\in\mathbb{R}$ by the solutions of the unperturbed system starting at time $t=0$ on the plane $z=0$ and with angular variable $\theta=\theta_0\in [0,2\pi)$ by:
\begin{eqnarray}
 r&=&\displaystyle\hetr(t):=\frac{(\coef+1)}{2b}\frac{1}{\cosh^2(\coef t)},\label{hetr}\medskip\\
 \theta&=&\hetth(t,\theta_0):=\theta_0-\frac{\alpha}{\delta}t-\frac{c}{\coef}\log\cosh(\coef t),\notag\medskip\\
z&=&\hetz(t):=\tanh(\coef t)\label{hetz}.
\end{eqnarray}

\begin{remark} For bounded $\Vert \xi \Vert$, with $\xi=(x,y,z,1,\param)$,
$f(\delta \xi), g(\delta \xi), h(\delta \xi)= \mathcal{O}(\delta^3)$. Thus, using classical perturbation methods,
one can easily see that the difference between the 2-dimensional invariant manifolds is of order $\mathcal{O}(\param)+\mathcal{O}(\delta^{p+3})$.
Therefore, if $\param$ is not of order $\delta^{p+3}$, this difference is not exponentially small in $\delta$.
For this reason, in the rest of the paper we assume that $|\param|\leq \param^*\delta^{p+3}$, for some constant
$\param^*$, since the exponentially small case is the only one where the Shilnikov phenomenon can occur, see~\cite{diks}.
\end{remark}

From now on, we will omit the dependence of $\alpha$ with respect to $\delta$ and $\param$.

\subsection{Local parameterizations of the invariant manifolds}\label{secthmoutloc}

System~\eqref{initsys-outer2DX} has two critical points $S_{\pm}(\delta,\param)$ of saddle-focus type, see~\cite{BCS13} for instance, and also
Lemma~\ref{lemchangCuCs}.
The goal in this subsection is to provide good parameterizations for the two dimensional invariant manifold associated to $S_{\pm}(\delta,\param)$.

It is useful to write system~\eqref{initsys-outer2DX} in symplectic cylindric coordinates~\eqref{cylindriccoordsunpertub}:
\begin{equation}\label{syspolar}
 \begin{array}{rcl}
  \displaystyle \frac{dr}{dt}&=&\displaystyle 2r(\param-\coef z)+\delta^p \Fb (\delta r,\theta,\delta z,\delta,\delta\param),\medskip\\
  \displaystyle \frac{d\theta}{dt}&=&\displaystyle-\frac{\alpha}{\delta}-cz+\delta^p \Gb (\delta r,\theta,\delta z,\delta,\delta\param),\medskip\\
  \displaystyle \frac{dz}{dt}&=&\displaystyle-1+2br+z^2+\delta^p \Hb (\delta r,\theta,\delta z,\delta,\delta\param),
 \end{array}
\end{equation}
where $\mathbf{X}_1=( \Fb, \Gb, \Hb)$ is defined as
\begin{equation}\label{defFGH}
\mathbf{X}_1 (\delta r, \theta, \delta z,\delta,\delta \param) = \left (
\begin{array}{ccc} \sqrt{2r} \cos \theta & \sqrt{2r} \sin \theta & 0 \\
-\frac{1}{\sqrt{2r}} \sin \theta & \frac{1}{\sqrt{2r}} \cos \theta & 0 \\
0 & 0 & 1 \end{array}\right ) X_1(\delta \zeta,\delta,\delta \param)
\end{equation}
being $\zeta = (\sqrt{2r} \cos \theta, \sqrt{2r}\sin \theta, z)$.

Let us explain how we construct good parameterizations of the invariant manifolds which will be solutions of the same equation.
As the experts in the field know this is a key point to prove the exponentially smallness of their difference.
Due to the geometry of the unperturbed system, it seems natural
to write them as graphs over $z$ and the angular variable $\theta$ (see Figure~\ref{figdist1d2d}).
However, we will not do exactly that, but instead we will introduce a new variable $\vu$ defined by:
$\vu=\hetz^{-1}(z)=\coef^{-1}\mathrm{atanh}(z)$, or equivalently $z=\hetz(\vu)$ (recall that $\hetz$ was defined in~\eqref{hetz}).
The invariant manifolds in symplectic polar coordinates will be parameterized by:
\begin{equation}\label{defrunssta}
r=r^{\uns,\sta} (\vu,\theta),\quad z=\hetz(\vu),
\end{equation}
or in Cartesian coordinates
$$
x=\sqrt{2r^{\uns,\sta}(\vu,\theta)}\cos \theta,\qquad y=\sqrt{2r^{\uns,\sta}(\vu,\theta)}\sin \theta,\qquad z=\hetz(\vu).
$$
This method, being very useful for our purposes, has some drawbacks. For example, it is obvious
that $z = \hetz(\vu)\to \pm 1$ as $\vu\to \pm \infty$. Thus, if the $z$-component of the critical points $S_{\pm}(\delta,\param)$ is not equal
to $\pm 1$ respectively, these parameterizations will not work for large values of $|\vu|$. Nevertheless, we will
prove that these parameterizations exist for bounded values of $u$.

Now we give the invariance equation that our parameterizations $r^{\uns,\sta}$ satisfy. To simplify the notation,
we introduce
\begin{equation}\label{notationFGHbis}
\bar{X}_1(r)(\vu,\theta)
=\mathbf{X}_1 (\delta (\hetr(\vu)+r(\vu,\theta)),\theta,\delta\hetz(\vu),\delta,\delta\param), \quad \bar{X}_1=(F,G,H)
\end{equation}
for a given function $r(\vu,\theta)$.
To avoid cumbersome notations, if there is not danger of confusion, we will omit the dependence on variables $(\vu,\theta)$.
Using this notation, the parameterizations $r^{\uns,\sta}$ have to satisfy the following PDE:
$$
\frac{d\theta}{dt}\partial_\theta r +\frac{d\vu}{dt}\partial_\vu r =2(\param-\coef\hetz(\vu))r +\delta^pF (r-\hetr(\vu)),
$$
and, using equations~\eqref{syspolar} and that $\frac{d\vu}{dt}=\coef^{-1}(1-\hetz^2(\vu))^{-1}\frac{dz}{dt}$:
\begin{align}\label{PDEinisimp}
\big(-\frac{\alpha}{\delta}-c\hetz(\vu) +\delta^p G(r-\hetr(\vu))\big)&\partial_\theta r+\left(\frac{-1+2br+\hetz^2(\vu)+\delta^p H(r-\hetr(\vu))}{\coef(1-\hetz^2(\vu))}\right)\partial_\vu r \notag\\
=&2(\param-\coef\hetz(\vu))r +\delta^p F(r-\hetr(\vu)).
\end{align}

Since it is reasonable to consider system~\eqref{syspolar} as a perturbation of the unperturbed system~\eqref{syspolarunperturb} ($\param=0$ and $\mathbf{X}_1=0$)
studied in Section~\ref{sec:unperturbed},
we impose that $r^{\uns,\sta} (\vu,\theta)=\hetr(\vu)+r_1^{\uns,\sta} (\vu,\theta)$,
where $\hetr$ is given in~\eqref{hetr}.
Using the relations
$$\hetr'(\vu)=-2\coef\hetr(\vu)\hetz(\vu),\qquad
-1+2b\hetr(\vu)+\hetz^2(\vu)=\coef(1-\hetz^2(\vu)),
$$
and putting all terms which are either small or
non-linear in $r_1^{\uns,\sta}$ in the right-hand side of the equality and the remaining terms in the left, equation~\eqref{PDEinisimp} writes out as
\begin{equation}\label{PDEequal}
\Lout(r_1)=\Fout (r_1),
\end{equation}
where $\Lout$ and $\Fout$ are the differential operators defined by:
\begin{align}
 \Lout(r):=&\left(-\delta^{-1}\alpha-c\hetz(\vv)\right)\partial_\theta r +\partial_\vu r -2\hetz(\vu) r, \label{defopL} \\
\Fout(r):=&2\param(\hetr(\vu)+r)+\delta^pF(r)+\delta^p\frac{\coef+1}{b}\hetz(\vu)H(r)\nonumber\\
&-\delta^pG(r)\partial_\theta r-\left(\frac{2br+\delta^pH(r)}{\coef(1-\hetz^2(\vu))}\right)\partial_\vu r.\label{defopF}
\end{align}

We now define the complex domains in which $r^{\uns,\sta}_1$ (and therefore $r^{\uns,\sta}$) will be defined.
We first deal with the unstable case.
We want these domains to be close to the singularities of the heteroclinic connection of the unperturbed system (see~\eqref{hetr}--\eqref{hetz}) closest to the real line.
These are $\pm\frac{i\pi}{2\coef}$. Moreover, it will be convenient that these domains have a triangular shape. To this aim, let $0<\beta<\pi/2$ and $\dist^*>0$
be two constants independent of $\delta$ and $\param$. Take $\dist=\dist(\delta)$ any function satisfying that for $0<\delta<1$:
\begin{equation}\label{conddist}
 \dist^*\delta\leq\dist\delta<\frac{\pi}{8\coef}.
\end{equation}
Then we define the domain (see Figure~\ref{figureDoutuns}):
\begin{equation}\label{defdoutuns}
\Dout{\uns}=\left\{\vv\in\mathbb{C} \,:\, |\im\vv|\leq \frac{\pi}{2\coef}-\dist\delta-\tan\beta\re\vv\right\}.
\end{equation}
We will split the domain $\Dout{\uns}$ in two subsets. Let $T>0$ be any constant independent of $\beta$, $\dist^*$, $\delta$ and $\param$. Then we define (see Figure~\ref{figureDoutuns}):
$$
  \Doutinf{\uns}=\left\{\vv\in \Dout{\uns}\, : \, \re\vv\leq-T\right\},\quad  \DoutT{\uns}=\left\{\vv\in \Dout{\uns}\, : \, \re\vv\geq-T\right\}.
$$
\begin{figure}
       \centering
        \begin{subfigure}[b]{0.45\textwidth}
	  \centering
		\includegraphics[width=6cm]{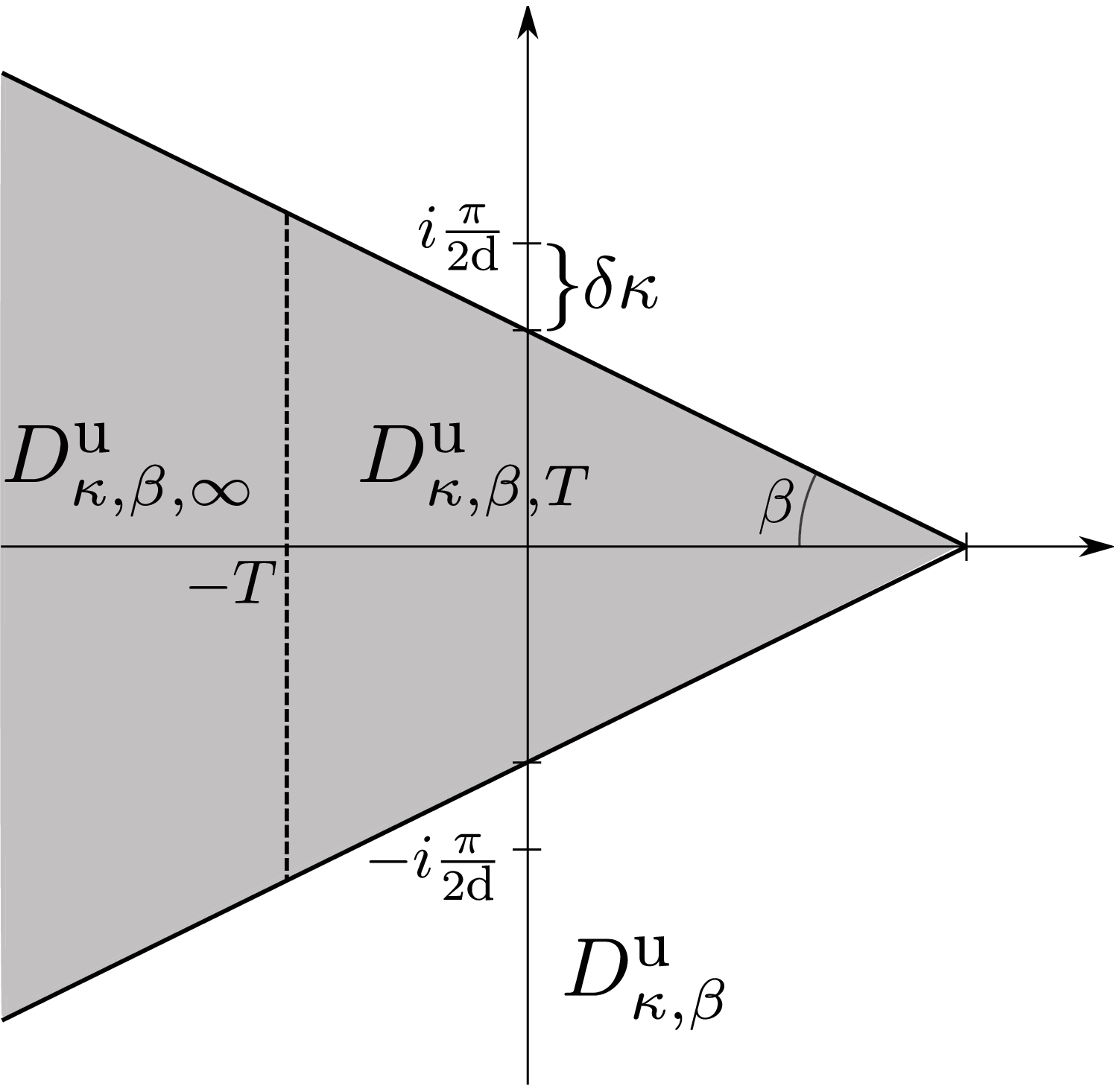}
		\caption{Outer domain $\Dout{\uns}$ for the unstable case with subdomains $\DoutT{\uns}$ and $\Doutinf{\uns}$.}
		 \label{figureDoutuns}
        \end{subfigure}%
         \qquad
        \begin{subfigure}[b]{0.45\textwidth}
	  \centering
		\includegraphics[width=6cm]{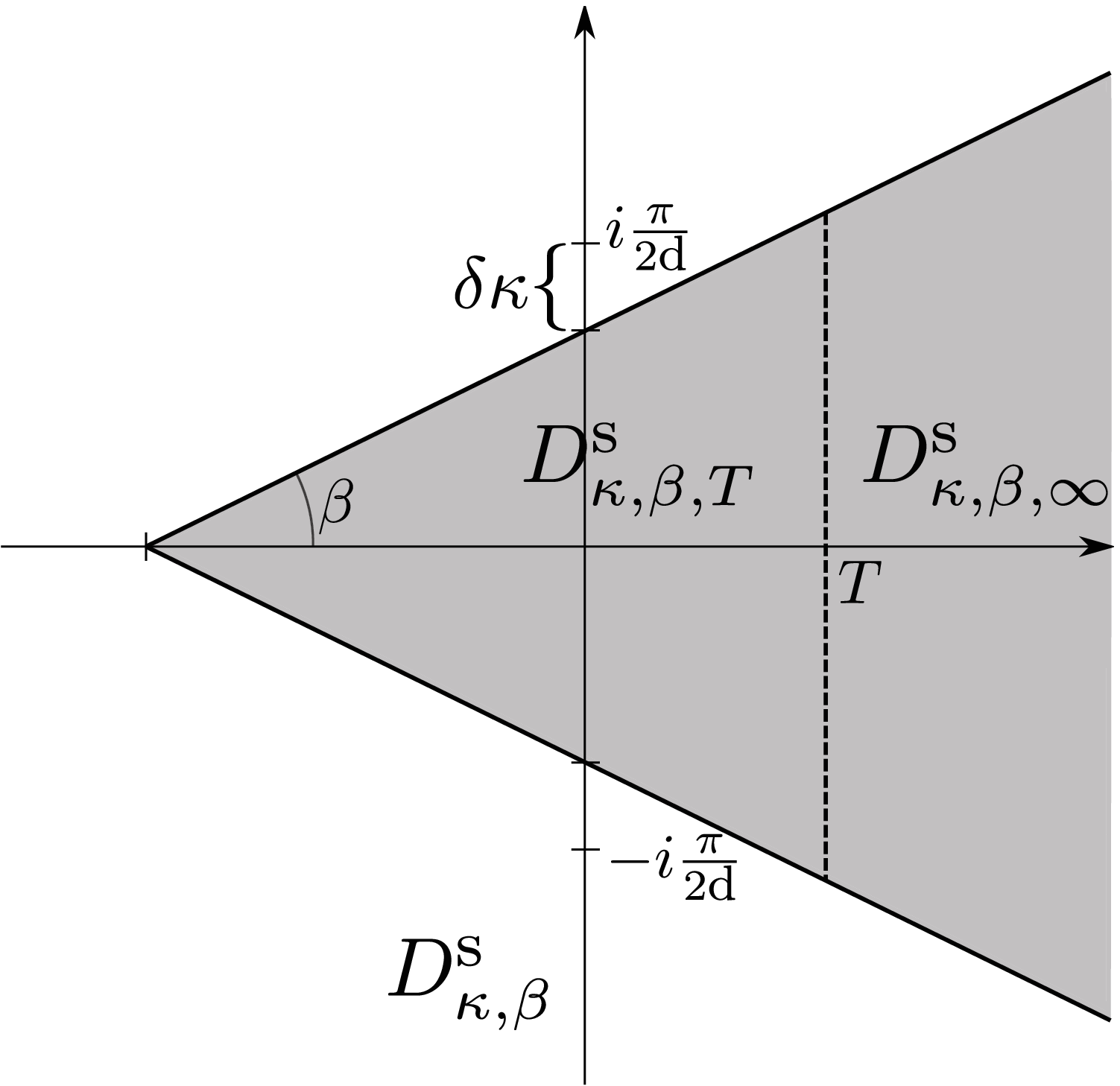}
		\caption{Outer domain $\Dout{\sta}$ for the stable case with subdomains $\DoutT{\sta}$ and $\Doutinf{\sta}$.}
		\label{figureDoutsta}
        \end{subfigure}
	\caption{The outer domains $\Dout{\uns}$ and $\Dout{\sta}$.}\label{figureDout}
\end{figure}
Analogously, for the stable case we define (see Figure~\ref{figureDoutsta}):
$$\Dout{\sta}=-\Dout{\uns},\qquad \Doutinf{\sta}=-\Doutinf{\uns},\qquad \DoutT{\sta}=-\DoutT{\uns}.$$
For any fixed real $\omega>0$, we also define the complex domains:
\begin{equation}\label{defTout}
\Tout:=\left\{\theta\in\mathbb{C}/(2\pi\mathbb{Z})\,:\,|\im\theta|\leq\ost\right\}.
\end{equation}

Next result gives the main properties of the functions $r_1^{\uns,\sta}$.
\begin{theorem}\label{thmoutloc}
Let $p\geq-2$ and $0<\beta<\pi/2$ be any constants. There exist $\dist^*\geq1$, $\param^*,\delta^*>0$, such that for all $0<\delta\leq\delta^*$,
if $\dist=\dist(\delta)$ satisfies
condition~\eqref{conddist} and $|\param|\leq\param^*\delta^{p+3}$, the unstable manifold of $S_-(\delta,\param)$ and the stable manifold of $S_+(\delta,\param)$
are given respectively by:
$$
\zeta^{\uns,\sta}(\vu,\theta)=(\sqrt{2r^{\uns,\sta}(\vu,\theta)}\cos\theta,\sqrt{2r^{\uns,\sta}(\vu,\theta)}\sin\theta,\hetz(\vu)),
\;\;(\vu,\theta)\in \DoutT{\uns,\sta} \times\Tout
$$
with
$r^{\uns,\sta}(\vu,\theta)=\hetr(\vu)+r_1^{\uns,\sta}(\vu,\theta)$
and $r_1^{\uns,\sta}$ satisfying equation~\eqref{PDEequal}.

Let us introduce
$$
\pu^{\mp}(w) = \alpha w \mp \delta( cw \mp c \coef^{-1} \log(1+ e^{\pm 2 \coef w})).
$$
We decompose $\r_1^{\uns,\sta}$ into
$r_1^{\uns,\sta}=r_{10}^\uns+r_{11}^{\uns,\sta}$
being
$$
r_{10}^{\uns,\sta}(\vu,\theta) = \cosh^{\frac{2}{\coef}}(\coef \vu) \int_{\mp \infty}^{u}
\frac{\Fout(0)\left (w,\theta-\delta^{-1} \big (\pu^{\mp}(w) - \pu^{\mp}(u)\big )\right)}{\cosh^{\frac{2}{\coef}}(\coef w)}dw,
$$
with $\Fout$ in~\eqref{defopF} and we take $-$ in the unstable case and $+$ in the stable one.

Then, there exists
$M>0$ such that for all $(\vu,\theta)\in \DoutT{\uns,\sta} \times\Tout$:
$$|r^{\uns,\sta}_{10}(\vu,\theta)|\leq M\delta^{p+3}|\cosh(\coef\vu)|^{-3} $$
$$|r^{\uns,\sta}_{11}(\vu,\theta)|\leq M\left(\delta^{2p+6}|\cosh(\coef\vu)|^{-4}+\delta^{p+4}|\cosh(\coef\vu)|^{-1}\right),$$
and:
$$
|\partial_\vu r_1^{\uns,\sta}(\vu,\theta)|\leq M\delta^{p+3}|\cosh(\coef\vu)|^{-4},\qquad |\partial_\theta r_1^{\uns,\sta}(\vu,\theta)|\leq
M\delta^{p+4}|\cosh(\coef\vu)|^{-4}.
$$
In addition,
the function $r_{10}^{\uns,\sta}$ is defined in the full domain $\Dout{\uns,\sta}\times \Tout$.
\end{theorem}
The proof of this Theorem can be found in Section~\ref{secproofthmoutloc}. We stress that this result is also valid in the singular case $p=-2$.

\subsection{The Melnikov function}\label{subsecintromelnikov}

Our final aim is to find an asymptotic formula of the difference $\Delta=r^\uns-r^\sta=r_{1}^{\uns}-r_{1}^{\sta}$. Recall that by Theorem~\ref{thmoutloc} we have:
$$
 \Delta(\vu,\theta)=r_{10}^\uns(\vu,\theta)-r_{10}^\sta(\vu,\theta)+r_{11}^\uns(\vu,\theta)-r_{11}^\sta(\vu,\theta).
$$
Also by Theorem~\ref{thmoutloc}, we know that $\r_{10}^\uns$ and $\r_{10}^\sta$ are larger than $\r_{11}^\uns$ and $\r_{11}^\sta$. Hence, it is natural to think that the first order of the difference is given by the difference of these dominant terms. That is, we expect that:
$$
\Delta(\vu,\theta)=\r_{10}^\uns(\vu,\theta)-\r_{10}^\sta(\vu,\theta)+\mathrm{h.o.t}.
$$
We will see that this approach is valid for $p>-2$, that is, for non-generic unfoldings. We postpone the study of the case $p=-2$
to~\cite{BCS16b}, where we will see that this assumption is not true.

Let us consider the difference $r_{10}^\uns-r_{10}^\sta$, which is $2\pi-$periodic in $\theta$, so that we can write its Fourier series:
\begin{equation}\label{defMelnikov}
M(\vu,\theta):=\r_{10}^\uns(\vu,\theta)-\r_{10}^\sta(\vu,\theta)=\sum_{l\in\mathbb{Z}}M^{[l]}(\vu)e^{il\theta}.
\end{equation}
We introduce $\pu(w)= \alpha w + \delta c\coef^{-1} \log(\cosh \coef w)$. We observe that for real values of $\vu,w$,
$\pu(w) - \pu(\vu)=\pu^{\mp}(w) - \pu^{\mp}(\vu)$,
with $\pu^{\mp}$ introduced in Theorem~\ref{thmoutloc}.
Therefore, using the formula for $r_{10}^{\uns,\sta}$ in the mentioned result, we have that
\begin{equation}\label{defMelnikovequiv}
 M(\vu,\theta)=\cosh^{\frac{2}{\coef}}(\coef\vu)\int_{-\infty}^{+\infty}\frac{\Fout(0)\left (w,\theta-
\delta^{-1} \big (\pu(w) - \pu(u)\big )
\right)}{\cosh^{\frac{2}{\coef}}(\coef w)}dw,
\end{equation}
which is the Melnikov function adapted to this problem.
As $\Fout(0)(\vu,\theta)$ is periodic in $\theta$, the coefficients $M^{[l]}(\vu)$ for $\vu\in\mathbb{R}$ are:
\begin{equation}\label{defcoefMelnikovtilde}
M^{[l]}(\vu)=
\cosh^{\frac{2}{\coef}}(\coef\vu)e^{il\delta^{-1} \pu(u)}\int_{-\infty}^{+\infty}\frac{e^{-il\delta^{-1} \pu(w)}
\Fout^{[l]}(0)(w)}{\cosh^{\frac{2}{\coef}}(\coef w)}dw.
\end{equation}
Moreover, from~\eqref{defcoefMelnikovtilde} it is clear that we can write series~\eqref{defMelnikov} as:
\begin{equation}\label{melnifourier2}
 M(\vu,\theta)=\cosh^{\frac{2}{\coef}}(\coef\vu)\sum_{l\in\mathbb{Z}}\Upsilon_0^{[l]}e^{il(\theta+\delta^{-1}\alpha\vu+c\coef^{-1}\log\cosh(\coef\vu))},
\end{equation}
where $\Upsilon_0^{[l]}$ are the constants:
\begin{equation}\label{defcoefMelnikov}
 \Upsilon_0^{[l]}=\int_{-\infty}^{+\infty}\frac{e^{-il(\delta^{-1}\alpha w+c\coef^{-1}\log\cosh(\coef w))}\Fout^{[l]}(0)(w)}{\cosh^{\frac{2}{\coef}}(\coef w)}dw.
\end{equation}
In addition
$M^{[l]}(\vu)=\cosh^{\frac{2}{\coef}}(\coef\vu)e^{il(\delta^{-1}\alpha\vu+c\coef^{-1}\log\cosh(\coef\vu))}\Upsilon_0^{[l]}.$

In the following theorem we provide upper bounds for $\Upsilon_0^{[l]}$ for $|l|\ge 2$ and
closed formulas for $\Upsilon_0^{[1]}$ and $\Upsilon_0^{[-1]}$ in terms of Borel transform of some functions depending
on the perturbation terms. We also prove that (besides the average $\Upsilon_0^{[0]}$) they are the dominant coefficients of $M$.
To this purpose, we recall that given a function $m(w,\theta)=\sum_{n\geq0}m_n(\theta)w^{n+1+ik}$, periodic in $\theta$,
we define its Borel transform $\hat{m}(\zeta,\theta)$ as:
\begin{equation}\label{defBorel}
\hat{m}(\zeta,\theta)=\sum_{n\geq0}m_n(\theta)\frac{\zeta^{n+ik}}{\Gamma(n+1+ik)}.
\end{equation}
To avoid a cumbersome notation, we introduce
$$
\textbf{w}(w,\theta) = \left(\sqrt{\frac{\coef+1}{b}}w\cos \theta ,\sqrt{\frac{\coef+1}{b}}w\sin \theta ,-iw,0,0\right )
$$
and $\tilde{F}(w,\theta)=
\cos \theta f(\textbf{w}(w,\theta))+\sin \theta g(\textbf{w}(w,\theta))$ with $f$ and $g$ the perturbation terms in system~\eqref{initsys-outer2D}.
\begin{theorem}\label{thmasyformulaMelnikov}
Consider the $2\pi$-periodic in $\theta$ function
\begin{equation}\label{defm}
m(w,\theta)= \sqrt{\frac{\coef + 1}{b}} w^{1+\frac{2}{\coef}+i\frac{c}{\coef}}
\left (\tilde{F}(w,\theta)
-i \sqrt{\frac{\coef + 1}{b}} h(\mbox{\rm \textbf{w}}(w,\theta))\right).
\end{equation}
Let $\hat{m}(\zeta,\theta)$ be its Borel transform as defined in~\eqref{defBorel} and  $\hat{m}^{[1]}$ its first Fourier coefficient.
Then, writing
$\mathcal{C}=\mathcal{C}_1 - i\mathcal{C}_2 = \frac{4\pi}{\coef}\hat{m}^{[1]}\left(\frac{\alpha}{\coef}\right)$, $\mathcal{C}_1, \mathcal{C}_2\in \mathbb{R}$,
\begin{equation*}
\Upsilon_0^{[1]} = \overline{\Upsilon_0^{[-1]}}=\delta^{p-\frac{2}{\coef}-i\frac{c}{\coef}}e^{-\frac{\alpha\pi}{2\coef\delta}}\left( \frac{\mathcal{C}}{2} + \mathcal{O}(\delta)\right).
\end{equation*}
Moreover, there exists a constant $K$ such that:
\begin{equation}\label{boundcoeffl}
 \left|\Upsilon_0^{[l]}\right|\leq K\delta^{p-\frac{2}{\coef}}e^{-\frac{\alpha\pi}{2\coef\delta}\frac{3|l|}{4}},\qquad |l|\geq 2.
\end{equation}

In conclusion, defining
$\vt(\vu,\delta)=\delta^{-1}\alpha\vu+c\coef^{-1}\left[\log\cosh(\coef\vu)-\log\delta\right]$
for $\vu\in\mathbb{R}$ and $\theta\in\mathbb{S}^1$ one has that:
\begin{align}\label{asyMelni}
M(\vu,&\theta)=\cosh^{\frac{2}{\coef}}(\coef\vu)\bigg[\Upsilon_0^{[0]}+\\
&\delta^{p-\frac{2}{\coef}}e^{-\frac{\alpha\pi}{2\coef\delta}}\Big(\mathcal{C}_1\cos(\theta+\vt(\vu,\delta))+
\mathcal{C}_2\sin(\theta+\vt(\vu,\delta))+\mathcal{O}(\delta)\Big)\bigg].\notag
\end{align}
\end{theorem}

The proof of this result can be found in Section~\ref{sectionMelnikov}.

Due to the exponential smallness of $\Upsilon_0^{[l]}$, $|l|\geq1$, the dominant term of the Melnikov function for real values of $\vu$ is
its average $\Upsilon_0^{[0]}$.
We will give more details about this coefficient in Section~\ref{subsec:firstorderdiff-outer}, Theorem~\ref{thmktilde0}.

\subsection{The difference}\label{secDifference}
In this section we study the difference $\Delta(\vu,\theta)=r_1^\uns(\vu,\theta)-r_1^\sta(\vu,\theta)$.
We give only the main result and some intuitive ideas of the proof. For all the details we refer the reader to Section~\ref{secproofDif}.

First, we find an equation for the difference $\Delta$.
To this aim, we subtract the PDEs~\eqref{PDEequal} for
$r_1^{\uns}$ and $r_1^{\sta}$, and then using the mean value theorem, we obtain an equation of the following form:
\begin{align}\label{PDE_difference_intro}
 \big(-\delta^{-1}\alpha-c\hetz(\vu)\big)\partial_\theta\Delta+\partial_\vu\Delta-2\hetz(\vu)\Delta=&(2\param+l_1(\vu,\theta))\Delta\notag \\
&+l_2(\vu,\theta)\partial_\vu\Delta+l_3(\vu,\theta)\partial_\theta\Delta.
\end{align}
Here the functions $l_1, l_2, l_3$, are functions which are ``small'' in the appropriate sense. More precisely,
denoting $r_\lambda=(r_1^\uns+r_1^\sta)/2+\lambda(r_1^\uns-r_1^\sta)/2$ and applying the mean value theorem, the functions $l_i$ are:
\begin{align}
 l_1(\vu,\theta)=&\frac{\delta^p}{2}\int_{-1}^1\partial_rF(r_\lambda)d\lambda+
\frac{\delta^{p}(\coef+1)}{2b}\hetz(\vu)\int_{-1}^1\partial_r H(r_\lambda)d\lambda\notag\\
&-
\frac{\delta^p}{2}\int_{-1}^1\partial_rG(r_\lambda)\partial_\theta r_\lambda d\lambda
-\frac{\delta^p}{2\coef(1-\hetz^2(\vu))}\int_{-1}^1\partial_rH(r_\lambda)\partial_\vu r_\lambda d\lambda\notag\\
&-\frac{b}{\coef(1-\hetz^2(\vu))}(\partial_\vu r_1^\uns+\partial_\vu r_1^\sta),\label{defl1}\\
l_2(\vu,\theta)=&-\frac{b}{\coef(1-\hetz^2(\vu))}(r_1^\uns+r_1^\sta)-\frac{\delta^p}{2\coef(1-\hetz^2(\vu))}\int_{-1}^1H(r_\lambda)d\lambda\label{defl2},\\
l_3(\vu,\theta)=&-\frac{\delta^p}{2}\int_{-1}^1G(r_\lambda)d\lambda.\label{defl3}
\end{align}
The precise meaning of ``small'' will be given in Lemma~\ref{lemnormint}.

Recall that $r_1^{\uns}$ and $r_1^{\sta}$ are defined respectively in the domains $\DoutT{\uns}\times\Tout$ and $\DoutT{\sta}\times\Tout$.
Thus, their difference will be defined in the intersection of these two domains.
So, from now on we will consider $(\vu,\theta)\in\Doutinter\times\Tout$, where we define $\Doutinter$ as (see Figure~\ref{figureDoutinter}):
\begin{equation}\label{Doutinter}
\Doutinter=\DoutT{\uns}\cap\DoutT{\sta}.
\end{equation}
\begin{figure}
\center
 \includegraphics[width=5cm]{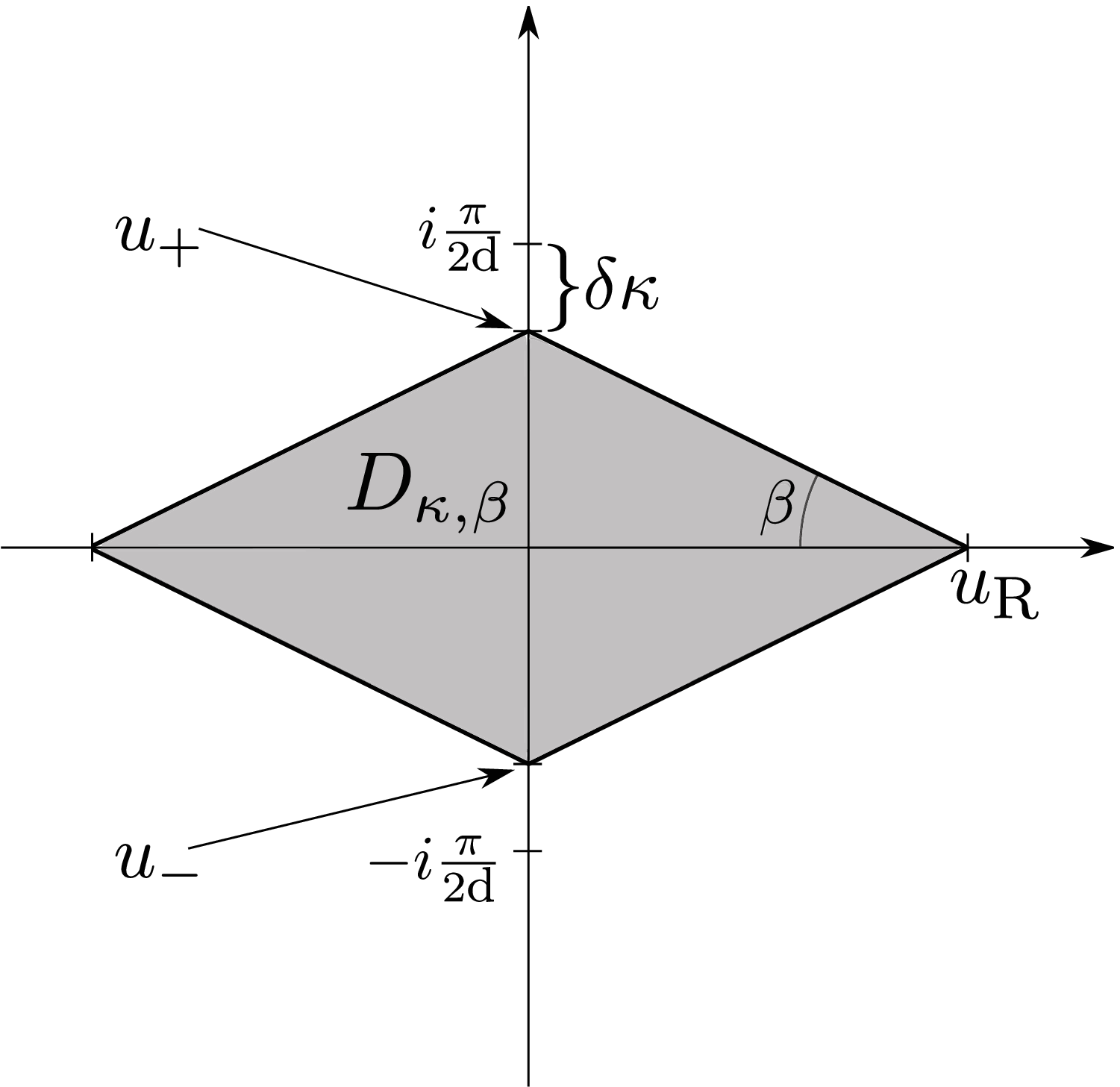}
 \caption{The domain $\Doutinter$.}
 \label{figureDoutinter}
\end{figure}

Now, we study \textit{all} the solutions of equation~\eqref{PDE_difference_intro}.
First we notice that, by the so-called method of variation of constants,
every solution $\Delta$ of~\eqref{PDE_difference_intro} can be written as:
\begin{equation}\label{Deltaforma1}
\Delta(\vu,\theta)=P(\vu,\theta)k(\vu,\theta),
\end{equation}
where $P$ is a particular solution of this same equation satisfying $P(\vu,\theta)\neq0$, and $k(\vu,\theta)$ satisfies the associated homogeneous PDE:
\begin{equation}\label{PDE_k}
 \left(-\delta^{-1}\alpha-c\hetz(\vu)\right)\partial_\theta k+\partial_\vu k=l_2(\vu,\theta)\partial_\vu k+l_3(\vu,\theta)\partial_\theta k.
\end{equation}
Let us now mention some properties of these functions $k$ and $P$.

To study the function $k$ we shall rely on the form of equation~\eqref{PDE_k}.
One of its main features is that if $\xi$ is a particular solution of~\eqref{PDE_k} such that $(\xi(\vu,\theta),\theta)$
is injective in $\Doutinter\times\Tout$, then \textit{any} solution $k$ of~\eqref{PDE_k} can be written as:
$$k(\vu,\theta)=\tilde{k}(\xi(\vu,\theta)),$$
for some function $\tilde k(\tau)$. As a consequence of~\eqref{Deltaforma1} and of the above equality
\begin{equation}\label{Deltaforma2}
\Delta(\vu,\theta) = P(\vu,\theta) \tilde{k}(\xi(\vu,\theta))
\end{equation}
with $P$ a particular solution of~\eqref{PDE_difference_intro} and $\xi$ a particular solution of~\eqref{PDE_k}.

Since the functions $l_i$ are ``small'', equation~\eqref{PDE_k} is a perturbation of:
$$ \left(-\delta^{-1}\alpha-c\hetz(\vu)\right)\partial_\theta k+\partial_\vu k=0.$$
A solution of this equation is given by $\xi_0(\vu,\theta)=\theta+\delta^{-1}\alpha\vu+c\coef^{-1}\log\cosh(\coef\vu)$.
Then, we look for a solution of~\eqref{PDE_k} of the form:
\begin{equation}\label{formaxi_intro}
\xi(\vu,\theta)=\theta+\delta^{-1}\alpha\vu+c\coef^{-1}\log\cosh(\coef\vu)+C(\vu,\theta),
\end{equation}
where, as expected, $C$ will be a ``small'' function.

Notice that, if the existence of $\xi$ of the form~\eqref{formaxi_intro} can be proven then the function $\tilde{k}(\tau)$ has
to be $2\pi$-periodic in its argument. Indeed, since $k$ is $2\pi$-periodic in $\theta$ one has that
$\tilde{k}(\xi(\vu,\theta+2\pi))= \tilde{k}(\xi(\vu,\theta))$. The claim follows from the fact that
$\xi(\vu,\theta+2\pi)=\xi(\vu,\theta)+2\pi$.

To study the particular solution $P$ of~\eqref{PDE_difference_intro} we note that, being $\param=\mathcal{O}(\delta^{p+3})$
and $l_i$ ``small'', equation~\eqref{PDE_difference_intro} is a perturbation of:
$$\left(-\delta^{-1}\alpha-c\hetz(\vu)\right)\partial_\theta \Delta+\partial_\vu\Delta-2\hetz(\vu)\Delta=0.$$
A solution of this equation is given by $P_0(\vu)=\cosh^{2/\coef}(\coef\vu)$. Therefore, we look for a particular
solution of~\eqref{PDE_difference_intro} of the form:
\begin{equation*}
P(\vu,\theta)=\cosh^{2/\coef}(\coef\vu)(1+P_1(\vu,\theta)),
\end{equation*}
where $P_1(\vu,\theta)$ will be ``small''.

As a conclusion of all the previous considerations, one obtains the following result, which characterizes the form of the difference $\Delta$ as well as the sizes of the functions
$P_1$ and $C$ described above.
\begin{theorem}\label{thmdifpartsolinjective}
Let $p\geq-2$ and $|\param|\leq\delta^{p+3}\param^*$. The difference $\Delta$ can be written as:
\begin{equation}\label{defDelta}
\Delta(\vu,\theta)=P(\vu,\theta) \tilde{k}(\xi(\vu,\theta))=\cosh^{2/\coef}(\coef\vu)(1+P_1(\vu,\theta))\tilde{k}(\xi(\vu,\theta)),
\end{equation}
where $\tilde k(\tau)$ is a $2\pi-$periodic function, the function $\xi$ is a solution of~\eqref{PDE_k} defined as:
\begin{equation}\label{defxitheorem}
\xi(\vu,\theta)=\theta+\delta^{-1}\alpha\vu+c\coef^{-1}\log\cosh(\coef\vu)+C(\vu,\theta),
\end{equation}
and it is such that $(\xi(\vu,\theta),\theta)$ is injective in $\Doutinter\times\Tout$. In addition, $P$ is a solution of~\eqref{PDE_difference_intro} and
$P_1$ and $C$ are real analytic functions, defined in $\Doutinter\times\Tout$ such that:
\begin{enumerate}
 \item There exist $L_0\in \mathbb{R}$ and functions $L(\vu)$ and $\chi(\vu,\theta)$ such that
\begin{equation}\label{formaCtheorem}
C(\vu,\theta) = \delta^{p+2} \coef^{-1} \alpha L_0 \log \cosh (\coef \vu) + \alpha L(\vu) + \chi(\vu,\theta),
\end{equation}
where, for all $(\vu,\theta) \in \Doutinter\times\Tout$:
\begin{equation}\label{boundLchi-prop}
|L(\vu)| \leq M\delta^{p+2},\quad |L'(\vu)| \leq M\delta^{p+2}, \quad |\chi(\vu,\theta)| \leq \frac{M\delta^{p+3}}{|\cosh (\coef \vu)|},
\end{equation}
for some constant $M$. $L_0$ and $L(u)$ are determined by a finite number of Taylor coefficients of the functions $f$, $g$ and $h$
appearing in~\eqref{initsys-outer2D}.  Formulae for $L_0$
and $L(\vu)$ are given in Remark~\ref{rmkL0L}.
Moreover, $L(0)=0$ and $L(\vu)$ is defined on the limit $\vu\to \pm i\pi/(2\coef)$.
\item There exists a constant $M$ such that for all  $(\vu,\theta)\in\Doutinter\times\Tout$:
\begin{equation}\label{boundP1-prop}
|P_1(\vu,\theta)|\leq \frac{M\delta^{p+3}}{|\cosh(\coef\vu)|}.
\end{equation}

Moreover, in the conservative case $P_1$ can be chosen as:
$$P_1(\vu,\theta)=\frac{\partial_\vu C(\vu,\theta)-l_3(\vu,\theta)}{\delta^{-1}\alpha+c\hetz(\vu)+l_3(\vu,\theta)},$$
where $l_3(\vu,\theta)$ is given by~\eqref{defl3}.
\end{enumerate}
\end{theorem}

The proof of this result can be found in Section~\ref{secproofDif}.
\begin{remark}
Notice that, if $p=-2$, the logarithmic term in the function $C$, (see~\eqref{formaCtheorem}) has the same size as the corresponding one
in definition~\eqref{defxitheorem} of $\xi$. However, when $p>-2$, the function $C$ is indeed a perturbation term of order
$\mathcal{O}(\delta^{p+2}|\log \delta|)$ over complex values of $\vu$.

In fact, when $p>-2$, we do not need the exact form~\eqref{formaCtheorem} of $C$, we only need
to know that $|C(u,\theta)| \leq K \delta^{p+2}$ when $\vu\in \mathbb{R}$ which is easier to check. However it is
mandatory in the generic case $p=-2$.
\end{remark}

\subsection{Sharp upper bound}\label{subsec-upperbound}
Even when this is not the final goal of this work, which deals with asymptotic expressions,
in Proposition~\ref{prop:sharpbound} of this section we provide an upper bound for $\Delta (\vu,\theta)$ when $\vu,\theta \in \mathbb{R}$.
On the one hand, we will gain some
intuition about the main problems we will have to overcome and, on the other hand, some of the results proven in
this section will be used in the proof of Theorem \ref{mainthm-intro}.

An straightforward consequence of Theorem~\ref{thmdifpartsolinjective} is that:
\begin{equation}\label{expression-delta}
\Delta(\vu,\theta)
=\cosh^{2/\coef}(\coef\vu)(1+P_1(\vu,\theta))\sum_{l\in\mathbb{Z}}\Upsilon^{[l]}e^{il\xi(\vu,\theta)},
\end{equation}
where $P_1$ and $\xi$ are given in Theorem~\ref{thmdifpartsolinjective} and $\Upsilon^{[l]}$, the Fourier
coefficients of the function $\tilde k(\tau)$, are unknown. They depend on $\delta$ and $\param$ although we do not write it explicitly.

Now we are going to study separately the average $\Upsilon^{[0]}$ in Theorem~\ref{thmktilde0} (in fact this result also deals with $\Upsilon_0^{[0]}$
of the Melnikov function) and the rest of the Fourier
coefficients $\Upsilon^{[l]}$, $l\neq0$ in Lemma~\ref{lemUpsilonlexpsmall}. The sharp upper bound for $\Delta(\vu,\theta)$ is an straightforward consequence
of these results.

\begin{theorem}\label{thmktilde0}
Let $p\geq-2$. Let $\Upsilon^{[0]}$ be the average of the function $\tilde k(\tau)$, given in Theorem~\ref{thmdifpartsolinjective}, and $\Upsilon_0^{[0]}$ be the constant defined in~\eqref{defcoefMelnikov}.
\begin{enumerate}
\item In the conservative case, for all $0\leq\delta\leq\delta_0$ one has:
$$\Upsilon^{[0]}=0,\qquad \Upsilon_0^{[0]}=0.$$
\item In the dissipative case, there exists a curve $\param=\param_*^0(\delta)=\mathcal{O}(\delta^{p+3})$ such that for all $0\leq\delta\leq\delta_0$ one has:
$$\Upsilon^{[0]}=\Upsilon^{[0]}(\delta,\param_*^0(\delta))=0.$$
In addition, given constants $a_1$, $a_2\in\mathbb{R}$ and $a_3>0$, there exists a curve $\param=\param_*(\delta)=\mathcal{O}(\delta^{p+3})$ such that for all $0\leq\delta\leq\delta_0$ one has:
$$\Upsilon^{[0]}=\Upsilon^{[0]}(\delta,\param_*(\delta))=a_1\delta^{a_2}e^{-\frac{a_3\pi}{2\coef\delta}}.$$
Along these curves one has:
$$\Upsilon_0^{[0]}=\Upsilon_0^{[0]}(\delta,\param_*(\delta))=\mathcal{O}(\delta^{p+4}).$$
\end{enumerate}
\end{theorem}
\begin{figure}[ht]
\center
 \includegraphics[width=5cm]{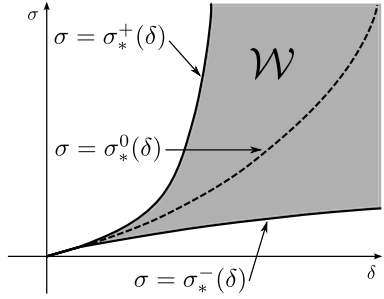}
 \caption[The curve $\param=\param_*^0(\delta)$ and a wedge-shaped domain $\mathcal{W}$ around it.]{The curve $\param=\param_*^0(\delta)$ and a wedge-shaped domain $\mathcal{W}$ around it. Inside this domain, the coefficient $\Upsilon^{[0]}$ is exponentially small.}
 \label{figurewedge}
\end{figure}

For the proof of this theorem we refer the reader to Section~\ref{sectionthmktilde0}.
We stress that this result is standard in the usual scenarios of
Hamiltonian or reversible vector fields and symplectic maps.
However in our setting, its proof involves delicate arguments.

\begin{remark}
In the dissipative case one can see that, given $a_1$, $a_2\in\mathbb{R}$ and $a_3>0$, the curves $\param_*(\delta),\param_*^0(\delta)$
in Theorem~\ref{thmktilde0} satisfy:
\begin{equation}\label{difcurvesparam}
\param_*(\delta)-\param_*^0(\delta)=a_1\delta^{a_2}e^{-\frac{a_3\pi}{2\coef\delta}}\left(1+\mathcal{O}(\delta)\right).
\end{equation}
Let us now fix some constants $a_1^+>0$ and $a_1^-<0$. Fix also $a_2^+,a_2^-\in\mathbb{R}$ and $a_3^+,a_3^->0$.
Define $\param_*^+(\delta)$ as the curve in of Theorem~\ref{thmktilde0} corresponding to the constants $a_1^+$,
$a_2^+$ and $a_3^+$, and $\param_*^-(\delta)$ as the curve in of Theorem~\ref{thmktilde0} corresponding to the constants
$a_1^-$, $a_2^-$ and $a_3^-$. By~\eqref{difcurvesparam} one has that $\param_*^-(\delta)\leq\param_*^+(\delta)$ for $\delta$
sufficiently small. Define the domain:
$$\mathcal{W}:=\{(\delta,\param)\in\mathbb{R}^2\,:\,\param_*^-(\delta)\leq \param \leq \param_*^+(\delta)\}$$
in the parameter plane. This domain is a wedge-shaped domain around $\param_*^0(\delta)$ (see Figure~\ref{figurewedge}).
Moreover there exists $\delta_0>0$ such that for all $0<\delta\leq\delta_0$ and $(\delta,\param)\in \mathcal{W}$, the
coefficient $\Upsilon^{[0]}(\delta,\param)$ is exponentially small. More precisely, let us denote $\bar a_3=\min\{a_3^+,\,a_3^-\}$.
Define $\bar a_1=a_1^+$ and $\bar a_2=a_2^+$ if the minimum is achieved in $a_3^+$, otherwise we take $\bar a_1=a_1^-$ and $\bar a_2=a_2^-$.
Then:
$$
|\Upsilon^{[0]}(\delta,\param)|\leq |\bar a_1|\delta^{\bar a_2}e^{-\frac{\bar{a}_3\pi}{2\coef\delta}},\qquad \qquad
\textrm{if}\quad0<\delta\leq\delta_0,\quad(\delta,\param)\in \mathcal{W}.
$$
\end{remark}

Now we are going to deal with the Fourier coefficients $\Upsilon^{[l]}$, with $l\neq 0$.
From expression~\eqref{expression-delta} of $\Delta$, one can already see that
$\Upsilon^{[l]}$ are exponentially small with respect to $\delta$.
Indeed, as a first exploration, we can consider the case $P_1= \chi\equiv0$. This case can give some insight since,
as one can see from the bounds of $P_1$ and $\chi$, given in \eqref{boundP1-prop} and~\eqref{boundLchi-prop} respectively,
they are ``small'' functions when we take large $\dist$.
If we make this simplification, using expression~\eqref{expression-delta} of $\Delta$:
$$
\Delta(u,\theta) = \cosh^{2/\coef}(\coef\vu)\sum_{l\in\mathbb{Z}}\Upsilon^{[l]}e^{il(\theta + \tilde{\xi}(\vu)}
$$
with $\tilde\xi(\vu)=\delta^{-1}\alpha\vu+\coef^{-1}(c+  \delta^{p+2} \alpha L_0)\log\cosh(\coef\vu)+\alpha L(\vu)$
(see formula~\eqref{defxitheorem} of $\xi$ and take $\chi\equiv0$). We have then that $\Upsilon^{[l]}e^{il\tilde\xi(\vu)}$
are the Fourier coefficients of the function
$\Delta(\vu,\theta)\cosh^{-2/\coef}(\coef\vu)$. In other words:
$$
 \left|\Upsilon^{[l]}\right|=\left|\frac{e^{-il\tilde{\xi}(\vu)}}{2\pi}\int_0^{2\pi}\frac{\Delta(\vu,\theta)e^{-il\theta}}{\cosh^{2/\coef}(\coef\vu)}d\theta\right|\\
\leq\left|e^{-il\tilde{\xi}(\vu)}\right|\sup_{\theta\in[0,2\pi]}\left|\frac{\Delta(\vu,\theta)}{\cosh^{2/\coef}(\coef\vu)}\right|.
$$
We note that this inequality is valid for all $\vu\in\Doutinter=\DoutT{\uns}\cap\DoutT{\sta}$.
In particular, taking $\vu=\vu_+:=i(\pi/(2\coef)-\dist\delta)$ for $l<0$ and $\vu=\vu_-:=-\vu_+$ for $l>0$ and using that the constant $L_0\in \mathbb{R}$
and that   $\im\log\cosh(\coef u_\pm)=\arg(\cosh(\coef u_\pm))=0$, one obtains:
$$
\left|\Upsilon^{[l]}\right|\leq K(\dist\delta)^{-2/\coef}e^{-\left(\frac{\alpha\pi}{2\coef\delta}-\alpha\dist-\alpha|\im L(\vu_\pm)|\right)|l|}\sup_{\theta\in[0,2\pi]}\left|\Delta(\vu_\pm,\theta)\right|.
$$
Recalling that $\Delta=r_1^\uns-r_1^\sta$ and using that $|r_1^{\uns,\sta}(\vu_\pm,\theta)|\leq M\delta^{p}\dist^{-3}$ by Theorem \ref{thmoutloc},
we obtain readily (renaming $K$):
$$
\left|\Upsilon^{[l]}\right|\leq K\frac{\delta^{p-2/\coef}}{\dist^{3+2/\coef}}e^{-\left(\frac{\alpha\pi}{2\coef\delta}-\alpha\dist-\alpha|\im L(\vu_\pm)|\right)|l|}.
$$
In particular, there exists a constant $K$, independent of $l$ such that:
$$
\left|\Upsilon^{[\pm1]}\right|\leq K\frac{\delta^{p-2/\coef}}{\dist^{3+2/\coef}}e^{-\frac{\alpha\pi}{2\coef\delta}+\alpha\dist},
$$
and for $|l|\geq2$:
$$
\left|\Upsilon^{[l]}\right|\leq K\frac{\delta^{p-2/\coef}}{\dist^{3+2/\coef}}e^{-\frac{\alpha\pi}{2\coef\delta}\frac{3|l|}{4}},
$$
where we have used that $\delta |\im L(\vu_\pm)|$ is arbitrarily small by bound~\eqref{boundLchi-prop} and condition~\eqref{conddist} on $\dist$.

The following result, whose proof is postponed to Section~\ref{sec:exponentially}, states that the same exponentially small bounds hold when
$P_1(\vu,\theta)\neq 0$ and $\chi(\vu,\theta)\neq 0$.
\begin{lemma}\label{lemUpsilonlexpsmall}
Let $\Upsilon^{[l]}$, $l\in\mathbb{Z}$, $l\neq0$, be the coefficients appearing in expression~\eqref{expression-delta} of $\Delta$.
Take $\dist$ as in Theorem~\ref{thmoutloc}. There exists a constant $M$ such that:
$$\left|\Upsilon^{[\pm1]}\right|\leq M\frac{\delta^{p-2/\coef}}{\dist^{3+2/\coef}}e^{-\frac{\alpha\pi}{2\coef\delta}+\alpha\dist},\qquad
\left|\Upsilon^{[l]}\right|\leq M\frac{\delta^{p-2/\coef}}{\dist^{3+2/\coef}}e^{-\frac{\alpha\pi}{2\coef\delta}\frac{3|l|}{4}}, \qquad |l|\geq 2 $$
\end{lemma}

As a consequence of this result we obtain the sharp upper bound:
\begin{proposition}\label{prop:sharpbound}
Let $\dist$ be as in Theorem~\ref{thmoutloc}.
In the dissipative case we take $\param=\param_*(\delta)$, where $\param^*$ is one of the curves defined in Theorem~\ref{thmktilde0}.
Let $\Upsilon^{[0]}=\Upsilon^{[0]}(\param_*(\delta),\delta)$ be the constant
provided by Theorem~\ref{thmktilde0}. In the conservative case $\Upsilon^{[0]}=0$. Then, for real values of $\vu$ and $\theta$:
$$
|\Delta(\vu,\theta)|\leq \cosh^{2/\coef}(\coef\vu) \left (|\Upsilon^{[0]}|+M\frac{\delta^{p-2/\coef}}{\dist^{3+2/\coef}}e^{-\frac{\alpha\pi}{2\coef\delta}+\alpha\dist}\right ).
$$
for some constant $M$.
\end{proposition}

\subsection{First order of the difference. End of the proof of Theorem~\ref{mainthm-intro}}\label{subsec:firstorderdiff-outer}
In this section we provide Theorem~\ref{mainthm} which gives a first order of the difference $\Delta(\vu,\theta)$ also studied
in Theorem~\ref{thmdifpartsolinjective}. Moreover, we will give the proof of Theorem~\ref{mainthm-intro} as a corollary of
Theorem~\ref{mainthm}.

Recall that $\Upsilon^{[0]}$ is the average of the function $\tilde k(\tau)$ of
Theorem~\ref{thmdifpartsolinjective}. Since we want to obtain (non-sharp) results also in the case $p=-2$ we define:
\begin{equation}\label{defktilde0}
\tilde{k}_0(\tau):=\Upsilon^{[0]}+\sum_{l\neq 0} \hat \Upsilon_0^{[l]}e^{il \tau}, \qquad \hat \Upsilon_0^{[l]}=\Upsilon_0^{[l]}
e^{-il \alpha \coef^{-1} L_0 \delta^{p+2} \log\delta},
\end{equation}
where $L_0\in \mathbb{R}$ is given in Theorem~\ref{thmdifpartsolinjective} and $\Upsilon_0^{[l]}$ are the constants appearing in the Fourier coefficients of
the Melnikov function, defined in~\eqref{defcoefMelnikov}.
Our candidate to be the first order of the difference is:
\begin{equation}\label{defDelta0}
 \Delta_0(\vu,\theta)  = \cosh^{2/\coef}(\coef\vu)(1+P_1(\vu,\theta))\tilde{k}_0(\xi(\vu,\theta)),
\end{equation}
with $\xi$ defined in~\eqref{defxitheorem}.
We note that we have not chosen the average of $\tilde k_0$ to be the coefficient $\Upsilon_0^{[0]}$ appearing in the average of the
Melnikov function (as one might expect) but $\Upsilon^{[0]}$, the average of $\tilde k(\tau)$ in Theorem~\ref{thmdifpartsolinjective}.

Next result shows that $\Upsilon^{[\pm 1]}$, Fourier coefficients of $\tilde k$,
are well approximated by $\hat \Upsilon^{[\pm 1]}$, Fourier coefficients of $\tilde k_0$ defined in~\eqref{defktilde0}.
The proof of this Proposition is done in Section~\ref{secproofproperrorUpsilons}.
\begin{proposition}\label{properrorUpsilons}
Let $p\geq -2$ and $\dist$ a sufficiently large constant.
Let $\Upsilon^{[\pm 1]}$ be the Fourier coefficients of order $\pm 1$ of $\tilde k(\tau)$, in Theorem~\ref{thmdifpartsolinjective},
and $\hat\Upsilon^{[\pm 1]}_0$ the ones given in~\eqref{defktilde0}.
Then there exists a constant $M$ such that:
$$
\left|\Upsilon^{[\pm 1]}-\hat\Upsilon_0^{[\pm 1]}\right|\leq M\left(\frac{|\log \dist|}{\dist^{4+\frac{2}{\coef}}}\delta^{2
\left(p+1-\frac{1}{\coef}\right)}|+\frac{\delta^{p+3-\frac{2}{\coef}}}{\dist^{1+\frac{2}{\coef}}}\right)
e^{-\frac{\alpha}{\delta}\left(\frac{\pi}{2\coef}-\dist\delta\right)},
$$
where we assume that, in the dissipative case, $\param=\param_*(\delta)$ is one of the curves defined in Theorem~\ref{thmktilde0}.
Recall that $\coef=1$ in the conservative case
\end{proposition}
Now we can state the theorem which gives the asymptotic for the difference of the invariant manifolds $\Delta=r^{\uns}_1-r^{\sta}_1$.
\begin{theorem}\label{mainthm}
Let $p\geq -2$. Consider the functions $m(w,\theta), \vt(\vu,\delta)$ and the constants $\mathcal{C}_1,\mathcal{C}_2$
defined in Theorem~\ref{thmasyformulaMelnikov}. In the dissipative case we take $\param=\param_*(\delta)$, where $\param^*$
is one of the curves defined in Theorem~\ref{thmktilde0}. Let $\Upsilon^{[0]}=\Upsilon^{[0]}(\param_*(\delta),\delta)$ be the constant
provided by this Theorem. In the conservative case recall that $\Upsilon^{[0]}=0$.

There exists $T_0>0$ such that for all $\vu\in[-T_0,T_0]$ and $\theta\in\mathbb{S}^1$
\begin{eqnarray*}
 \Delta(\vu,\theta)&=&\cosh^{\frac{2}{\coef}}(\coef\vu)\Upsilon^{[0]}\Big(1+\mathcal{O}\left(\delta^{p+3}\right)\Big)\\
&&+\delta^{p-\frac{2}{\coef}}\cosh^{\frac{2}{\coef}}(\coef\vu)e^{-\frac{\alpha\pi}{2\coef\delta}}
\Bigg[\mathcal{C}_1\cos\Big(\theta+\vt(\vu,\delta)-\alpha \coef^{-1} L_0 \delta^{p+2} \log \delta \Big)\\
&&+\mathcal{C}_2\sin\Big(\theta+\vt(\vu,\delta) -\alpha \coef^{-1} L_0 \delta^{p+2} \log \delta \Big)+\mathcal{O}\left(\delta^{p+2}+\delta^3\right)\Bigg],
\end{eqnarray*}
where we recall that $\coef=1$ in the conservative case.
\end{theorem}
\begin{remark}
Even when this result is valid for $p\geq -2$ it only provides an asymptotic formula for $\Delta$ in the case
$p>-2$.
However, when $p=-2$, it gives an upper bound which coincides with the one given in Proposition \ref{prop:sharpbound}.
\end{remark}
\begin{proof}
Recalling the definition~\eqref{defDelta0} of $\Delta_0$ and the form~\eqref{defDelta} of $\Delta$ given in Theorem~\ref{thmdifpartsolinjective}, we can write:
$\Delta(\vu,\theta)=\Delta_0(\vu,\theta)+\Delta_1(\vu,\theta)$,
where:
\begin{equation}\label{Delta01Fourier}
\Delta_1(\vu,\theta)=\cosh^{\frac{2}{\coef}}(\coef\vu)(1+P_1(\vu,\theta))\sum_{l\neq0}\left(\Upsilon^{[l]}-\hat \Upsilon_0^{[l]}\right)e^{il\xi(\vu,\theta)}.
\end{equation}
First of all we note that since we are taking $\vu\in[-T_0,T_0]$ and $\theta\in\mathbb{S}^1$ we have that all the functions
are real and bounded.
Then, using Proposition~\ref{properrorUpsilons} to bound $\left(\Upsilon^{[\pm 1]}-\hat \Upsilon_0^{[\pm 1]}\right)$, Lemma~\ref{lemUpsilonlexpsmall} to bound
$|\Upsilon^{[l]}|$ and Theorem~\ref{thmasyformulaMelnikov} to bound $|\hat \Upsilon_0^{[l]}|$ for $|l|\ge 2$, we obtain
$$
\left|\Delta_1(\vu,\theta)\right|\leq K\left(\delta^{2\left(p+1-\frac{1}{\coef}\right)}+
\delta^{p+3-\frac{2}{\coef}}\right)e^{-\frac{\alpha\pi}{2\coef\delta}}.
$$
Again, we omit the explicit dependence on $\dist$. Since $\Delta_1$ has the size of the remainder in
the asymptotic expansion for $\Delta$, we only need to deal with $\Delta_0$.

Recall that by~\eqref{defktilde0}, $\hat \Upsilon_0^{[l]}=\Upsilon_0^{[l]} e^{-il \alpha \coef^{-1} L_0 \delta^{p+2} \log \delta}$ so that both have the same modulus.
Then,
by the bounds obtained in Theorem~\ref{thmasyformulaMelnikov} for the coefficients $\Upsilon_0^{[l]}$, $l\neq 0$ and using the expression~\eqref{defDelta0}
of $\Delta_0$ one has that
\begin{align*}
\Delta_0(\vu,\theta) = &\cosh^{\frac{2}{\coef}} (\coef u) (1+ P_1(\vu,\theta))\left [\Upsilon^{[0]} +
2 \re \left (\hat \Upsilon^{[1]}_0e^{i\xi(\vu,\theta)} \right )\right . \\
&+ \left .\mathcal{O}\left (\delta^{p-\frac{2}{\coef}}e^{-\frac{3\alpha \pi}{4\coef \delta}}\right)\right ].
\end{align*}
Again, using formula for $\Upsilon^{[1]}_0$ in Theorem~\ref{thmasyformulaMelnikov} as well as Theorem~\ref{thmdifpartsolinjective} for $\xi$ one has that
$$
\hat \Upsilon^{[1]}_0e^{i\xi(\vu,\theta)} = \delta^{p-\frac{2}{\coef}}e^{-\frac{\alpha \pi}{2 \coef \delta}}
\left (\frac{\mathcal{C}}{2} e^{i(\theta +\alpha\delta^{-1} u +\frac{c}{\coef}\log (\cosh \coef u) - \frac{1}{\coef}[c +\alpha L_0 \delta^{p+2}] \log \delta) }+
\mathcal{O}(\delta^{p+2})\right )
$$
and the result follows since by Theorem~\ref{thmdifpartsolinjective}, $|P_1(\vu,\theta)| \leq K \delta^{p+3}$ for $\vu\in \mathbb{R}$.
\end{proof}

Theorem~\ref{mainthm} easily yields Theorem~\ref{mainthm-intro}:
\begin{proof}[End of the proof of Theorem~\ref{mainthm-intro}]
We point out that $\Delta(\vu,\theta)$ is not the actual distance between the invariant manifolds, since we computed the difference in ``symplectic'' cylindric coordinates. The actual distance is given by:
\begin{eqnarray*}
D(\vu,\theta)&=&\sqrt{2(\hetr(\vu)+r_1^\uns(\vu,\theta))}-\sqrt{2(\hetr(\vu)+r_1^\sta(\vu,\theta))}
\\ &=&\frac{1}{\sqrt{2\hetr(\vu)}}\Delta(\vu,\theta)+\mathcal{O}_2(\Delta(\vu,\theta)).
\end{eqnarray*}
Using the definition~\eqref{hetr} of $\hetr(\vu)$ one obtains:
$$
D(\vu,\theta)=\sqrt{\frac{b}{\coef+1}}\cosh(\coef\vu)\Delta(\vu,\theta)+\mathcal{O}_2(\Delta(\vu,\theta)).
$$
To obtain the formulas given in Theorem~\ref{mainthm-intro} we undo the change of variables in
Section~\ref{subsecpreliminary}, we take into account the notation
$\delta=\sqrt{\mu}$, $\param=\delta^{-1}\nu=\nu/\sqrt{\mu}$, $b=\gamma_2$, $\coef=\beta_1$ and $c=\alpha_3$
(so that the conservative case is proven). Moreover, redefining the coefficients $a_1$ and $a_2$ the formula in Remark~\ref{rmkmthm} is checked
for the dissipative case. Taking $a_1=0$ we get the asymptotic formula in Theorem~\ref{mainthm-intro}.
\end{proof}

The remaining part of this work includes the proofs of the above results. However these proofs are not exposed in the
order provided in this section. We have preferred to postpone the most technical but standard demonstrations to the end of this work
and give priority to the ones involving the exponentially small behavior of the difference $\Delta(\vu,\theta)$ and of the
Melnikov function $M(\vu,\theta)$ when $\vu,\theta \in \mathbb{R}$, namely, the results in Sections~\ref{subsecintromelnikov} and~\ref{subsec:firstorderdiff-outer}
above. As any expert in exponentially small phenomena knows, the results for real values of $\vu,\theta$ are consequence of the results for complex values.
Therefore, we will perform the proofs of the above mentioned results, assuming that the result about
the existence of complex parameterizations (Theorem~\ref{thmoutloc}) and the general form of the difference $\Delta(\vu,\theta)$ for complex
values of $\vu$ (Theorem~\ref{thmdifpartsolinjective}) hold true. We will do this in Section~\ref{sec:exponentially}.
Then, we will proof Theorems~\ref{thmoutloc} and~\ref{thmdifpartsolinjective} in Sections~\ref{sec:param} and~\ref{secproofDif}
respectively.

All the constants that appear in the statements of the following results might depend on $\delta^*$, $\param^*$ and $\dist^*$,
but never on $\delta$, $\param$ and $\dist$. We assume that $\delta^*$ and $\param^*$ are sufficiently small, and $\dist^*$ is sufficiently large
satisfying condition~\eqref{conddist}. Finally, to make formulas shorter and avoid keeping track of constants that do not play any role in the proofs, we will use $K$ to denote \textit{any} constant independent of the parameters $\delta$, $\param$ and $\dist$. These conventions are valid for all the sections of this work.
We shall not write the proofs that are either for standard results or too technical and that do not provide any interesting insight. For these proofs we refer the reader to~\cite{CastejonPhDThesis}.

\section{The exponentially small behavior}\label{sec:exponentially}
We first begin with the results related to the exponentially small behavior for real values of $\vu$.
That is, Theorems~\ref{thmasyformulaMelnikov} and~\ref{thmktilde0} and Proposition~\ref{properrorUpsilons}.

\subsection{The Melnikov function. Proof of Theorem~\ref{thmasyformulaMelnikov}}
\label{sectionMelnikov}
Since for real values of $(\vu,\theta)$ the Melnikov function $M(\vu,\theta)\in\mathbb{R}$
(see~\eqref{defMelnikovequiv} for its definition), one has that $\Upsilon_0^{[-l]}=\overline{\Upsilon_0^{[l]}}$,
where the coefficients $\Upsilon_0^{[l]}$ were defined in~\eqref{defcoefMelnikov}.
Hence, we just compute $\Upsilon_0^{[l]}$ with $l>0$.

For $C\in\mathbb{R}$ and $l,n,Q\in\mathbb{N}$, we define the following integrals:
\begin{equation}\label{defintegralInqlC}
I_{n,Q}^{l,C}=\int_{-\infty}^{+\infty}\frac{e^{-\delta^{-1}\alpha i|l| s}\sinh^n(\coef s)}{\cosh^{Q+1+iC|l|}(\coef s)}ds,\qquad\qquad Q+1>n.
\end{equation}
Let us denote by $f_{qkmn}$, $g_{qkmn}$ and $h_{qkmn}$ the Taylor coefficients of $f$, $g$ and $h$ respectively, namely:
\begin{equation}\label{taylorf}
f(\delta x,\delta y,\delta z,\delta,\delta\param)=\sum_{q=3}^\infty\delta^q\sum_{k+m+n\leq q}f_{qkmn}(\param)x^ky^mz^n,
\end{equation}
and analogously for $g$ and $h$. In the following we shall write $f_{qkmn}$ instead of $f_{qkmn}(\param)$, but of course these coefficients still depend on $\param$. Note that one has $f_{qkmn}=f_{qkmn}(0)+\mathcal{O}(\param)=f_{qkmn}(0)+\mathcal{O}(\delta^{p+3})$, since we just consider the case $|\param|\leq\param^*\delta^{p+3}$.

Denote by $a_{k,m}^{[l]}$ the $l$-th Fourier coefficient of the function $\cos^k\theta\sin^m\theta$.
Recalling the definition~\eqref{defopF} of $\Fout$, the notation~\eqref{notationFGHbis} and the definition~\eqref{defFGH} of $\Fb$ and $\Hb$, it can be seen that, for $l>0$, $\Upsilon_0^{[l]}$ introduced in~\eqref{defcoefMelnikov}, writes out as:
\begin{equation}\label{desenvolMelni}
\Upsilon_0^{[l]} = \Upsilon_{0,f}^{[l]}+\Upsilon_{0,g}^{[l]}+\Upsilon_{0,h}^{[l]}
\end{equation}
with,
\begin{equation}\label{desenvolMelnideffgh}
\begin{aligned}
 \Upsilon_{0,f}^{[l]}&=\delta^p\sum_{q=3}^\infty\sum_{k+m+n\leq q}\delta^qf_{qkmn}
\left(\sqrt{\frac{\coef+1}{b}}\right)^{k+m+1}a_{k+1,m}^{[l]}I_{n,k+m+n+2\coef^{-1}}^{l,c\coef^{-1}},\\
\Upsilon_{0,g}^{[l]}&=\delta^p\sum_{q=3}^\infty\sum_{k+m+n\leq q}\delta^qg_{qkmn}
\left(\sqrt{\frac{\coef+1}{b}}\right)^{k+m+1}a_{k,m+1}^{[l]}I_{n,k+m+n+2\coef^{-1}}^{l,c\coef^{-1}},\\
\Upsilon_{0,h}^{[l]}&=\delta^p\sum_{q=3}^\infty\sum_{k+m+n\leq q}\delta^qh_{qkmn}
\left(\sqrt{\frac{\coef+1}{b}}\right)^{k+m+2}a_{k,m}^{[l]}I_{n+1,k+m+n+2\coef^{-1}}^{l,c\coef^{-1}},
\end{aligned}
\end{equation}
being $I_{n,Q}^{l,C}$ the integrals defined in~\eqref{defintegralInqlC}. We are interested in bounding these integrals for $|l|\geq 2$:
\begin{lemma}\label{lemboundInql2}
Let $C$ be fixed. For any $M>0$ there exists $\delta_0,K>0$ satisfying that for all $0<\delta<\delta_0$,
$|l|\geq2$, $Q\geq1$ and $n$ such that $Q+1>n$:
 $$
\left|I_{n,Q}^{l,C}\right|\leq K M^{Q}\delta^{-Q}e^{-\frac{\alpha\pi}{2\coef\delta}\frac{3|l|}{4}}.
$$
\end{lemma}
\begin{proof} Take $\rho>0$.
Using Cauchy's theorem, the integration path of the integrals $I_{n,Q}^{l,C}$ can be changed to:
$s=s(t):=-\frac{i}{\coef}\left(\frac{\pi}{2}- \rho\delta\right)+t$ , $t\in(-\infty,\infty)$.
Then one obtains:
\begin{equation}\label{integralInQlC-changepath}
I_{n,Q}^{l,C}=e^{-\frac{\alpha}{\coef}|l|\left(\frac{\pi}{2\delta}-\rho\right)}\int_{-\infty}^{+\infty}
\frac{e^{-\delta^{-1}\alpha i|l|t}\sinh^n(\coef s(t))}{\cosh^{Q+1+iC|l|}(\coef s(t))}dt.
\end{equation}
Since for $z\in\mathbb{C}$,
$|z^{Q+1+iC|l|}|\geq|z|^{Q+1}e^{-|Cl\arg z|}$ and $|\arg\cosh(s(t))|\leq\pi/2$, then
$$|\cosh^{Q+1+iC|l|}(\coef s(t))|\geq |\cosh^{Q+1}(\coef s(t))|e^{-|Cl|\frac{\pi}{2}}.$$
Using this bound in expression~\eqref{integralInQlC-changepath} of $I_{n,Q}^{l,C}$ and standard arguments, one can prove that there exists
$K>0$ (which is also independent of $\rho$) such that
$$
\left|I_{n,Q}^{l,C}\right|\leq  K^{Q+1}(\rho \delta)^{-Q}e^{-\frac{\alpha}{\coef}|l|\left(\frac{\pi}{2\delta}-\rho-|C|\frac{\pi}{2}\right)}.
$$
The proof is finished taking $\rho$ sufficiently large and $\delta$ sufficiently small.
\end{proof}

Our goal now will be to find an asymptotic formula for the integrals $I_{n,Q}^{l,C}$ with $l=1$, which will dominate over the integrals with $|l|\geq2$. First of all, we give a recurrence that is valid for all $l\neq0$. The proof follows integrating by parts.
\begin{lemma}\label{lemrecurrence}
Let $C$ be fixed. Then, for all $l\neq0$, $n\geq1$ and $Q>0$ such that $Q+1>n$, the following recurrence holds:
$$I_{n,Q}^{l,C}=\frac{-|l|\alpha i}{\coef\delta(Q+iC|l|)}I_{n-1,Q-1}^{l,C}+\frac{n-1}{Q+iC|l|}I_{n-2,Q-2}^{l,C}.$$
\end{lemma}

Now we summarize some properties of the Gamma function that will be needed later on.
\begin{lemma}\label{lempropsGamma}
Let $z,A\in\mathbb{C}$. Then:
\begin{enumerate}
 \item  $\Gamma(z)\Gamma(\overline{z})=\left|\Gamma(z)\right|^2.$
 \item \label{stirling} (\textit{Stirling Formula}) If $|\arg z|<\pi$, then:
 $$\Gamma(z)=e^{-z}e^{\left(z-\frac{1}{2}\right)\log z}(2\pi)^{\frac{1}{2}}(1+\mathcal{O}(z^{-1})).$$
 \item \label{GammaImaginary} If $z=iy$, $y\in\mathbb{R}$, then:
$$|\Gamma(iy)|=\frac{\sqrt{\pi}}{|y\sinh(\pi y)|^{1/2}}.$$
 \item If $|\arg z|<\pi$ and $|A|\leq A^*$ for some constant $A^*$, then:
 $$\Gamma(z+A)=\Gamma(z)z^{A}(1+\mathcal{O}(z^{-1})).$$
 \item \label{GammazAreal} There exists a constant $M\geq3/2$ and a function $J(z,A)$ such that for all $z\in\mathbb{C}$ with $|z|\geq 3$, $|\arg z|<\pi$, and all $A\in\mathbb{R}$ with $A\ge1$, one has:
 $$\Gamma(z+A)=\Gamma(z)z^{A}(1+z^{-1}J(z,A)),$$
 and $|J(z,A)|\leq M\Gamma(A)$.
\end{enumerate}
\end{lemma}
\begin{proof}
Every item above, except item~\ref{GammazAreal}, are standard facts, see for instance in~\cite{AS}. To prove item~\ref{GammazAreal},
we use previous property for $A^*\geq 3$ and after that we proceed by induction. See the details in~\cite{CastejonPhDThesis}.
\end{proof}

Finally, we can give an asymptotic formula of $I_{n,Q}^{1,C}$.
\begin{lemma}\label{lemasyInq}
Let $C$ be fixed.
Then for all $Q\geq1$ and $n\geq0$ such that $Q+1>n$ one has:
$$
I_{n,Q}^{1,C}=\frac{2\pi}{\coef}\left(\frac{\alpha}{\coef\delta}\right)^{Q+iC}\frac{(-i)^n}{\Gamma(Q+1+iC)}e^{-\frac{\alpha\pi}{2\coef\delta}}+
\mathcal{O}\left(\left(\frac{\alpha}{\coef\delta}\right)^{Q-1}e^{-\frac{\alpha\pi}{2\coef\delta}}\right),
$$
where the $\mathcal{O}$ means uniformly on $n,Q$ and $C$.
\end{lemma}
\begin{proof}
We first deal with the case $n=0$ and after that we will proceed by induction.
Performing the change of variables $w=\tanh(\coef s)$, one has that:
$$
I_{0,Q}^{1,C}=\frac{1}{\coef}\int_{-1}^1(1+w)^{\frac{\coef(Q-1+iC)-i\delta^{-1}\alpha}{2\coef}}
(1-w)^{\frac{\coef(Q-1+iC)+i\delta^{-1}\alpha}{2\coef}}dw.$$
Naming:
\begin{equation*}
 a=\frac{\coef(Q+1+iC)+i\delta^{-1}\alpha}{2\coef},\qquad
 b=Q+1+iC,
\end{equation*}
we can rewrite the last equation as, see for instance~\cite{AS}:
$$
I_{0,Q}^{1,C}=\frac{1}{\coef}\int_{-1}^1(1+w)^{b-a-1}(1-w)^{a-1}dw=2^{b-1}\coef^{-1}\frac{\Gamma(b-a)\Gamma(a)}{\Gamma(b)},
$$
so that we can write:
\begin{equation}\label{I0q1C}
 I_{0,Q}^{1,C}=2^{Q+iC}\coef^{-1}\frac{\Gamma_Q^C}{\Gamma(Q+1+iC)}, \qquad \Gamma_Q^C:=\Gamma(b-a)\Gamma(a).
\end{equation}
We now shall find an asymptotic expression for $\Gamma_Q^C$. Let:
$$A=\frac{Q+1}{2}\geq1,\qquad z_\pm=i\frac{\coef C\pm\delta^{-1}\alpha}{2\coef},$$
so that $b-a=A+z_-$ and $a=A+z_+$. We note that $|\arg z_\pm|=\pi/2<\pi$ and that for sufficiently small $\delta$ one has $|z_\pm|\geq 3$. Then, by item~\ref{GammazAreal} of Lemma~\ref{lempropsGamma} we have $\Gamma_Q^C=\Gamma(A+z_-)\Gamma(A+z_+)$ and consequently:
$$
 \Gamma_Q^C=z_+^Az_-^A\Gamma(z_-)\Gamma(z_+)\left(1+\frac{1}{z_+}J(z_+,A)\right)\left(1+\frac{1}{z_-}J(z_-,A)\right),
$$
with$|J(z_\pm,A)|\leq M\Gamma(A)$. Now we are going to give the asymptotic behavior of the above expression.
We have that:

\begin{align}\label{temrzpm}
 &z_+^Az_-^A =\left(\frac{\alpha}{2\coef\delta}\right)^{Q+1}\left(1-\frac{\coef^2C^2\delta^2}{\alpha^2}\right)^{\frac{Q+1}{2}}, \notag\\
 &\Gamma(z_-)\Gamma(z_+)=2\pi \left(\frac{\alpha}{2\coef\delta}\right)^{iC-1}e^{-\frac{\pi\alpha}{2\coef\delta}}(1+\mathcal{O}(\delta)), \\
&\left(1+\frac{1}{z_+}J(z_+,A)\right)\left(1+\frac{1}{z_-}J(z_-,A)\right)=1+|\Gamma(Q+1+iC)|e^{\frac{\pi |C|}{2}}\mathcal{O}\left(\delta\right). \notag
\end{align}
The first equality is straightforward from definition. The second one has to be proven by using
items~\ref{GammaImaginary} and~\ref{stirling} of Lemma~\ref{lempropsGamma}. The third one is the most involved.
Taking into account that $|J(z_\pm,A)|\leq M|\Gamma(A)|$, $A=(Q+1)/2$ and that $Q\geq 1$, one
checks that
\begin{equation}\label{boundtermsJ-3}
\left| \left(1+\frac{1}{z_+}J(z_+,A)\right)\left(1+\frac{1}{z_-}J(z_-,A)\right)-1\right|\leq K\delta\Gamma(Q+1).
\end{equation}
On the one hand, for $C=0$ it is clear that~\eqref{boundtermsJ-3} yields~\eqref{temrzpm}. On the other hand, for $C\neq0$, we have that
$|\Gamma(Q+1+iC)|\geq \Gamma(Q+1) |C\Gamma(iC)|$. Thus, using item~\ref{GammaImaginary} of Lemma~\ref{lempropsGamma} we obtain:
\begin{equation}\label{changeGammaiC}
\Gamma(Q+1)\leq\frac{|\Gamma(Q+1+iC)||\sinh(\pi C)|^{1/2}}{(\pi |C|)^{1/2}}\leq K\Gamma(Q+1+iC)e^{\frac{\pi |C|}{2}}.
\end{equation}
Equations~\eqref{boundtermsJ-3} and~\eqref{changeGammaiC} yield the last equality in~\eqref{temrzpm}.

Substituting the equalities in~\eqref{temrzpm} in expression~\eqref{I0q1C} of $I_{0,Q}^{1,C}$ and using that $|\Gamma(Q+1+iC)|\geq K>0$
we obtain the result for $n=0$.

For $n\geq1$ and $Q+1>n$ we proceed by induction, using the recurrence of Lemma~\ref{lemrecurrence}.
The important fact is that, in the recurrence for $I_{n,Q}^{1,C}$, only the term involving $I_{n-1,Q-1}^{1,C}$ contributes to $I_{n,Q}^{1,C}$
being the other one smaller.
\end{proof}

\begin{proof}[End of the proof of Theorem~\ref{thmasyformulaMelnikov}]
First we focus on $\Upsilon_0^{[1]}$. We shall study $\Upsilon_{0,f}^{[1]}$ appearing in formula~\eqref{desenvolMelni} of $\Upsilon_0^{[l]}$ taking $l=1$, the other two are done analogously. We decompose $\Upsilon_{0,f}^{[1]}$ into
\begin{equation}\label{defUpsilonf}
\Upsilon_{0,f}^{[1]} = \Upsilon_{0,0}^{[1]}+ \Upsilon_{0,1}^{[1]}
\end{equation}
where, following formula~\eqref{desenvolMelnideffgh} of $\Upsilon_{0,f}^{[1]}$,
\begin{align}\label{splitfirstsum}
\Upsilon_{0,0}^{[1]}&=
\delta^p\sum_{q=3}^\infty\sum_{k+m+n=q}\delta^qf_{qkmn}\left(\sqrt{\frac{\coef+1}{b}}\right)^{k+m+1}a_{k+1,m}^{[1]}I_{n,q+2\coef^{-1}}^{1,c\coef^{-1}} \notag\\
\Upsilon_{0,1}^{[1]}&=
\delta^p\sum_{q=3}^\infty\sum_{k+m+n< q}\delta^qf_{qkmn}\left(\sqrt{\frac{\coef+1}{b}}\right)^{k+m+1}a_{k+1,m}^{[1]}I_{n,k+m+n+2\coef^{-1}}^{1,c\coef^{-1}}
\end{align}

On the one hand, using Lemma~\ref{lemasyInq} with $C=c\coef^{-1}$ and $Q=q+2\coef^{-1}$:
\begin{align}\label{asyfirstterm}
\Upsilon^{[1]}_{0,0} =&\frac{2\pi}{\coef}\delta^{p-\frac{2}{\coef}-i\frac{c}{\coef}}e^{-\frac{\alpha\pi}{2\coef\delta}}
\sum_{q=3}^\infty\sum_{k+m+n=q}\frac{f_{qkmn}\left(\sqrt{\frac{\coef+1}{b}}\right)^{k+m+1}(-i)^na_{k+1,m}^{[1]}
\alpha^{q+\frac{2}{\coef}+i\frac{c}{\coef}}}{\coef^{q+\frac{2}{\coef}+i\frac{c}{\coef}}\Gamma\left(q+1+\frac{2}{\coef}+i\frac{c}{\coef}\right)}\nonumber\\
&+\delta^{p-\frac{2}{\coef}}e^{-\frac{\alpha\pi}{2\coef\delta}}\sum_{q=3}^\infty\sum_{k+m+n=q}\frac{f_{qkmn}\left(\sqrt{\frac{\coef+1}{b}}\right)^{k+m+1}a_{k+1,m}^{[1]}\alpha^{q-1+\frac{2}{\coef}+i\frac{c}{\coef}}}{\coef^{q-1+\frac{2}{\coef}+i\frac{c}{\coef}}}\mathcal{O}(\delta)\nonumber\\
=&\frac{2\pi}{\coef}\delta^{p-\frac{2}{\coef}-i\frac{c}{\coef}}e^{-\frac{\alpha\pi}{2\coef\delta}}\sum_{q=3}^\infty\sum_{k+m+n=q}\frac{f_{qkmn}\left(\sqrt{\frac{\coef+1}{b}}\right)^{k+m+1}(-i)^na_{k+1,m}^{[1]}\alpha^{q+\frac{2}{\coef}+i\frac{c}{\coef}}}{\coef^{q+\frac{2}{\coef}+i\frac{c}{\coef}}\Gamma\left(q+1+\frac{2}{\coef}+i\frac{c}{\coef}\right)}\nonumber\\
&+\mathcal{O}\left(\delta^{p+1-\frac{2}{\coef}}e^{-\frac{\alpha\pi}{2\coef\delta}}\right),
\end{align}
where we have used that $|a_{k,m}^{[1]}|\leq1$ for all $k$ and $m$ (because $a_{k,m}^{[1]}$ are Fourier coefficients of the functions
$\cos^k\theta\sin^m\theta$), and we have assumed that the radius of convergence of $f$ is sufficiently large and thus second sum converges.
To bound $\Upsilon_{0,1}^{[1]}$, using again Lemma~\ref{lemasyInq} with $C=c\coef^{-1}$ and $Q=k+m+n+2\coef^{-1}$ it is easy to see that:
$$
\delta^q\left|I_{n,k+m+n+2\coef^{-1}}^{1,c\coef^{-1}}\right|\leq \delta^{q-(k+m+n+\frac{2}{\coef})}\left(\frac{\alpha}{\coef}\right)^{k+m+n+\frac{2}{\coef}}e^{-\frac{\alpha\pi}{2\coef\delta}}.
$$
Then, $\Upsilon_{0,1}^{[1]}$ in~\eqref{splitfirstsum} can be bounded by:
\begin{equation}\label{asysecondterm}
\big |\Upsilon_{0,1}^{[1]}\big |\leq K\delta^{p+1-\frac{2}{\coef}}e^{-\frac{\alpha\pi}{2\coef\delta}},
\end{equation}
where again, we have assumed that the radius of convergence of $f$ is sufficiently large.
\begin{remark}
We need to assure that a point of the form $\alpha (x_0,y_0,z_0,0,0)$ is in $B(r_0)$, the ball of analyticity of $f$. For that, we
note that, rescaling $\delta=\epsilon \bar\delta$, we can consider $\alpha=\epsilon \bar \alpha $ as small as we want.
\end{remark}
Using ~\eqref{asyfirstterm} and~\eqref{asysecondterm} in~\eqref{defUpsilonf} we obtain
an asymptotic expression for $\Upsilon_{0,f}^{[1]}$
and reasoning analogously for the other sums appearing in formula~\eqref{desenvolMelni} of $\Upsilon_0^{[1]}$ we obtain:
\begin{align*}
 \Upsilon_0^{[1]}=&\frac{2\pi}{\coef}\delta^{p-\frac{2}{\coef}-i\frac{c}{\coef}}
e^{-\frac{\alpha\pi}{2\coef\delta}}\left [
\sum_{q=3}^\infty\sum_{k+m+n=q} \frac{\left(\sqrt{\frac{\coef+1}{b}}\right)^{k+m+1}(-i)^n\alpha^{q+\frac{2}{\coef}+i\frac{c}{\coef}}}{\coef^{q+\frac{2}{\coef}+i\frac{c}{\coef}}\Gamma\left(q+1+\frac{2}{\coef}+i\frac{c}{\coef}\right)}\right.
\\
&\left. \left (f_{qkmn}a_{k+1,m}^{[1]}+g_{qkmn}a_{k,m+1}^{[1]}-i \sqrt{\frac{\coef+1}{b}} h_{qkmn}a_{k,m}^{[1]}\right ) \right ]\\
& + \mathcal{O}\left(\delta^{p+1-\frac{2}{\coef}}e^{-\frac{\alpha\pi}{2\coef\delta}}\right )
\end{align*}

Substituting $f,g$ and $h$ by their Taylor's expansion~\eqref{taylorf} and using that, taking the two last variables in $f$ equal to zero implies that the second sum
in~\eqref{taylorf} is done only over the terms $k+m+n=q$ we have that the function $m$ defined in~\eqref{defm} in Theorem~\ref{thmasyformulaMelnikov}
can be written as
\begin{align*}
m(w,\theta)=& \sum_{q=3}^\infty\sum_{k+m+n=q}
\left(\sqrt{\frac{\coef+1}{b}}\right)^{k+m+1}(-i)^n  w^{q+1+\frac{2}{\coef}+i\frac{c}{\coef}} \\
&\cos^{k}\theta\sin^m\theta\left (f_{qkmn}\cos \theta + g_{qkmn} \sin \theta -i \sqrt{\frac{\coef+1}{b}}h_{qkmn}\right ).
\end{align*}
Therefore, using definition~\eqref{defBorel} of the Borel transform, a direct computation shows the asymptotic expression of $\Upsilon_0^{[1]}$ given
in Theorem~\ref{thmasyformulaMelnikov} and we are done in this case.

To bound $\Upsilon_0^{[l]}$ for $|l|\geq 2$, we use formula~\eqref{desenvolMelni} and the bound in Lemma~\ref{lemboundInql2}
with $C=c\coef^{-1}$ and $Q=k+m+n+2\coef^{-1}$ to obtain:
\begin{align*}
\left|\Upsilon_0^{[l]}\right| &\leq K\delta^{p-\frac{2}{\coef}}e^{-\frac{\alpha\pi}{2\coef\delta}\frac{3|l|}{4}}\left[\sum_{q=3}^\infty\sum_{k+m+n\leq q}
M^{k+m+n}\delta^{q-(k+m+n)} \left(\sqrt{\frac{\coef+1}{b}}\right)^{k+m+1} \right.\notag\\
 &\left. \left (\big|f_{qkmn} a_{k+1,m}^{[l]} \big |+ \big | g_{qkmn} a_{k,m+1}^{[l]}\big | + \sqrt{\frac{\coef+1}{b}}\big |h_{qkmn} a_{k,m}^{[l]}\big |\right )\right ]
\notag\\
& \leq K\delta^{p-\frac{2}{\coef}}e^{-\frac{\alpha\pi}{2\coef\delta}\frac{3|l|}{4}},
\end{align*}
where in the last inequality we have used that $q-(k+m+n)\geq0$ that $f,g$ and $h$, are analytic functions and that the constant
$M$ in Lemma~\ref{lemboundInql2} can be taken sufficiently small so that the series is convergent.

Finally, to prove the asymptotic expression~\eqref{asyMelni} of $M(\vu,\theta)$, we first take the definition~\eqref{melnifourier2} of the Melnikov function
and use bounds~\eqref{boundcoeffl} of $\Upsilon_0^{[l]}$ with $|l|\geq2$. Then, for $\vu\in\mathbb{R}$ and $\theta\in\mathbb{S}^1$, one has that:
\begin{align*}
M(\vu,\theta)=
&\cosh^{\frac{2}{\coef}}(\coef\vu)\left[\Upsilon_0^{[0]}+\Upsilon_0^{[1]}e^{i(\theta+\delta^{-1}\alpha\vu+c\coef^{-1}\log\cosh(\coef\vu))}\right.\nonumber\\
&\left.+\Upsilon_0^{[-1]}e^{-i(\theta+\delta^{-1}\alpha\vu+c\coef^{-1}\log\cosh(\coef\vu))}+\mathcal{O}\left(\delta^{p-\frac{2}{\coef}}
e^{-\frac{\alpha\pi}{2\coef\delta}\frac{3}{2}}\right)\right].
\end{align*}
Using the asymptotic formulas for $\Upsilon_0^{[1]}$ and $\Upsilon_0^{[-1]}$ and the fact that $\delta^{-i\frac{c}{\coef}}=e^{-i\frac{c}{\coef}\log\delta}$,
we obtain directly expression~\eqref{asyMelni}.
\end{proof}

\subsection{The average of the difference. Proof of Theorem~\ref{thmktilde0}}\label{sectionthmktilde0}
Note that both, $\Upsilon_0^{[0]}$ and $M^{[0]}(u)$ (in~\eqref{defcoefMelnikov} and~\eqref{defcoefMelnikovtilde}) are defined by means of an integral
involving $\Fout^{[0]}(0)$ which from~\eqref{defopF} turns out to be:
\begin{equation}\label{avgFout}
\Fout^{[0]}(0)(u) = 2\param \hetr(u)+\delta^p (F(0))^{[0]} + \delta^p \frac{\coef+1}{b} \hetz(u) (H(0))^{[0]}.
\end{equation}
In addition, from definition of $F,G$ and $H$ in~\eqref{notationFGHbis},
$$
(F(0))^{[0]}=\Fb^{[0]}(\delta \hetr(u),\delta \hetz(u),\delta), \qquad (H(0))^{[0]}=\Hb^{[0]}(\delta \hetr(u),\delta \hetz(u),\delta)
$$
with $\Fb^{[0]}, \Gb^{[0]}$ and $\Hb^{[0]}$ the average of the functions $\Fb,\Gb$ and $\Hb$ defined by~\eqref{defFGH}.

\subsubsection{The conservative case}
In this subsection we shall prove that the coefficients $\Upsilon^{[0]}$ and $\Upsilon_0^{[0]}$ are zero in the conservative case.
In this setting we have $\coef=1$ and $\param=0$. Whenever we refer to previous formulas and expressions where these parameters appear,
we shall substitute them for these values directly.

\begin{proposition}\label{propaveragemelni}
If the vector field~\eqref{initsys-outer2DX} is conservative, $\Upsilon_0^{[0]}=0$.
\end{proposition}
\begin{proof}
We consider the system
\begin{equation}\label{syspolar_average}
 \begin{aligned}
\frac{dr}{dt}&=\displaystyle -2rz+\delta^p\Fb^{[0]}(\delta r,\delta z,\delta),\\
\frac{d\theta}{dt}&=\displaystyle-\frac{\alpha}{\delta}-cz+\delta^p\Gb^{[0]}(\delta r,\delta z,\delta),\\
\frac{dz}{dt}&=\displaystyle-1+2br+z^2+\delta^p\Hb^{[0]}(\delta r,\delta z,\delta).
 \end{aligned}
\end{equation}
As system~\eqref{initsys-outer2DX} is conservative, system~\eqref{syspolar_average} is still conservative and one has:
\begin{equation}\label{divergence_average}
 \partial_r \Fb^{[0]}(\delta r,\delta z,\delta)=-\partial_z \Hb^{[0]}(\delta r,\delta z,\delta).
\end{equation}
Using~\eqref{divergence_average} one can easily see that system~\eqref{syspolar_average} has the following first integral:
$$
 \mathcal{U}(r,z)=-r+br^2+rz^2+\delta^p\int_0^r\Hb^{[0]}(\delta s,\delta z,\delta)ds.
$$
Note that, using definition~\eqref{defcoefMelnikov} of $\Upsilon_0^{[0]}$, with $\mathcal{F}^{[0]}(0)$ in~\eqref{avgFout},
and property~\eqref{divergence_average}:
$$
 \Upsilon_0^{[0]}=\int_{-\infty}^{+\infty}\frac{\mathcal{F}^{[0]}(0)(w)}{\cosh^2w}dw =
- \int_{-\infty}^{+\infty} \frac{d}{dw}\left(\mathcal{U}(\hetr(w),\hetz(w))\right) dw.
$$
Then, we have:
$$
\Upsilon_0^{[0]}=-\lim_{t\to\infty}\left[\mathcal{U}(\hetr(t),\hetz(t))-\mathcal{U}(\hetr(-t),\hetz(-t))\right].
$$
Noting that $(\hetr(\pm t),\hetz(\pm t))\to(0,\pm 1)$ as $t\to\pm\infty$ and that $\mathcal{U}(0,\pm 1)=0$, we obtain $\Upsilon_0^{[0]}=0$.
\end{proof}

Now we will prove that $\Upsilon^{[0]}=0$. This proof is more involved and requires some previous considerations. We shall use the fact that, in the conservative setting, the 2-dimensional invariant manifolds of $S_+$ and $S_-$ always intersect. This can be seen using standard arguments of volume preservation. Let us introduce some notation concerning this intersection. We fix $\theta_0\in[0,2\pi)$ and consider the following plane:
$$\Sigma_{\theta_0}=\{(x,y,z)\in\mathbb{R}^3\,:\, x\sin\theta_0-y\cos\theta_0=0\}.$$
We define $p_1$ as the first common intersection of the $2-$dimensional invariant manifolds of $S_+$ and $S_-$ contained in the section $\Sigma_{\theta_0}$. This point $p_1$ is $\mathcal{O}(\delta^{p+3})$-close to $(\frac{1}{b}\cos\theta_0,\frac{1}{b}\sin\theta_0,0)$, which is the first intersection in the unperturbed case. The orbit of $p_1$, namely:
\begin{equation}\label{orbitp1}
\Gamma_{p_1}:=\{\varphi_t(p_1),\,t\in\mathbb{R}\},
\end{equation}
where $\varphi_t$ stands for the flow the vector field~\eqref{initsys-outer2D}, is a heteroclinic orbit and for small $\delta$ it intersects many times the section $\Sigma_{\theta_0}$. We define:
$$t_2=\min\{t>0\,:\,\varphi_t(p_1)\in\Sigma_{\theta_0}\},\quad p_2=\varphi_{t_2}(p_1),$$
and:
$$t_3=\min\{t>t_2\,:\,\varphi_t(p_1)\in\Sigma_{\theta_0}\},\quad p_3=\varphi_{t_3}(p_1).$$
\begin{remark}\label{rmkanglep3}
Note that, $\dot\theta<0$ provided that $\delta$ is sufficiently small. Indeed, this can be easily seen since
$\dot{\theta} = -\alpha /\delta - c z +\delta^p G(\delta r, \theta, \delta z, \delta,\delta \param)$. Then $p_2$ has angular variable
$\theta_0-\pi$ and $p_3$ has angular variable $\theta_0-2\pi$.
\end{remark}

We define $z_i$ and $u_i$ as:
\begin{equation}\label{defzui}
z_i=\pi_z(p_i),\qquad \vu_i=\hetz^{-1}(z_i)=\mathrm{atanh}\,(z_i),\qquad i=1,2,3.
\end{equation}
with $\pi_z$ the projection on the third component. See Figure~\ref{figureT1}.
\begin{figure}
        \centering
        \begin{subfigure}[b]{0.45\textwidth}
	  \centering
		\includegraphics[width=5.5cm]{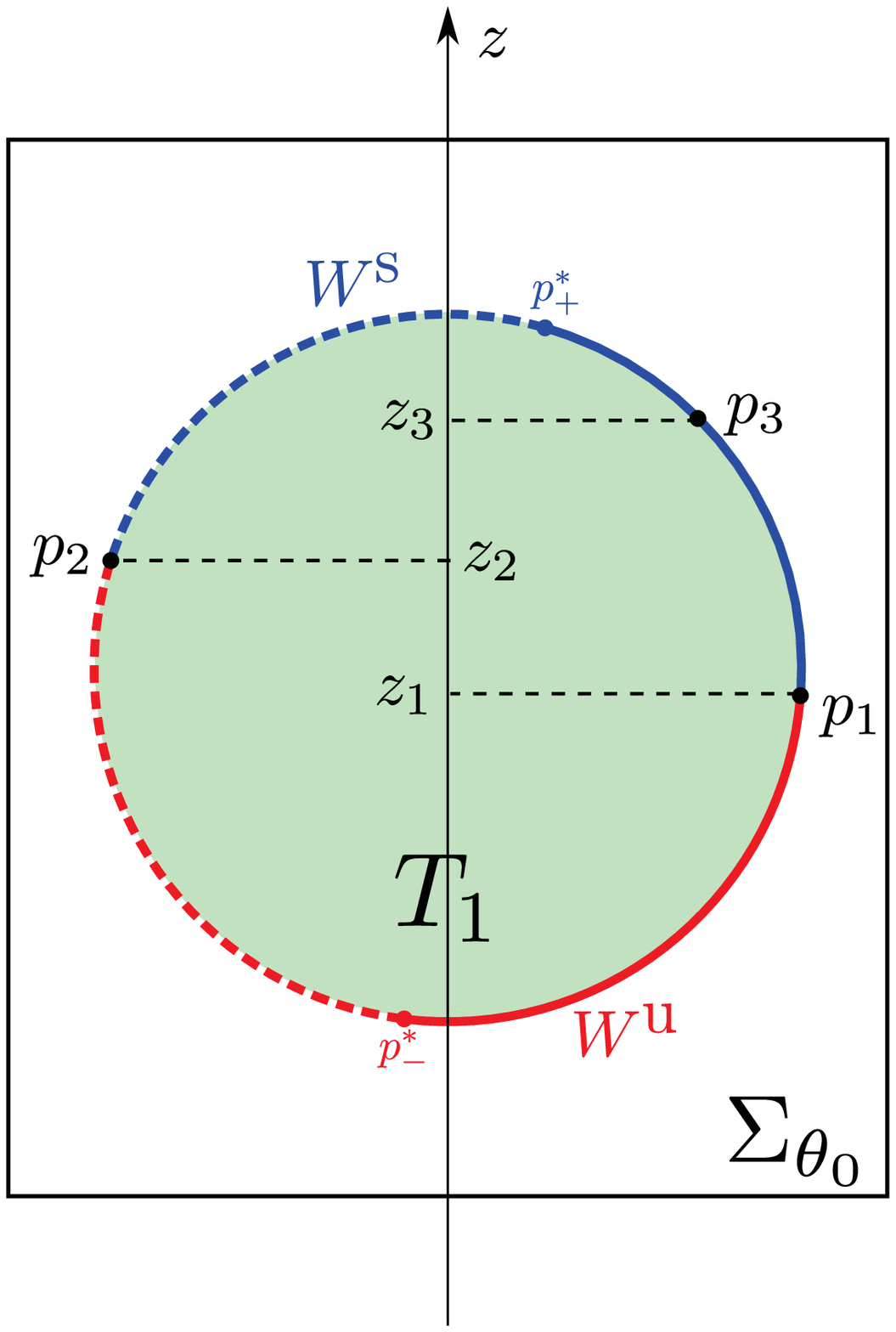}
		\caption{The domain $T_1$ on the section $\Sigma_{\theta_0}$.}
               \label{figureT1}
        \end{subfigure}%
         \qquad
        \begin{subfigure}[b]{0.45\textwidth}
			
	  \centering
		\includegraphics[width=5.5cm]{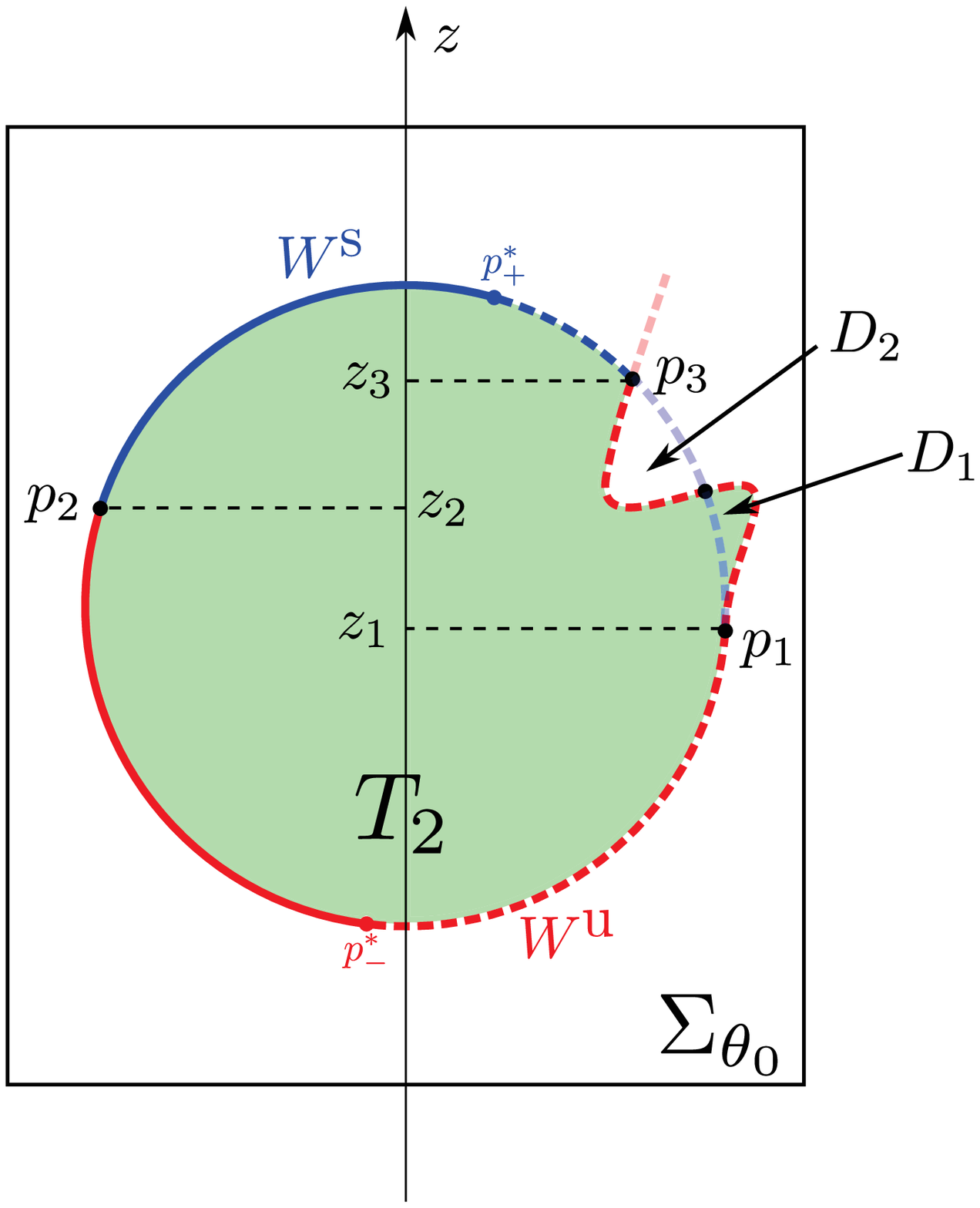}
		\caption{The domain $T_2$ on the section $\Sigma_{\theta_0}$.}
               \label{figureT2}
       \end{subfigure}
       \caption[The domains $T_1$ and $T_2$]{The domains $T_1$ and $T_2$. In red, the unstable manifold of $S_-$, and in blue the stable manifold of $S_+$.
			The continuous (respectively, discontinuous) lines on the left are mapped to the continuous (discontinuous) lines on the right with the same color
			via the flow $\phi$.}\label{figureT1T2}
\end{figure}

We point out that with this notation we can write:
$$\Delta(\vu_i,\theta_0)=0,\qquad i=1,2,3,$$
where as usual $\Delta(\vu,\theta)=r^\uns(\vu,\theta)-r^\sta(\vu,\theta)$.

\begin{lemma}\label{lemu1u3}
Let $\vu_1$ and $\vu_3$ be defined as in~\eqref{defzui}. Define:
$$\tau^*=\xi(\vu_1,\theta_0)=\theta_0+\delta^{-1}\alpha\vu_1+c\log\cosh\vu_1+C(\vu_1,\theta_0),$$
where $\xi(\vu,\theta)$ and $C(\vu,\theta)$ are the functions given in Theorem~\ref{thmdifpartsolinjective}. Then:
$$\xi(\vu_3,\theta_0)=\theta_0+\delta^{-1}\alpha\vu_3+c\log\cosh\vu_3+C(\vu_3,\theta_0)=\tau^*+2\pi.$$
\end{lemma}
\begin{proof}
Let $s_0>1$. For any $s\in[-s_0,s_0]$, we define $u=u(s)$ as the (unique) solution
of:
\begin{equation}\label{equs}
\xi(\vu(s),\theta_0-2\pi s)=\tau^*.
\end{equation}
The fact that equation~\eqref{equs} has a unique solution for all $s\in[-s_0,s_0]$ if
$\delta$ is sufficiently small can be seen, for instance, by the implicit function theorem.
By definition of $\tau^*$, the unique solution at $s=0$ is $u(0)=u_1$.

Now, since $\Delta(\vu_1,\theta_0)=0$ and $\Delta(\vu,\theta)=\cosh^2(\vu)(1+P_1(\vu,\theta))\tilde k(\xi(\vu,\theta))$ by Theorem~\ref{thmdifpartsolinjective}, using  that $\cosh(\vu_1)\neq0$ and that $P_1$ is small we have:
$$0=\tilde k(\xi(\vu_1,\theta_0))=\tilde k(\tau^*)=\tilde k(\xi(\vu(s),\theta_0-2\pi s)).$$
Thus $\Delta(\vu(s),\theta_0-2\pi s)=0$. Hence defining
$$r_{\rm h}(s):=r^\uns(\vu(s),\theta_0-2\pi s)=r^\sta(\vu(s),\theta_0-2\pi s),$$
we have that the curve:
$$\gamma_{\rm h}(s):=(r_{\rm h}(s),\theta_0-2\pi s,\hetz(\vu(s)),\qquad s\in[-s_0,s_0],$$
is part of a heteroclinic orbit expressed in the symplectic polar coordinates. Since $u(0)=u_1$ and $p_1$ in these coordinates is
$(r^\uns(u_1,\theta_0),\theta_0,\hetz(u_1))=\gamma_{\rm h}(0)$, clearly $\gamma_{\rm h}(s)$ is a part of the heteroclinic orbit
$\Gamma_{p_1}$, defined in~\eqref{orbitp1}.

Taking $s=1$, we obtain the point in $\Gamma_{p_1}$ with angular variable $\theta_0-2\pi$. By Remark~\ref{rmkanglep3},
this point is precisely $p_3$. This implies that $u(1)=u_3$, and then equation~\eqref{equs} yields.
$$\theta_0-2\pi+\delta^{-1}\alpha\vu_3+c\log\cosh\vu_3+C(\vu_3,\theta_0-2\pi)=\tau^*,$$
and since $C(\vu,\theta)$ is $2\pi$ periodic in $\theta$ we obtain:
$$\theta_0+\delta^{-1}\alpha\vu_3+c\log\cosh\vu_3+C(\vu_3,\theta_0)=\tau^*+2\pi.$$
\end{proof}

\begin{lemma}\label{lemupsilon0change}
Let $\vu_1$ and $\vu_3$ be the $u-$coordinate of the heteroclinic points $p_1,p_3\in\Sigma_{\theta_0}$ respectively, defined in~\eqref{defzui}. Let $\Upsilon^{[0]}$ be the average of the function $\tilde k(\tau)$ given in Theorem~\ref{thmdifpartsolinjective}. Then one has:
$$\Upsilon^{[0]}=\frac{1}{2\pi}\int_{u_1}^{u_3}\frac{\Delta(\vu,\theta_0)}{\cosh^2(\vu)(1+P_1(\vu,\theta_0))}\left(\delta^{-1}\alpha+c\hetz(\vu)+\partial_\vu C(\vu,\theta_0)\right)du.$$
\end{lemma}
\begin{proof}
 It can be obtained straightforwardly from the fact that:
$$\Upsilon^{[0]}=\frac{1}{2\pi}\int_{\tau^*}^{\tau^*+2\pi}\tilde k(\tau)d\tau,$$
where $\tau^*=\theta_0+\delta^{-1}\alpha\vu_1+c\log\cosh\vu_1+C(\vu_1,\theta_0)$.
Indeed, one just has to perform the change $\tau=\theta_0+\delta^{-1}\alpha\vu+c\log\cosh\vu+C(\vu,\theta_0)$.
Then, recalling that by Theorem~\ref{thmdifpartsolinjective}:
$$\Delta(\vu,\theta_0)=\cosh^2(\vu)(1+P_1(\vu,\theta_0))\tilde k(\theta_0+\delta^{-1}\alpha u+c\log\cosh(\vu)+C(\vu,\theta_0)),$$
and that by Lemma~\ref{lemu1u3}:
$$\theta_0+\delta^{-1}\alpha\vu_3+c\log\cosh\vu_3+C(\vu_3,\theta_0)=\tau^*+2\pi,$$
one obtains the claim of the lemma.
\end{proof}

\begin{proposition}\label{propintegral0}
One has:
\begin{equation}\label{integralu}
\int_{u_1}^{u_3}\frac{\Delta(\vu,\theta_0)}{\cosh^2(\vu)}\left(\delta^{-1}\alpha+c\hetz(\vu)+l_3(\vu,\theta_0)\right)du=0.
\end{equation}
\end{proposition}
\begin{proof}
Let us denote by $\tilde r^\uns(z,\theta):=r^\uns(\hetz^{-1}(z),\theta)$ the $r-$component of the unstable manifold of $S_-$ as a function of $z$ and $\theta$,
and similarly $\tilde r^\sta(z,\theta)$ for the stable manifold of $S_+$. We denote:
$$\tilde G(r,z)=\Gb(\delta r,\theta_0,\delta z,\delta),$$
where $\Gb$ is the function defined in~\eqref{defFGH} (recall that in the conservative case there is no dependence on the parameter $\param$).
We shall prove the following:
\begin{equation}\label{integralz}
 \int_{z_1}^{z_3}\int_{\tilde r^\sta(z,\theta_0)}^{\tilde r^\uns(z,\theta_0)}(\delta^{-1}\alpha+cz-\delta^p\tilde G(r,z))drdz=0,
\end{equation}
with $z_1$ and $z_3$ defined in~\eqref{defzui}.
This yields claim~\eqref{integralu}. Indeed, assume~\eqref{integralz} is true. Then we make the change:
$$
r=\tilde r_\lambda:=\frac{1}{2}(\tilde r^\uns(z,\theta_0)+\tilde r^\sta(z,\theta_0))+
\frac{\lambda}{2}(\tilde r^\uns(z,\theta_0)-\tilde r^\sta(z,\theta_0)),\quad\lambda\in[-1,1],
$$
and, denoting $\tilde \Delta(z,\theta)=\tilde r^\uns(z,\theta_0)-\tilde r^\sta(z,\theta_0)$, equation~\eqref{integralz} becomes:
$$ \int_{z_1}^{z_3}\left(\delta^{-1}\alpha+cz-\frac{1}{2}\int_{-1}^{1}\delta^p\tilde G(\tilde r_\lambda,z)d\lambda\right)\tilde \Delta(z,\theta_0)dz=0.$$
Then, we perform the change $z=\hetz(\vu)$ and recalling the definition~\eqref{defzui} of $u_1$ and $u_3$, and the definition~\eqref{defl3} of $l_3$
we obtain~\eqref{integralu}.

To prove~\eqref{integralz} we shall use basically that the system is divergence-free, and apply the divergence theorem in a suitable $3-$dimensional domain.
However, we first need to introduce some notation. Consider the intersection of the $2-$dimensional unstable manifold of $S_-$ and $\Sigma_{\theta_0}$.
The lower part of this intersection is a curve that joins $p_1$ and $p_2$, having a shape close to an arch of ellipse.
Similarly, if we consider the intersection of the $2-$dimensional stable manifold of $S_+$ and $\Sigma_{\theta_0}$, its upper part is a curve that also joins
$p_1$ and $p_2$, with a similar shape. We define $T_1\subset\Sigma_{\theta_0}$
as the domain bounded by these two curves (see Figure~\ref{figureT1}).

As in~\eqref{initsys-outer2DX} we denote by $X$ the vector field defining our system and $X_x$, $X_y$ and $X_z$ each of its components. We note that if $p\in\partial T_1$ and:
$$X_x(p)\sin\theta_0-X_y(p)\cos\theta_0\neq0$$
then there exists a unique $\tau(p)>0$ such that $\varphi_{\tau(p)}(p)$ is the next intersection of the orbit going through $p$ and $\Sigma_{\theta_0}$. This is clear from the fact that the orbits inside $W^\uns(S_-)$ are $\mathcal{O}(\delta)-$close to the orbits of the heteroclinic connection of the unperturbed system
for $t\in(-\infty,T]$, for some constant $T$, and the same happens for the orbits inside $W^\sta(S_+)$ and $t\in[T,\infty)$. Moreover there are just two points $p_-^*,p_+^*\in\Sigma_{\theta_0}$ (close to $S_-$ and $S_+$ respectively) such that:
$$X_x(p_\pm^*)\sin\theta_0-X_y(p_\pm^*)\cos\theta_0=0.$$
See Figure~\ref{figureT1}. For such points we can define $\tau(p_\pm^*)=0$. With this definition, the function $\varphi_{\tau(p)}(p)$
is continuous for $p\in\partial T_1$. Then we define $T_2\subset\Sigma_{\theta_0}$ (see Figure~\ref{figureT2}) as the domain bounded by $\partial T_2$, where:
$$\partial T_2=\{\varphi_{\tau(p)}(p)\,:\,p\in\partial T_1\}.$$
Finally we define:
$$T_3=\{\varphi_{t}(p)\,:\,p\in\partial T_1,\,t\in(0,\tau(p))\}.$$
We point out that $T_3$ is tangent to the flow of $X$. Moreover, $T_1$, $T_2$ and $T_3$ are the boundary of a closed $3-$dimensional domain. That is, there exists a closed domain $V\subset\mathbb{R}^3$ such that
$T_1\cup T_2\cup T_3=\partial V$. Now we use the divergence theorem in this domain $V$. Since $\textrm{div} X\equiv 0$ we have:
\begin{eqnarray}\label{divthm}
 0&=&\iiint_V\textrm{div}XdV=\iint_{\partial V}X\cdot\vec{n}_{\partial V}dS\nonumber\\
 &=&\iint_{T_1}X\cdot\vec{n}_{T_1}dS+\iint_{T_2}X\cdot\vec{n}_{T_2}dS+\iint_{T_3}X\cdot\vec{n}_{T_3}dS,
\end{eqnarray}
where $\vec{n}_{\partial V}$ denotes the unitary normal vector to $\partial V$ pointing outside $V$, and the same with $\vec{n}_{T_i}$, $i=1,2,3$.
Since $T_3$ is tangent to the flow, $X\cdot\vec{n}_{T_3}=0$ and moreover,
$\vec{n}_{T_1}=(-\sin\theta_0,\cos\theta_0,0)=-\vec{n}_{T_2}$.
Thus~\eqref{divthm} becomes:
\begin{equation}\label{difT1T2}
 0=\iint_{D_1}(X_x\sin\theta_0-X_y\cos\theta_0)dS -\iint_{D_2}(X_x\sin\theta_0-X_y\cos\theta_0)dS,
\end{equation}
where $D_1=T_2\setminus T_1$ and $D_2=T_1\setminus T_2$ (see Figure~\ref{figureT2}). We take the parameterization:
$$x=\sqrt{2r}\cos\theta_0,\quad y=\sqrt{2r}\sin\theta_0,\quad z=z$$
and we note that
\begin{align*}
X_x(\sqrt{2r}\cos\theta_0,\sqrt{2r}\sin\theta_0,z)\sin\theta_0&-X_y(\sqrt{2r}\cos\theta_0,\sqrt{2r}\sin\theta_0,z)\cos\theta_0\\
&=\sqrt{2r}\left(\delta^{-1}\alpha+cz-\delta^p\tilde G(r,z)\right).
\end{align*}
With this parameterization, equality~\eqref{difT1T2} yields~\eqref{integralz}.
\end{proof}

\begin{proof}[End of the proof of Theorem~\ref{thmktilde0} (conservative case)]
Proposition~\ref{propaveragemelni} says that $\Upsilon_0^{[0]}=0$. To see that $\Upsilon^{[0]}=0$ we note that from item 2. in
Theorem~\ref{thmdifpartsolinjective},
we choose $P_1$ such that:
$$\frac{\delta^{-1}\alpha+c\hetz(\vu)+\partial_\vu C(\vu,\theta_0)}{1+P_1(\vu,\theta_0)}=\delta^{-1}\alpha+c\hetz(\vu)+l_3(\vu,\theta_0).$$
Then, substituting this in the equality of Lemma~\ref{lemupsilon0change} we get:
$$\Upsilon^{[0]}=\frac{1}{2\pi}\int_{u_1}^{u_3}\frac{\Delta(\vu,\theta_0)}{\cosh^2(\vu)}\left(\delta^{-1}\alpha+c\hetz(\vu)+l_3(\vu,\theta_0)\right)du.$$
Finally, Proposition~\ref{propintegral0} yields that $\Upsilon^{[0]}=0$, and the proof is finished.
\end{proof}

\subsubsection{The dissipative case}
In this section we will prove the statements about the coefficients $\Upsilon^{[0]}_0$ and $\Upsilon^{[0]}$ in Theorem~\ref{thmktilde0}.
We have:
\begin{equation}\label{upsilon0}
\Upsilon^{[0]}=\frac{1}{2\pi}\int_0^{2\pi}\tilde k(\tau)d\tau.
\end{equation}
We perform the change $\tau=\theta+C(0,\theta)$ in the previous integral, where $C(\vu,\theta)$ is  the function in Theorem~\ref{thmdifpartsolinjective} and we use that by Theorem~\ref{thmdifpartsolinjective} we have:
$$\tilde k(\theta+\delta^{-1}\alpha u+c\coef^{-1}\log\cosh(\coef\vu)+C(\vu,\theta))=\frac{\Delta(\vu,\theta)}{\cosh^{2/\coef}(\coef\vu)(1+P_1(\vu,\theta))}.$$
After this change~\eqref{upsilon0} becomes:
\begin{equation}\label{upsilon0v2}
\Upsilon^{[0]}=\frac{1}{2\pi}\int_{\theta_1}^{\theta_2}\frac{\Delta(0,\theta)}{1+P_1(0,\theta)}(1+\partial_\theta C(0,\theta))d\theta,
\end{equation}
where, using bounds for $C(0,\theta)$ obtained in Theorem~\ref{thmdifpartsolinjective},
\begin{equation}\label{boundstheta12}
 \theta_1=0+\mathcal{O}\left(\delta^{p+3}\right)\qquad \theta_2=2\pi+\mathcal{O}\left(\delta^{p+3}\right).
\end{equation}
Now, on the one hand, by Theorem~\ref{thmoutloc} we have:
\begin{equation}\label{boundDeltarough}
|\Delta(0,\theta)|\leq |r_1^\uns(0,\theta)|+|r_1^\sta(0,\theta)|\leq K\delta^{p+3}.
\end{equation}
and recalling the notation $M(\vu,\theta)=r^\uns_{10}(\vu,\theta)-r^\sta_{10}(\vu,\theta)$:
$$
|\Delta(0,\theta)-M(0,\theta)|\leq |r^\uns_{11}(0,\theta)|+|r^\sta_{11}(0,\theta)|\leq K\left(\delta^{2p+6}+\delta^{p+4}\right),
$$
where we have used the bounds of $r_{11}^{\uns,\sta}(\vu,\theta)$ given in Theorem~\ref{thmoutloc}. On the other hand, by Theorem~\ref{thmdifpartsolinjective}:
\begin{equation}\label{boundpartial}
 |\partial_\theta C(0,\theta)|\leq K\delta^{p+3}, \qquad \left|\frac{1}{1+P_1(0,\theta)}-1\right|\leq K|P_1(0,\theta)|\leq K\delta^{p+3}.
\end{equation}
Thus, using bounds~\eqref{boundstheta12},~\eqref{boundDeltarough} and~\eqref{boundpartial} in equation~\eqref{upsilon0v2} we obtain:
\begin{equation}\label{upsilon0v3}
\Upsilon^{[0]}=\frac{1}{2\pi}\int_{0}^{2\pi}M(0,\theta)d\theta+\mathcal{O}(\delta^{p+4})=M^{[0]}(0)+\mathcal{O}(\delta^{p+4}),
\end{equation}
where we have used that $p\geq-2$.

We introduce the following notation:
\begin{align*}
I&=\frac{\coef+1}{b}\int_{-\infty}^{+\infty}\frac{1}{\cosh^{\frac{2}{\coef}+2}(\coef w)}dw,
\\J&=\delta^{-3}\int_{-\infty}^{+\infty}\frac{(F(0))^{[0]}+\frac{\coef+1}{b}\hetz(w)(H(0))^{[0]}}{\cosh^{\frac{2}{\coef}}(\coef w)}dw
\end{align*}
and observe that for all $w\in\mathbb{R}$:
$$|F(0)|=|\Fb(\delta\hetr(w),\theta,\delta\hetz(w),\delta,\delta\param)|\leq K\delta^{3}$$
and also $|H(0)|\leq K\delta^3$, so that $J$ is bounded as $\delta\to 0$.
Now, by formula~\eqref{defcoefMelnikovtilde} of $M^{[l]}(\vu)$ and expression~\eqref{avgFout} of $\Fout^{[0]}$, we get:
$M^{[0]}(0) = \param I + \delta^{p+3} J$. We rewrite~\eqref{upsilon0v3} as:
$$\Upsilon^{[0]}=\param I+\delta^{p+3}J+\mathcal{O}(\delta^{p+4}).$$
Then, putting $\param=\hat\param\delta^{p+3}$, we have that $\Upsilon^{[0]}=a_1\delta^{a_2}e^{-\frac{a_3\pi}{2\coef\delta}}$
if:
$$f(\hat\param,\delta):=\hat\param I+J+\mathcal{O}(\delta)-a_1\delta^{a_2-p-3}e^{-\frac{a_3\pi}{2\coef\delta}}=0.$$
It is clear that $I\neq0$, and thus:
$$f\left(-\frac{J}{I},0\right)=0,\qquad \frac{\partial f}{\partial\hat\param}\left(-\frac{J}{I},0\right)=I\neq0,$$
where we have used that $a_3>0$ so that the last term and all its derivatives vanish at $\delta=0$.  Then we can apply the implicit function theorem, so that there exists $\delta_0$ and a curve $\hat\param_*(\delta)=-J/I+\mathcal{O}(\delta)$ such that $f(\hat\param_*(\delta),\delta)=0$ for all $0\leq\delta\leq\delta_0$. The curve $\param_*(\delta):=\hat\param_*(\delta)\delta^{p+3}$ is the one in the statement of the lemma.

Clearly, since $\Upsilon_0^{[0]}=M^{[0]}(0)=\param I+\delta^{p+3}J$, one has:
$$
\Upsilon_0^{[0]}=\Upsilon_0^{[0]}(\delta,\param_*(\delta))=
\Upsilon_0^{[0]}\left(\delta,-\frac{J}{I}\delta^{p+3}+\mathcal{O}(\delta^{p+4})\right)=\mathcal{O}(\delta^{p+4}).
$$

\subsection{The exponentially smallness of $\Upsilon^{[l]}$. Proof of Lemma~\ref{lemUpsilonlexpsmall}} \label{subsec:sharpbound}
Let us to introduce the function
\begin{equation*}
F(\vu,\theta)=\delta \alpha^{-1}(\xi(\vu,\theta)-\theta)=\vu+\delta\alpha^{-1}\left[c\coef^{-1}\log\cosh(\coef\vu)+C(\vu,\theta)\right],
\end{equation*}
where $\xi$ and $C$ are defined in Theorem~\ref{thmdifpartsolinjective}.
In this result is proven that $(\xi(\vu,\theta),\theta)$
is injective in $\Doutinter\times\Tout$ then $(F(\vu,\theta),\theta)$ is also injective in the same domain.
In particular, for all $(\vu,\theta)\in\Doutinter\times\mathbb{S}^1$, the change $(w,\theta)=(F(u,\theta),\theta)$ is a diffeomorphism between $\Doutinter\times\mathbb{S}^1$ and its image $\Doutintertilde\times\mathbb{S}^1$, with inverse $(\vu,\theta)=(G(w,\theta),\theta)$. Then, if we define the function:
$$
\Diffw(w,\theta)=\sum_{l\in\mathbb{Z}}\Upsilon^{[l]}e^{il(\theta+\delta^{-1}\alpha w)}.
$$
one has that $G(w,\theta)$ satisfies:
\begin{equation}\label{DeltaDeltatilde-inner}
 \Diffw(w,\theta)=\frac{\Delta(G(w,\theta),\theta)}{\cosh^{2/\coef}(\coef G(w,\theta))(1+P_1(G(w,\theta),\theta))}.
\end{equation}
Note that $\Diffw(w,\theta)$ is $2\pi-$periodic in $\theta$, and its $l-$th Fourier coefficient is:
$$
\Diffw^{[l]}(w)=\Upsilon^{[l]}e^{il\delta^{-1}\alpha w}.
$$
Hence we know that for all $w\in \Doutintertilde$:
\begin{equation}\label{boundinterupsilons-inner}
\left|\Upsilon^{[l]}\right|=\frac{1}{2\pi}\left|e^{-i\delta^{-1}\alpha wl}
\int_0^{2\pi}\Diffw(w,\theta)e^{-il\theta}d\theta\right|
\leq\left|e^{-i\delta^{-1}\alpha wl}\right|\sup_{\theta\in\mathbb{S}^1}\left|\Diffw(w,\theta)\right|.
\end{equation}
This inequality is valid for all $\vw\in\Doutintertilde$. Let us denote $\vu_\pm=\pm i\left(\frac{\pi}{2\coef}-\dist\delta\right)$. Then, if in \eqref{boundinterupsilons-inner} we take $w=w_+:=F\left(u_+,\theta\right)\in\Doutintertilde$ for $l<0$ and $w=w_-:=F\left(u_-,\theta\right)\in\Doutintertilde$ for $l>0$, one obtains:
\begin{equation}\label{boundinterupsilons-inner-v2}
\left|\Upsilon^{[l]}\right|\leq e^{-\left(\frac{\alpha\pi}{2\coef\delta}-\alpha\dist-|\im C(\vu_\pm,\theta)|\right)|l|}
\sup_{\theta\in\mathbb{S}^1}\left|\Diffw(w_\pm,\theta)\right|.
\end{equation}
Recall that $F$ is the inverse of $G$, so that from \eqref{DeltaDeltatilde-inner} we obtain:
$$
\Diffw(w_\pm,\theta)=\frac{\Delta(\vu_\pm,\theta)}{\cosh^{2/\coef}(\coef\vu_\pm)(1+P_1(\vu_\pm,\theta))}.
$$
Thus, using bound~\eqref{boundP1-prop} for $P_1$, that $|\cosh (\coef \vu_\pm) |\geq K \delta \dist$, and taking $\dist$ sufficiently large,
bound~\eqref{boundinterupsilons-inner-v2} writes out as:
\begin{equation}\label{boundinterupsilons-supDelta}
\left|\Upsilon^{[l]}\right|\leq \frac{K}{\delta^{\frac{2}{\coef}}\dist^{\frac{2}{\coef}}}
e^{-\left(\frac{\alpha\pi}{2\coef\delta}-\alpha\dist-|\im C(\vu_\pm,\theta)|\right)|l|}\sup_{\theta\in\mathbb{S}^1}\left|\Delta(\vu_\pm,\theta)\right|.
\end{equation}

Now, on the one hand, taking into account that the constant $L_0$, given in Theorem \ref{thmdifpartsolinjective}, satisfies $L_0\in\mathbb{R}$, we have:
$$
|\im C(\vu_\pm,\theta)|\leq \coef^{-1}(c+\alpha L_0)|\im\log\cosh(\coef\vu_\pm)|+\alpha |L(\vu_\pm)|+|\chi(\vu_\pm,\theta)|.
$$
Since $\vu_\pm$ is purely imaginary, $\im\log\cosh(\coef u_\pm)=\arg(\cosh(\coef u_\pm))=0.$
Then, using~\eqref{boundLchi-prop} in Theorem \ref{thmdifpartsolinjective}, we obtain $|\im C(\vu_\pm,\theta)|\leq K \delta^{p+2}.$
Therefore:
\begin{equation}\label{boundexp-distC-l1}
 \left|e^{-\left(\frac{\alpha\pi}{2\coef\delta}-\alpha\dist-|\im C(\vu_\pm,\theta)|\right)}\right|\leq Ke^{-\frac{\alpha\pi}{2\coef\delta}+\alpha\dist}.
\end{equation}
Moreover, we take $\delta$ sufficiently small so that:
$$1-\frac{2\coef\delta}{\alpha\pi}\left(\alpha\dist+|\im C(\vu_\pm,\theta)|\right)\geq \frac{3}{4},$$
and then, for $|l|\geq2$, one has:
\begin{equation}\label{boundexp-distC-lgeq2}
 \left|e^{-\left(\frac{\alpha\pi}{2\coef\delta}-\alpha\dist-|\im C(\vu_\pm,\theta)|\right)|l|}\right|\leq e^{-\frac{\alpha\pi}{2\coef\delta}\frac{3|l|}{4}}.
\end{equation}
On the other hand, by Theorem \ref{thmoutloc} we have:
\begin{equation}\label{boundsupDeltatilde}
|\Delta(\vu_\pm,\theta)|\leq |r_1^\uns(\vu_\pm,\theta)|+|r_1^\sta(\vu_\pm,\theta)|\leq \frac{K\delta^{p+3}}{|\cosh(\coef\vu_\pm)|^3}\leq K\frac{\delta^{p}}{\dist^3}.
\end{equation}

To obtain the claim of the lemma for $|l|=1$, we use bounds \eqref{boundexp-distC-l1} and \eqref{boundsupDeltatilde} in equation \eqref{boundinterupsilons-supDelta}.
Similarly, for $|l|\geq2$ we use bounds~\eqref{boundexp-distC-lgeq2} and~\eqref{boundsupDeltatilde} in equation~\eqref{boundinterupsilons-supDelta}.

\subsection{Fourier coefficients of $\Delta_1$. Proof of Proposition~\ref{properrorUpsilons}}\label{secproofproperrorUpsilons}
Consider the function $\Delta_1(\vu,\theta)=\Delta(\vu,\theta)-\Delta_0(\vu,\theta)$ defined in~\eqref{Delta01Fourier}:
$$
\Delta_1(\vu,\theta)=\cosh^{2/\coef}(\coef\vu)(1+P_1(\vu,\theta))\sum_{l\neq0}\left(\Upsilon^{[l]}-\hat \Upsilon_0^{[l]}\right)e^{il \xi(\vu,\theta)},
$$
with $\xi(\vu,\theta)=\theta+\delta^{-1}\alpha\vu+\coef^{-1}\log\cosh(\coef\vu)+C(\vu,\theta)$ defined in~\eqref{defxitheorem} and
$\Upsilon^{[l]}, \hat \Upsilon_0^{[l]}$ the Fourier coefficients of $\tilde{k}$ and $\tilde{k}_0$ in~\eqref{defDelta} and~\eqref{defktilde0} respectively.
We point out that in order to obtain sharp bounds for $\Upsilon^{[\pm 1]}-\hat \Upsilon_0^{[\pm 1]}$ we need to take
$\vu\in\Doutinter\subset\mathbb{C}$ (see \eqref{Doutinter}), but $\theta$ can be taken real.
Thus, we will take $\theta\in\mathbb{S}^1$.

Proceeding as in beginning of the previous section~\ref{subsec:sharpbound}, one can prove the following bound for
$|\Upsilon^{[\pm 1]}-\hat \Upsilon^{[\pm 1]}_0|$
which similar to the one for $|\Upsilon^{[l]}|$ in~\eqref{boundinterupsilons-supDelta}:
$$
\left|\Upsilon^{[\mp 1]}-\hat \Upsilon^{[\mp 1]}_0\right|\leq \frac{K}{\delta^{\frac{2}{\coef}}\dist^{\frac{2}{\coef}}}
e^{-\left(\frac{\alpha\pi}{2\coef\delta}-\alpha\dist-|\im C(\vu_\pm,\theta)|\right)}\sup_{\theta\in\mathbb{S}^1}\left|\Delta_1(\vu_\pm,\theta)\right|,
$$
where $\vu_\pm=\pm i\left(\frac{\pi}{2\coef}-\dist\delta\right)$.
Using bound~\eqref{boundexp-distC-l1}, we obtain:
\begin{equation}\label{boundUpsilons1-1}
\left|\Upsilon^{[\mp 1]}-\hat \Upsilon^{[\mp 1]}_0\right|\leq \frac{K}{\delta^{\frac{2}{\coef}}\dist^{\frac{2}{\coef}}}
e^{-\frac{\alpha\pi}{2\coef\delta}+\alpha\dist}\sup_{\theta\in\mathbb{S}^1}\left| \Delta_1(\vu_\pm,\theta)\right|.
\end{equation}

We claim that exists a constant $K$ such that for all $\theta\in\mathbb{S}^1$:
\begin{equation}\label{boundDelta1rough}
|\Delta_1(\vu_\pm,\theta)|\leq K\left(\frac{\delta^{2(p+1)} |\log \dist|}{\dist^4}+\frac{\delta^{p+3}}{\dist}\right).
\end{equation}
Indeed, first we write $\Delta_1=\Delta-\Delta_0$ in a more adequate form.
We recall that, by definition~\eqref{defMelnikov} of the Melnikov function $M$,
$\Delta = M + r_{11}^{\uns}-\r_{11}^{\sta}$. Then,
\begin{equation*}
\Delta_1(\vu,\theta) = M(\vu,\theta)-\Delta_0(\vu,\theta) + r_{11}^{\uns}(\vu,\theta)-\r_{11}^{\sta}(\vu,\theta).
\end{equation*}
It is clear that, by Theorem~\ref{thmoutloc},
\begin{equation*}
|r_{11}^\uns(\vu_\pm,\theta)-r_{11}^\sta(\vu_\pm,\theta)|\leq K\left( \frac{\delta^{2(p+1)}}{\dist^4}+\frac{\delta^{p+3}}{\dist}\right),
\end{equation*}
which is smaller than the upper bound in~\eqref{boundDelta1rough}. Therefore, to prove~\eqref{boundDelta1rough},
it only remains to study the difference between $M$ and $\Delta_0$.

We introduce some notation:
\begin{equation}\label{defCtilde}
\begin{aligned}
F_0(\vu) = u+ \delta \alpha^{-1} c \coef^{-1} &\log\cosh (\coef u),\qquad \hat C(\vu,\theta) = C(\vu,\theta)- \alpha \coef^{-1} L_0 \delta^{p+2} \log (\delta) \\
&\hat F(\vu,\theta) = F_0(\vu) +\delta \alpha^{-1} \hat C(\vu,\theta).
\end{aligned}
\end{equation}
Notice that $F_0(\vu)$ is injective so that it has an inverse. We also introduce the function
\begin{equation*}
f(\vu,\theta) = F_0^{-1} (\hat F(\vu,\theta))
\end{equation*}
and we note that, since by  \eqref{formaCtheorem} in Theorem~\ref{thmdifpartsolinjective}, $|\hat C(\vu_\pm,\theta)|\leq K\delta^{p+2} |\log \dist|$
(see~\eqref{defCtilde})
\begin{equation}\label{boundfu}
|f(\vu_\pm,\theta) -u_\pm |\leq K\delta^{p+3} |\log \dist|.
\end{equation}
Now we rewrite $M(\vu,\theta)$ in~\eqref{melnifourier2} as
\begin{equation*}
 M(\vu,\theta)=\cosh^{\frac{2}{\coef}}(\coef\vu)\sum_{l\in\mathbb{Z}}\Upsilon_0^{[l]}e^{il(\theta+\delta^{-1}\alpha F_0(\vu))}
\end{equation*}
and we observe that
\begin{equation}\label{melnif}
M(f(\vu,\theta),\theta) = \cosh^{\frac{2}{\coef}}(\coef f(\vu,\theta))\sum_{l\in\mathbb{Z}}\Upsilon_0^{[l]}e^{il(\theta+\delta^{-1}\alpha \hat F(\vu,\theta))}.
\end{equation}
In addition, the function $\Delta_0$ in~\eqref{defDelta0} is:
$$
\Delta_0(\vu,\theta) = \cosh^{2/\coef}(\coef\vu)(1+P_1(\vu,\theta))\left (\Upsilon^{[0]} + \sum_{l \neq 0} \Upsilon^{[l]}_0 e^{il(\theta+\delta^{-1} \alpha \hat F(\vu,\theta))}\right )
$$
where we have used that $\hat \Upsilon^{[l]}_0 = \Upsilon^{[l]}_0 e^{-il \alpha \coef^{-1}  L_0 \delta^{p+2}\log \delta}$.
As a consequence
\begin{align}\label{descompDelta1}
M&(\vu,\theta) -\Delta_0(\vu,\theta) = M(\vu,\theta)-M(f(\vu,\theta),\theta) \notag\\
&+  \cosh^{2/\coef}(\coef f(\vu,\theta))\Upsilon_0^{[0]} -\cosh^{2/\coef}(\coef\vu)(1+P_1(\vu,\theta))\Upsilon^{[0]}
\\ &+ \left (\sum_{l \neq 0} \Upsilon^{[l]}_0 e^{il(\theta+\delta^{-1} \alpha \hat F(\vu,\theta))}\right ) \left (\cosh^{2/\coef}(\coef f(\vu,\theta)) -
\cosh^{2/\coef}(\coef\vu)(1+P_1(\vu,\theta))\right )\notag
\end{align}
We shall prove bound~\eqref{boundDelta1rough} bounding each term in~\eqref{descompDelta1}, with $\vu=\vu_\pm$.

Recall that $M=r_{10}^{\uns}-r_{10}^{\sta}$ so that, by Theorem~\ref{thmoutloc} and~\eqref{boundfu}
\begin{equation}\label{boundMMf}
\begin{aligned}
|M(\vu_\pm,\theta) -M(f(\vu_\pm,\theta),\theta) | &\leq   K\delta^{p+3} |\vu_\pm-f(\vu_\pm,\theta)| \sup_{\vu\in\Doutinter} \frac{1}{|\cosh(\coef\vu)|^4} \\
&\leq
|\log \kappa|\frac{\delta^{2p+2}}{\kappa^{4}}
\end{aligned}
\end{equation}
so this term satisfies bound~\eqref{boundDelta1rough}.

By Theorem~\ref{thmktilde0}, the terms involving $\Upsilon_0^{[0]}$ and $\Upsilon^{[0]}$ in~\eqref{descompDelta1} are zero in the conservative case.
In the dissipative case, since we take $\param=\param_*(\delta)$ these terms satisfy:
$$
|\cosh^{2/\coef}(\coef\vu_\pm)(1+P_1(\vu_\pm,\theta))\Upsilon^{[0]}|\leq K \delta^{a_2}e^{-\frac{a_3\pi}{2\coef\delta}},
\qquad \left|\cosh^{2/\coef}(\coef\vu_\pm)\Upsilon_0^{[0]}\right| \leq K\delta^{p+4}
$$
where we have used that, from Theorem~\ref{thmdifpartsolinjective},
$|P_1(\vu_\pm,\theta)|\leq K\delta^{p+2}\dist^{-1}$.

Finally we deal with the last term which (see~\eqref{melnif}) we rewrite as
$$
\Sigma(\theta):=\left (M(f(\vu_\pm ,\theta),\theta) - \cosh^{2/\coef}(\coef \vu_{\pm})\Upsilon_{0}^{[0]} \right ) \left [
1- \frac{\cosh^{2/\coef}(\coef \vu_{\pm})}{\cosh^{2/\coef} (\coef f(\vu_\pm,\theta))} (1+P_1(\vu_{\pm}\theta))\right ].
$$
We first note that, since $M=r_{10}^{\uns}-r_{10}^{\sta}$, using Theorem~\ref{thmoutloc} to bound $r_{10}^{\uns,\sta}$, Theorem~\ref{thmktilde0} to bound $\Upsilon_0^{[0]}$
and~\eqref{boundMMf} one has that
$$
|\Sigma(\theta)|\leq K\frac{\delta^{p}}{\kappa^3} \left |
1- \frac{\cosh^{2/\coef}(\coef \vu_{\pm})}{\cosh^{2/\coef} (\coef f(\vu_\pm,\theta))} (1+P_1(\vu_{\pm}\theta))\right |.
$$
Therefore, using bound~\eqref{boundfu} of $|f(\vu_\pm,\theta)-\vu_\pm|$, one has that
$$
\left |\left [\frac{\cosh^{2/\coef}(\coef \vu_{\pm})}{\cosh^{2/\coef}(\coef f(\vu_\pm,\theta))} \right ] - 1 \right |  \leq K \frac{\delta^{p+2}|\log \dist|}{\dist}
$$
and, since $|P_1(\vu_\pm,\theta)|\leq K\delta^{p+2}\dist^{-1}$, we conclude that
$$
|\Sigma(\theta)| \leq K\frac{\delta^{2(p+1)}|\log \dist|}{\dist^4}
$$
and bound~\eqref{boundDelta1rough} for $\Delta_1(\vu_{\pm},\theta)$ is proven.

Finally we use bound~\eqref{boundDelta1rough} in~\eqref{boundUpsilons1-1} and the proposition is proven.

\section{Parameterizations of the invariant manifolds}\label{sec:param}
We will prove Theorem~\ref{thmoutloc} in Section~\ref{secproofthmoutloc}, as a non trivial consequence of the existence result Proposition~\ref{propout}.
This proposition is proven by using the fixed point theorem and its proof is given in Section~\ref{secproofthmout} below.

The complete proofs of the results in the present section are extremely technical and can be found with all the details in~\cite{CastejonPhDThesis}.
Here we present a summary of the methodology used in the proof.

We will use the notation and properties stated in Subsections~\ref{subsecpreliminary}, and~\ref{secthmoutloc}. Moreover, as usual, $\pi_{x},\pi_{y}, \pi_{z}$
will denote, respectively, the projection over the $x,y$ and $z$-component of a given vector.

\subsection{Existence of complex parameterizations}\label{secproofthmout}
As we explained in Section~\ref{secthmoutloc}, on the one hand the parameterizations $r^{\uns,\sta}(\vu,\theta)$ in~\eqref{defrunssta} can not be extended up to
the unbounded domains $\vu\in \Dout{\uns,\sta}$. On the other hand, the characterization of the unstable and stable manifolds is given for $\vu \to \pm \infty$.
To overcome this disagreement, we deal separately with the unstable or the stable manifold of the critical points $S_{-}(\delta,\param)$ or $S_{+}(\delta,\param)$ respectively, of the vector
field $X$ in~\eqref{initsys-outer2DX}.

\subsubsection{Setting and the existence of invariant manifolds result}\label{subsubsec:settingparam}
We perform two different linear changes of variables, $C^{\uns}$ and $C^{\sta}$, to the vector field $X$, such that they put
\begin{itemize}
\item the critical points $S_{\mp}(\delta,\param)$
at $(0,0,\mp 1)$ and
\item the linear part $DX(S_{\mp}(\delta,\param),\delta, \delta \param)$ in their Jordan form.
\end{itemize}
Here have taken the sign $-$ for the $-\uns-$ case and $+$ in the case $-\sta-$.

\begin{lemma}\label{lemchangCuCs}
Let $|\param|\leq \delta^{p+3}\param^*$. We write $\hat{S}_{\mp}=(0,0,\mp 1)$ and $\zeta =(x,y,z)$.
The two critical points $S_\mp(\delta,\param)=(x_\mp(\delta,\param),y_\mp(\delta,\param),z_\mp(\delta,\param))$ of
the vector field $X$ in~\eqref{initsys-outer2DX} are of the form:
$$
x_\mp(\delta,\param)=\mathcal{O}(\delta^{p+5}),\qquad y_\mp(\delta,\param)=\mathcal{O}(\delta^{p+5}),\qquad z_\mp(\delta,\param)=\pm 1+\mathcal{O}(\delta^{p+4}).
$$

There are two linear changes of variables $C^{\uns}$ and $C^{\sta}$ of the form
\begin{equation}\label{changeC1C2}
\hat{\zeta} = C^{\uns,\sta}(\zeta,\delta,\param) = M_{\mp}(\delta,\param) \big(\zeta +S_{\mp}(\delta , \param)\big) -\hat{S}_{\mp},
\end{equation}
where $M_{\mp}(\delta,\param)=\mathrm{Id}+\mathcal{O}(\delta^{p+5})$ (and therefore $C^{\uns,\sta} = \mathrm{Id}+\mathcal{O}(\delta^{p+4})$), such that
\begin{equation}\label{initsys_usX}
\frac{d {\hat{\zeta}}}{d t}=X^{\uns,\sta}(\delta \hat \zeta,\delta,\delta \param) = X_0(\hat \zeta,\delta,\param) + \delta^p X_1^{\uns,\sta}(\delta \hat \zeta,\delta,\delta\param),
\end{equation}
with $X_0$ the same as in~\eqref{initsys-outer2DX} and
\begin{enumerate}
\item the vector field $X_1^{\uns, \sta}(\delta \hat \zeta,\delta,\delta\param)=\mathcal{O}_3(\delta \hat \zeta,\delta ,\delta \param )$
and it is real analytic in $B^3(\hat r_0)\times B(\hat \delta_0)\times B(\hat \param_0)\subset \mathbb{C}^3 \times \mathbb{C}^2$.
\item $\hat{S}_{\mp}=(0,0,\mp 1)$ are critical points of $X^{\uns,\sta}$ respectively and
the linear part is in real Jordan form, that is
\begin{equation*}
\begin{aligned}
X_1^{\uns}(\delta \hat{S}_-,\delta, \delta \sigma)&=X_1^{\sta}(\delta \hat{S}_+,\delta,\delta\sigma)=0,\\
DX_1^{\uns,\sta}(\delta \hat{S}_{\mp},\delta, \delta \sigma) &= \left(\begin{array}{ccc}\mathcal{O}(\delta^{p+3}) & \mathcal{O}(\delta^{p+3}) & 0\\ \mathcal{O}(\delta^{p+3}) & \mathcal{O}(\delta^{p+3}) & 0 \\ 0 & 0 & \mathcal{O}(\delta^{p+3})\end{array}\right).
\end{aligned}
\end{equation*}
\end{enumerate}
\end{lemma}
\begin{proof} The proof of this result can be encountered in~\cite{CastejonPhDThesis} and uses that the vector field $X$ is written up to its
normal form of order three, see Remark~\ref{remarkFormanormal3}.
\end{proof}

Now we perform the symplectic cylindric change~\eqref{cylindriccoordsunpertub} to system~\eqref{initsys_usX}:
 \begin{align*}
 \frac{dr}{dt}&=2r(\param-\coef z)+\delta^p \Fb^{\uns,\sta} (\delta r,\theta,\delta z,\delta,\delta\param),\\
\frac{d\theta}{dt}&=-\frac{\alpha}{\delta}-cz+\delta^p \Gb^{\uns,\sta} (\delta r,\theta,\delta z,\delta,\delta\param),\\
\frac{dz}{dt}&=-1+2br+z^2+\delta^p \Hb^{\uns,\sta} (\delta r,\theta,\delta z,\delta,\delta\param),
 \end{align*}
where $\mathbf{X}^{\uns,\sta}_1=( \Fb^{\uns,\sta}, \Gb^{\uns,\sta}, \Hb^{\uns,\sta})$ is defined as
\begin{equation}\label{defFGHuns}
\mathbf{X}^{\uns,\sta}_1 (\delta r, \theta, \delta z,\delta,\delta \param) = \left (
\begin{array}{ccc} \sqrt{2r} \cos \theta & \sqrt{2r} \sin \theta & 0 \\
-\frac{1}{\sqrt{2r}} \sin \theta & \frac{1}{\sqrt{2r}} \cos \theta & 0 \\
0 & 0 & 1 \end{array}\right ) X_1^{\uns,\sta}(\delta \hat{\zeta},\delta,\delta\param)
\end{equation}
and $\hat{\zeta}$ denotes $\hat \zeta =(\sqrt{2r} \cos \theta, \sqrt{2r}\sin \theta, z)$.

The invariant manifolds associated to $\hat{S}_{\mp}=(0,0,\mp 1)$ will be parameterized as:
$$
r=R^{\uns,\sta} (\vv,\theta),\quad z=\hetz(\vv),\qquad \qquad\vv\in\mathbb{R},\,\theta\in\mathbb{S}^1
$$
with $R^{\uns,\sta} (\vv, \theta)\to 0$ as $v\to \mp \infty$ respectively. As the critical points are $\hat{S}_{\mp}$, one has
$$
(\sqrt{2 R^{\uns,\sta}(\vv,\theta)}\cos \theta, \sqrt{2 R^{\uns,\sta}(\vv,\theta)}\sin \theta,\hetz(\vv))\to \hat{S}_{\mp} \qquad \text{as  }\vv \to \mp \infty.
$$

We look for the parameterizations $R^{\uns,\sta}$ of the form $R^{\uns,\sta}=\hetr + R^{\uns,\sta}_1$.
We introduce the analogous notation to the one in~\eqref{notationFGHbis}:
\begin{equation}\label{notationFGH}
\begin{aligned}
\bar{X}^{\uns,\sta}_1(R)(\vv,\theta)&=\mathbf{X}^{\uns,\sta}_1 (\delta (\hetr(\vv)+R(\vv,\theta)),\theta,\delta\hetz(\vv),\delta,\delta\param),\\
\bar{X}_1^{\uns,\sta}&= (F^{\uns,\sta},G^{\uns,\sta},H^{\uns,\sta}).
\end{aligned}
\end{equation}
As we did in Section~\ref{secthmoutloc}, we will omit the dependence on $(\vv,\theta)$ if there is not danger of confusion.
One can easily check that $R_1^{\uns,\sta}$ have to satisfy the
invariance equation given by
$\Lout(R_1^{\uns,\sta} )=\Fout^{\uns,\sta} (R_1^{\uns,\sta} )$, where $\Lout$ is the linear differential operator in~\eqref{defopL}:
$$
 \Lout(R):=\left(-\delta^{-1}\alpha-c\hetz(\vv)\right)\partial_\theta R +\partial_\vv R -2\hetz(\vv) R
$$
and $\Fout^{\uns,\sta}$ are:
\begin{align}\label{defopFus}
\Fout^{\uns,\sta}(R):=&2\param(\hetr(\vv)+R)+\delta^pF^{\uns,\sta} (R)+\delta^p\frac{\coef+1}{b}\hetz(\vv)H^{\uns,\sta} (R)\nonumber\\
&-\delta^pG^{\uns,\sta} (R)\partial_\theta R-\left(\frac{2bR+\delta^pH^{\uns,\sta} (R)}{\coef(1-\hetz^2(\vv))}\right)\partial_\vv R.
\end{align}

The functions $R^{\uns,\sta}=\hetr+R_1^{\uns,\sta}$ lead to parameterizations of the invariant manifolds if $R_1^{\uns,\sta}$ satisfy respectively:
\begin{equation}\label{problemuns}
\Lout(R_1^\uns)=\Fout^\uns(R_1^\uns),\qquad \lim_{\vv\rightarrow-\infty}R_1^\uns(\vv,\theta)=0,
\end{equation}
\begin{equation}\label{problemsta}
\Lout(R_1^\sta)=\Fout^\sta(R_1^\sta),\qquad \lim_{\vv\rightarrow+\infty}R_1^\sta(\vv,\theta)=0.
\end{equation}

Problems~\eqref{problemuns} and~\eqref{problemsta} can be written as fixed point equations using suitable right inverses of the operator $\Lout$.
These right inverses can be found easily solving the ordinary differential equations satisfied by the Fourier coefficients $R^{[l]}(\vv)$ of any
periodic function in $\theta$, $R(\vv,\theta)$ that is a solution of $\Lout(R)=\phi$, for a given function $\phi$. Indeed, given $\phi(\vv,\theta)$,
we define:
$$
 \Gout^\us(\phi)(\vv,\theta):=\sum_{l\in\mathbb{Z}}{\Gout^\us}^{[l]}(\phi)(\vv)e^{il\theta}, \qquad  \us=\uns,\sta,
$$
with ${\Gout^\us}^{[l]}$ as:
\begin{equation}\label{defcoefsGunsintro}
{\Gout^\us}^{[l]}(\phi)(\vv) = \cosh^{\frac{2}{\coef}}(\coef \vv)\int_{\mp\infty}^0
\frac{e^{-il\delta^{-1}  (\pu^{\mp}(\vv +s)-\pu^{\mp}(\vv) )}}{\cosh^{\frac{2}{\coef}}(\coef (\vv+s))}\phi^{[l]}(\vv+s)ds, \qquad  \us=\uns,\sta,
\end{equation}
where, on the one hand, we take $-$ sign in the unstable case and $+$ in the stable one and on the other hand, we have used the notation $\pu^{\mp}$
introduced in Theorem~\ref{thmoutloc}:
$$
\pu^{\mp}(w) = \alpha w \mp \delta( cw \mp c \coef^{-1} \log(1+ e^{\pm 2 \coef w})).
$$
We stress that a compact expression for $\Gout^{\us}$ is given by:
\begin{equation}\label{defGuscompact}
\Gout^{\us}(\phi)(\vv,\theta) =  \cosh^{\frac{2}{\coef}}(\coef \vu) \int_{\mp \infty}^{u}
\frac{\phi\left (w,\theta-\delta^{-1} \big (\pu^{\mp}(w) - \pu^{\mp}(u)\big )\right)}{\cosh^{\frac{2}{\coef}}(\coef w)}dw,  \qquad  \us=\uns,\sta,
\end{equation}
\begin{remark}
Using that, if $w$ is real $\pu^{\mp} (w) = \alpha w + c \coef^{-1} \log (2 \cosh w)$ one obtains the more natural expression for the Fourier coefficients:
\begin{equation}\label{defcoefsGunsreal}
{\Gout^\us}^{[l]}(\phi)(\vv)=\cosh^{\frac{2}{\coef}}(\coef \vv)\int_{\mp\infty}^0
\frac{e^{-il\left(\delta^{-1}\alpha s+\frac{c}{\coef}\log\frac{\cosh(\coef(\vv+s))}{\cosh(\coef\vv)}\right)}}
{\cosh^{\frac{2}{\coef}}(\coef (\vv+s))}\phi^{[l]}(\vv+s)ds , \qquad  \us=\uns,\sta.
\end{equation}
However, expressions in~\eqref{defcoefsGunsreal} are not well defined when we take complex values of $\vv$. For this reason we take definitions
\eqref{defcoefsGunsintro}, which for real values of $\vv$ coincide with the ones in~\eqref{defcoefsGunsreal}, and are well defined when we take
$\vv\in \Dout{\uns,\sta}$.
\end{remark}

\begin{lemma}
One has $\Lout\circ\Gout^{\uns,\sta} =\textup{Id}$.
Moreover, if we define the operators:
$$
 \Foutfp^{\uns,\sta} :=\Gout^{\uns,\sta} \circ\Fout^{\uns,\sta} ,
$$
with $\Fout^{\uns,\sta}$ given in~\eqref{defopFus}, we have that if $R_1^{\uns,\sta}$ satisfy the fixed point equations:
\begin{equation}\label{fpeq}
 R_1^\uns=\Foutfp^\uns(R_1^\uns),\qquad  R_1^\sta=\Foutfp^\sta(R_1^\sta),
\end{equation}
then they are solutions of problems~\eqref{problemuns} and~\eqref{problemsta} respectively.
\end{lemma}

Since we will find solutions of the fixed point equation~\eqref{fpeq} by means of the fixed point theorem, we now
set the Banach spaces we work with. For the unstable case
we will consider functions $\phi:\Dout{\uns} \times\Tout\rightarrow\mathbb{C}$, where the domain $\Dout{\uns}$ is defined in~\eqref{defdoutuns}
and $\Tout$ is defined in~\eqref{defTout} (see also Figure~\ref{figureDout}). They can be written in their Fourier series:
$\phi(\vv,\theta)=\sum_{l\in\mathbb{Z}}\phi^{[l]}(\vv)e^{il\theta}$. We define the norms:
\begin{align*}
\left\|\phi^{[l]}\right\|_{n,m}^\uns& =\sup_{\vv\in\DoutT{\uns}}\left|\cosh^n(\coef\vv)\phi^{[l]}(\vv)\right|+\sup_{\vv\in\Doutinf{\uns}}\left|\cosh^m(\coef\vv)\phi^{[l]}(\vv)\right| \\
\|\phi\|_{n,m,\ost}^\uns&=\sum_{l\in\mathbb{Z}}\left\|\phi^{[l]}\right\|_{n,m}^\uns e^{|l|\ost}, \\
\llfloor \phi\rrfloor_{n,m,\ost}^\uns&=\|\phi\|_{n,m,\ost}^\uns+\|\partial_\vv \phi\|_{n+1,m,\ost}^\uns+\delta^{-1}\|\partial_\theta \phi\|_{n+1,m,\ost}^\uns.
\end{align*}
and we consider the Banach spaces:
\begin{align*}
\mathcal{X}^{\uns}_{n,m}&:=\left\{\phi:\Dout{\uns}\rightarrow\mathbb{C} \,:\, \phi\textrm{ is analytic and } \|\phi\|_{n,m}^\uns<+\infty\right\},\\
\Bsout{\uns}{n,m}&:=\left\{\phi:\Dout{\uns}\times\Tout\rightarrow\mathbb{C} \,:\, \phi\textrm{ is analytic and } \|\phi\|_{n,m,\ost}^\uns<+\infty\right\},\\
\Bsoutdiff{\uns}{n,m}&:=\left\{\phi:\Dout{\uns}\times\Tout\rightarrow\mathbb{C} \,:\, \phi\textrm{ is analytic and }
\llfloor \phi\rrfloor_{n,m,\ost}^\uns<+\infty\right\}.
\end{align*}
For functions $\Phi=(\phi_1,\phi_2,\phi_3)\in
\Bsout{\uns,\times}{n,m}:=\Bsout{\uns}{n,m}\times\Bsout{\uns}{n,m}\times\Bsout{\uns}{n,m}$, belonging to the product space,
we will take the norm:
\begin{equation*}
\|\Phi\|_{n,m,\ost}^{\uns,\times}:=\max\left\{\|\phi_1\|_{n,m,\ost}^\uns,\,\|\phi_2\|_{n,m,\ost}^\uns,\,\|\phi_3\|_{n,m,\ost}^\uns\right\}.
\end{equation*}
For the stable case, we consider norms and Banach spaces analogously defined in the corresponding domains $\Dout{\sta}$ and
for $\phi:\Tout\rightarrow\mathbb{C}$, we will take the Fourier norm:
$\|\phi\|_{\ost}:=\sum_{l\in\mathbb{Z}}\left|\phi^{[l]}\right|e^{|l|\ost}$.

Now we can state the result which guarantees the existence of solutions of the fixed point equations~\eqref{fpeq}.
We will devote the rest of the section to prove it.
During this section we will modify the value of the parameters $\dist$, $\beta$ $T$ and $\ost$ of the domains
$\Dout{\uns}\times\Tout$ and $\Dout{\sta}\times\Tout$ a finite number of times.
We will abuse notation and use the same letters for the modified values.
\begin{proposition}\label{propout}
Let $p\geq-2$ and $0<\beta<\pi/2$ be any constants. There exist $\param^*,\delta^*>0$ and $\dist^*\geq1$, such that
for all $0<\delta\leq\delta^*$ if $\dist=\dist(\delta)$ satisfies condition~\eqref{conddist} and $\param$ satisfies
$|\param|\leq\param^*\delta^{p+3}$, the fixed point equations in~\eqref{fpeq} have solutions $R_1^{\uns,\sta}$
defined respectively in $\Dout{\uns,\sta}\times\Tout$.

Moreover, they satisfy that $R_1^{\uns,\sta}=R_{10}^{\uns,\sta} +R_{11}^{\uns,\sta} $ with the following properties:
\begin{enumerate}
 \item $R_{10}^{\uns,\sta} =\Foutfp^{\uns,\sta}(0)\in\Bsoutdiff{\uns,\sta} {3,2}$ and there exists $M>0$:
$$\llfloor R_{10}^{\uns,\sta} \rrfloor_{3,2,\ost}^{\uns,\sta} \leq M\delta^{p+3}.$$
\item $R_{11}^{\uns,\sta} \in\Bsoutdiff{\uns,\sta} {4,2}$, and there exists a constant $M$ such that:
$$\llfloor R_{11}^{\uns,\sta} \rrfloor_{4,2,\ost}^{\uns,\sta} \leq M\delta^{p+3}\llfloor R_{10}^{\uns,\sta} \rrfloor_{3,2,\ost}^{\uns,\sta}.$$
\end{enumerate}
\end{proposition}
This result yields the following corollary:
\begin{corollary}\label{corpropout}
Under the same assumptions of Proposition~\ref{propout}, the two dimensional invariant manifolds of system~\eqref{initsys_usX} can be
parameterized, in symplectic polar coordinates as:
$$
r=R^{\uns,\sta}(\vv,\theta)=\hetr(\vv) + R^{\uns,\sta}_1(\vv,\theta),\qquad z=\hetz(\vv),\qquad (\vv,\theta) \in \Dout{\uns,\sta}\times \Tout,
$$
with $R^{\uns,\sta}_1$ satisfying the properties in Proposition~\ref{propout}.
\end{corollary}
In the following we will sketch the proof of Proposition~\ref{propout} in the unstable case.

\subsubsection{Solutions of the fixed point equation. Proof of Proposition~\ref{propout}}\label{Banachout}
Fist, in Lemma~\ref{lempropsopG}, we study the behaviour of the linear operator $\Gout^{\uns,\sta}$ acting on functions $\phi$ belonging to the Banach spaces
$\Bsout{\uns}{n,m}$. Secondly, in Lemma~\ref{lemfitaindepterm},
we deal with the independent term $\Foutfp^{\uns}(0)$. Finally, in Lemma~\ref{lemLipconst}, we check that the
operator $\Foutfp^{\uns,\sta}$ is a contraction.

\begin{lemma}\label{lempropsopG}
 Let $n\geq0$, $m\geq0$ and $\phi\in\Bsout{\uns}{n,m}$. There exists a constant $M$ such that for all $l\in\mathbb{Z}$:
\begin{enumerate}
\item If $n\geq1$, then $\|{\Gout^\uns}^{[l]}(\phi)\|_{n-1,m}^\uns\leq M\left\|\phi^{[l]}\right\|_{n,m}^\uns.$
\item  If $l\neq0$ and $n\geq0$, then $\displaystyle\|{\Gout^\uns}^{[l]}(\phi)\|_{n,m}^\uns\leq \frac{\delta M\left\|\phi^{[l]}\right\|_{n,m}^\uns}{|l|}.$
\item  As a consequence we have that if $n\geq1$, $\|\Gout^\uns(\phi)\|_{n-1,m,\ost}^\uns\leq M\|\phi\|_{n,m,\ost}^\uns$.
Moreover, if $\phi^{[0]}(\vv)=0$, then for all $n\geq0$:
$$\|\Gout^\uns(\phi)\|_{n,m,\ost}^\uns\leq M\delta\|\phi\|_{n,m,\ost}^\uns.$$
\item If $n\geq0$, $\|\partial_\theta\Gout^\uns(\phi)\|_{n,m,\ost}^\uns\leq M\delta\|\phi\|_{n,m,\ost}^\uns.$
\item If $n\geq1$, $\|\partial_\vv\Gout^\uns(\phi)\|_{n,m,\ost}^\uns\leq M\|\phi\|_{n,m,\ost}^\uns.$
\end{enumerate}
\item  In conclusion, if $\phi\in\Bsout{\uns}{n,m}$, $n\geq1$, then $\Gout^\uns(\phi)\in\Bsoutdiff{\uns}{n-1,m}$ and:
$$\llfloor \Gout^\uns(\phi)\rrfloor_{n-1,m,\ost}^\uns\leq M\|\phi\|_{n,m,\ost}^\uns.$$
\end{lemma}
\begin{proof}[Sketch of the proof]
The main idea to prove this result is to redefine adequately the Fourier coefficients ${\Gout^\uns}^{[l]}(\phi)$ changing the
path of integration. Take $\vv\in\Dout{\uns}$ fixed and consider $s=s_\pm(t,\vv)$ defined implicitly by (see Figure~\ref{figurecurve-s_pm}):
\begin{figure}
\center
 \includegraphics[width=6cm]{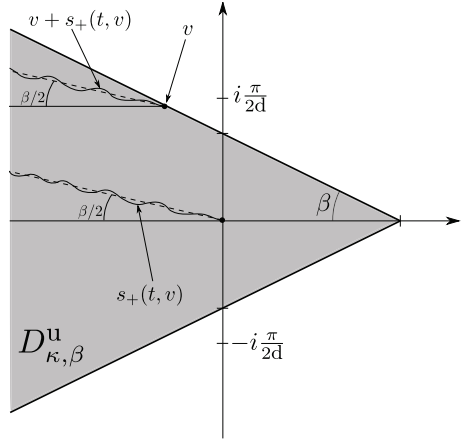}
 \caption[The domain $\Dout{\uns}$ with an example of the curves $s_+(t,\vv)$ and $v+s_+(t,\vv)$.]{The domain $\Dout{\uns}$ with an example of the curves $s_+(t,\vv)$ and $v+s_+(t,\vv)$. The discontinuous lines are $-te^{i\beta/2}$ and $v-te^{i\beta/2}$ respectively.}
 \label{figurecurve-s_pm}
\end{figure}
$$
s_\pm-\frac{c\delta}{\alpha}s_\pm+\frac{c\delta}{\coef\alpha}\left(\log\left(1+e^{2\coef(\vv+s_\pm)}\right)-
\log\left(1+e^{2\coef\vv}\right)\right)=-te^{\pm i\frac{\beta}{2}}.
$$
It can be proven that the function $s_\pm(t,\vv)$ is well-defined for all $t\in[0,+\infty)$ and $\vv\in\Dout{\uns}$ and moreover
that $\vv+s_\pm(t,\vv)\in\Dout{\uns}$. Consider the curve (see Figure~\ref{figurecurve-s_pm}):
$$\Gamma^R_\pm:=\left\{z\in\mathbb{C}\,:\,z=s_\pm(t,\vv),\, t\in[0,R]\right\}.$$
Then, one can prove that, if $m>0$ and $\phi\in\Bsout{\uns}{n,m}$ one has:
$$
{\Gout^\uns}^{[l]}(\phi)(\vv)=-\lim_{R\rightarrow+\infty}\cosh^{\frac{2}{\coef}}(\coef\vv)
\int_{\Gamma^R_\pm}\frac{e^{-il\delta^{-1} \left(\pu^{-}(\vv+z)- \pu^{-}(\vv)\right )}}{\cosh^{\frac{2}{\coef}}(\coef(\vv+z))}\phi^{[l]}(\vv+z)dz,
$$
where the coefficients ${\Gout^\uns}^{[l]}$ were defined in~\eqref{defcoefsGunsintro}, and we take the integral over $\Gamma_+^R$ for $l\geq0$ and over $\Gamma_-^R$ otherwise.

The proof of Lemma~\ref{lempropsopG} follows now from standard arguments.
\end{proof}

Now we are going to bound the independent term of the fixed point equation~\eqref{fpeq} which is
$\Foutfp^\uns(0)=\Gout^\uns\circ\Fout^\uns(0)$ (see~\eqref{defopFus}) with
$$
\Fout^\uns(0)=2\param\hetr(\vv)+\delta^p F^{\uns}(0) +\delta^p\frac{\coef+1}{b}\hetz(\vv) H^{\uns}(0).
$$

\begin{lemma}\label{lemfitaHiF}
Let $C_R$ be some constant, and $R$ such that $\|R\|_{2,2,\ost}\leq C_R$. There exists a constant $M$ such that:
$$
\|F^{\uns}(R-R_0)\|_{4,2,\ost}^\uns, \;\; \|G^{\uns}(R-R_0)\|_{2,0,\ost}^\uns, \;\; \|H^{\uns}(R-R_0)\|_{3,2,\ost}^\uns \leq M\delta^3,
$$
with $F^{\uns},G^{\uns},H^{\uns}$ defined in~\eqref{notationFGH}.
In particular, this holds for $R=\hetr$.
\end{lemma}
\begin{proof}
We will use the properties of $X_1^{\uns}:=(f^{\uns},g^{\uns},h^{\uns})$ stated in Lemma~\ref{lemchangCuCs}. We first prove
the bound for $F^{\uns}$ being the one for $G^{\uns}$ analogous. By definitions~\eqref{notationFGH} and~\eqref{defFGHuns}
\begin{align*}
F^{\uns}&(R-R_0)(\vv,\theta)=\Fb^\uns(\delta R(\vv,\theta),\theta,\delta\hetz(\vv),\delta,\delta\param)  \\
&=\sqrt{2R(\vv,\theta)}\cos\theta f^\uns(\Phi(\vv,\theta),\delta,\delta \param)+
\sqrt{2R(\vv,\theta)}\sin\theta g^\uns(\Phi(\vv,\theta),\delta,\delta \param).
\end{align*}
with
$$
\Phi(\vv,\theta)=\big (\delta\sqrt{2R(\vv,\theta)}\cos\theta,\delta\sqrt{2R(\vv,\theta)}\sin\theta,\delta\hetz(\vv)\big).
$$
By Lemma~\ref{lemchangCuCs}, $f^{\uns}$ is of order three in all their variables and $f^\uns(0,0,-\delta,\delta,\delta\param)=0$ for all $\delta$ and $\param$.
Therefore, since $\|\sqrt{2R}\|_{1,1,\ost}<+\infty$ and $\hetz(\vu) = \tanh(\coef \vu)$, we have that
$$
\|f^\uns(\Phi(\vv,\theta),\delta,\delta\param)\|_{3,1,\ost}^\uns \leq K\delta^{3}.
$$
Reasoning analogously, we obtain the same bound for $g^{\uns}$ and thus:
$$
\|F^\uns(R-R_0)\|_{4,2,\ost}^\uns\leq K\delta^3.
$$

With respect to $H^{\uns}$, we have:
$$H^{\uns}(R-R_0)(\vv,\theta)=\Hb^\uns(\delta R(\vv,\theta),\theta,\delta\hetz(\vv),\delta,\delta\param)=
h^\uns(\Phi(\vv,\theta),\delta,\delta\param).$$
Again, $h^\uns$ is of order three in all their variables, $h^\uns(0,0,-\delta,\delta,\delta\param)=0$ and moreover $\partial_xh^\uns(0,0,-\delta,\delta,\delta\param)=\partial_yh^\uns(0,0,-\delta,\delta,\delta\param)=0$. The bound for $H^{\uns}(R-R_0)$ follows by using the same arguments as above.
\end{proof}

Now we deal with the independent term $\Foutfp^\uns(0)$.
\begin{lemma}\label{lemfitaindepterm}
Let $|\param|\leq\param^*\delta^{p+3}$. There exists $M>0$ such that $\llfloor \Foutfp^\uns(0)\rrfloor_{3,2,\ost}^\uns\leq M\delta^{p+3}$.
\end{lemma}
\begin{proof}
By Lemma~\ref{lempropsopG} it is enough to prove that
$\|\Fout^\uns(0)\|_{4,2,\ost}^\uns\leq M\delta^{p+3}$ (recall that $\Foutfp^\uns=\Gout^\uns\circ\Fout^\uns$).
This is clear, by Lemma~\ref{lemfitaHiF}:
\begin{align*}
 \|\Fout^\uns(0)\|_{4,2,\ost}^\uns&\leq \|2\param\hetr\|_{4,2,\ost}^\uns+\delta^{p} \|F^{\uns}(0) \|_{4,2,\ost}^\uns+
\delta^p\frac{\coef+1}{b}\|\hetz \cdot H^\uns(0)\|_{4,2,\ost}^\uns \\
&\leq K\left (\param + \delta^{p} \|F^{\uns}(0) \|_{4,2,\ost}^\uns + \delta^p\|\hetz\|_{1,0}^\uns\|H^\uns(0)\|_{3,2,\ost}^\uns \right )\\
&\leq K(\param+\delta^{p+3})\leq K\delta^{p+3}.
\end{align*}
\end{proof}

We prove now that the operator $\Foutfp^{\uns}$ is contractive in an appropriate Banach space.
More precisely we prove the following result:
\begin{lemma}\label{lemLipconst}
Let $|\param|\leq\param^*\delta^{p+3}$. Assume that $\phi_1$, $\phi_2\in\Bsoutdiff{\uns}{3,2}$ satisfy for $C>0$ and $i=1,2$ that
$\llfloor\phi_i\rrfloor_{3,2,\ost}^\uns\leq C\delta^{p+3}$. Then there exists $M>0$ such that:
$$\llfloor\Foutfp^\uns(\phi_1)-\Foutfp^\uns(\phi_2)\rrfloor_{4,2,\ost}^\uns\leq M\delta^{p+3}\llfloor\phi_1-\phi_2\rrfloor_{3,2,\ost}^\uns.$$
\end{lemma}
\begin{proof}
We skip tedious computations and only give an sketch of the proof.
See details in~\cite{CastejonPhDThesis}.
Using that $\Gout^\uns$ is linear and Lemma~\ref{lempropsopG}:
$$
\llfloor\Foutfp^\uns(\phi_1)-\Foutfp^\uns(\phi_2)\rrfloor_{4,2,\ost}^\uns
\leq M \| \Fout^{\uns}(\phi_1) -\Fout^{\uns}(\phi_2))\|_{5,2,\ost}^{\uns}.
$$
It is only necessary to prove that if $\llfloor\phi_i\rrfloor_{3,2,\ost}^\uns\leq C\delta^{p+3}$, then:
\begin{equation}\label{LipboundFoutuns}
\|\Fout^\uns(\phi_1)-\Fout^\uns(\phi_2)\|_{5,2,\ost}^\uns\leq K\delta^{p+3}\llfloor\phi_1-\phi_2\rrfloor_{3,2,\ost}^\uns.
\end{equation}
We decompose the operator $\Fout^{\uns}$ in~\eqref{defopFus} into $\Fout^{\uns}=\Fout^{\uns}_1+\Fout^{\uns}_2+\Fout^{\uns}_3+\Fout^{\uns}_4$ with
\begin{align*}
\Fout^{\uns}_1(\phi)&=2\param (R_0(\vv)+\phi) + \delta^p F^\uns(\phi) + \delta^p \frac{\coef +1}{b} \hetz(\vv) H^{\uns}(\phi) \\
\Fout^{\uns}_2(\phi)&= -\delta^p G^{\uns}(\phi) \partial_{\theta} \phi  \\
\Fout^{\uns}_3(\phi)&= -\delta^p \frac{1}{\coef (1-\hetz^2(\vv))} H^{\uns}(\phi) \partial_{\vv} \phi \\
\Fout^{\uns}_4(\phi)&= \frac{2b}{\coef (1-\hetz^2(\vv))} \phi \partial_{\vv}\phi .
\end{align*}
Proceeding as in the proof of Lemma~\ref{lemfitaHiF}, one can prove (after tedious but easy computations):
\begin{align}
&\| F^{\uns}(\phi_1)-F^\uns(\phi_2)\|_{5,2,\ost}^\uns\leq M\delta^3\|\phi_1-\phi_2\|_{3,2,\ost}^\uns. \label{bounddifGHu}\\
&\|G^{\uns}(\phi_1) - G^{\uns}(\phi_2) \|_{3,0,\ost},  \|H^{\uns}(\phi_1) - H^{\uns}(\phi_2) \|_{4,2,\ost}\leq M \delta^{3} \| \phi_1 - \phi_2\|_{3,2,\ost}. \notag
\end{align}
As a consequence, the operator $\Fout^{\uns}_1$ satisfies the bound in~\eqref{LipboundFoutuns} provided $\|\hetz\|_{1,0}\leq K$.

With respect to $\Fout^{\uns}_2$, we write:
$$
\Fout^{\uns}_2(\phi_1)-\Fout^{\uns}_2(\phi_2) = \big (G^{\uns}(\phi_1)-G^{\uns}(\phi_2)\big ) \partial_{\theta}\phi_1 + G^{\uns}(\phi_2)
\partial_{\theta}\big (\phi_{1}-\phi_2).
$$
Then, since $\phi_i \in \Bsoutdiff{\uns}{3,2}$,
\begin{equation}\label{Fout2rewrite}
\begin{aligned}
\|\Fout^{\uns}_2(\phi_1)-\Fout^{\uns}_2(\phi_2)\|_{5,2,\ost}^{\uns}\leq&  \|G^{\uns}(\phi_1)-G^{\uns}(\phi_2)\|_{3,0,\uns}^{\uns} \|\partial_{\theta}\phi_1\|_{2,2,\ost}^{\uns}
\\&+ \|G^{\uns}(\phi_2) \|_{2,0,\ost}^{\uns} \|\partial_{\theta}\big (\phi_{1}-\phi_2) \|_{3,2,\ost}^{\uns}.
\end{aligned}
\end{equation}
Now we note that:
\begin{equation}\label{boundlipcon12}
 \|\partial_\theta {\phi}_1\|_{2,2,\ost}^\uns\leq
K\frac{\left\|\partial_\theta {\phi}_1\right\|_{4,2,\ost}^\uns}{\delta^2\dist^2}
\leq K\frac{\llfloor {\phi}_1\rrfloor_{3,2,\ost}^\uns}{\delta\dist^2}\leq
K\frac{\delta^{p+2}}{\dist^2}\leq K.
\end{equation}
and:
\begin{equation}\label{boundlipcon15}
 \|\partial_\theta({\phi}_1-{\phi}_2)\|_{3,2,\ost}^\uns\leq
\frac{K}{\delta\dist}\|\partial_\theta({\phi}_1-{\phi}_2)\|_{4,2,\ost}^\uns\leq
\frac{K}{\dist}\llfloor {\phi}_1-{\phi}_2\rrfloor_{3,2,\ost}^\uns.
\end{equation}
Note that since $\| \phi_i\|_{3,2,\ost}^{\uns} \leq C \delta^{p+3}$, then $\| \phi_i \|_{2,2,\ost}^{\uns}\leq C \kappa^{-1} \delta^{p+2}$.
Therefore, Lemma~\ref{lemfitaHiF} with $R=\phi_2$ applies.
Thus, using this lemma and the bounds in~\eqref{boundlipcon12}, \eqref{boundlipcon15} and~\eqref{bounddifGHu} in
inequality~\eqref{Fout2rewrite}, we also obtain that the operator $\Fout^{\uns}_2$ satisfies bound in~\eqref{LipboundFoutuns}.

We leave to the reader to check that $\Fout^{\uns}_3$ and $\Fout^{\uns}_4$ also satisfy bound~\eqref{LipboundFoutuns} provided
$\|(1-Z_0^{2})^{-1}\|_{-2,-2,\ost}^{\uns} \leq K$.
\end{proof}

\begin{proof}[End of proof of Proposition~\ref{propout}]
Proposition~\ref{propout} is a corollary of Lemmas~\ref{lemfitaindepterm} and~\ref{lemLipconst}. Indeed,
let $p\geq-2$ and $|\param|\leq\param^*\delta^{p+3}$. Define $\varrho:=2\llfloor\Foutfp^\uns(0)\rrfloor_{3,2,\ost}^\uns$ and
$B(\varrho)\subset\Bsoutdiff{\uns}{3,2}$ the ball of radius $\varrho$ centered at zero.

We claim that $\Foutfp^\uns$ has a unique fixed point in $B(\varrho)$. Indeed, we point out that $\varrho\leq K\delta^{p+3}$
by Lemma~\ref{lemfitaindepterm}. We first check that $\Foutfp^\uns$ is contractive. By the properties of the norm
$\llfloor.\rrfloor_{n,m,\ost}^\uns$ and Lemma~\ref{lemLipconst}, for $\phi_1,\phi_2\in B(\varrho)$:
$$
\llfloor\Foutfp^\uns(\phi_1)-\Foutfp^\uns(\phi_2)\rrfloor_{3,2,\ost}^\uns\leq
\frac{K}{\delta\dist}\llfloor\Foutfp^\uns(\phi_1)-\Foutfp^\uns(\phi_2)\rrfloor_{4,2,\ost}^\uns
\leq\frac{K\delta^{p+2}}{\dist}\llfloor\phi_1-\phi_2\rrfloor_{3,2,\ost}^\uns.
$$
Clearly, since $p\geq-2$ and $\dist^*$ is large enough, $\Foutfp^\uns$ is contractive in $B(\varrho)$.

It remains to check that $\Foutfp^\uns:B(\varrho)\rightarrow B(\varrho)$.
If $\phi\in B(\varrho)$, by Lemma~\ref{lemLipconst}:
$$
\llfloor\Foutfp^\uns(\phi)\rrfloor_{3,2,\ost}^\uns\leq
\llfloor\Foutfp^\uns(\phi)-\Foutfp^\uns(0)\rrfloor_{3,2,\ost}^\uns+\llfloor\Foutfp^\uns(0)\rrfloor_{3,2,\ost}^\uns
\leq K\frac{\delta^{p+2}}{\dist}\llfloor\phi\rrfloor_{3,2,\ost}^\uns+\frac{1}{2}\varrho.
$$
Taking $\kappa^* \leq  \kappa$ big enough, $\llfloor\Foutfp^\uns(\phi)\rrfloor_{3,2,\ost}^\uns<\varrho$.
That is, $\Foutfp^\uns:B(\varrho)\rightarrow B(\varrho)$. Therefore, by the fixed point theorem $\Foutfp$ has a unique fixed point in $B(\varrho)$.

It is clear that $R_1^\uns\in B(\varrho)$ is the fixed point of $\Foutfp^\uns$ obtained before.
Then, item 1 of Proposition~\ref{propout} is a direct consequence of Lemma~\ref{lemfitaindepterm}.
To prove item 2, we just need to note that:
$
R_{11}^\uns=R_1^\uns-R_{10}^\uns=\Foutfp^\uns(R_1^\uns)-\Foutfp^\uns(0).
$
Using Lemma~\ref{lemLipconst} and that $\llfloor R_1^\uns\rrfloor_{3,2,\ost}^\uns \leq  K\llfloor R_{10}^\uns\rrfloor_{3,2,\ost}^\uns$,
we obtain that:
$$
\llfloor R_{11}^\uns\rrfloor_{4,2,\ost}^\uns\leq K\delta^{p+3}\llfloor R_1^\uns\rrfloor_{3,2,\ost}^\uns \leq K \delta^{p+3} \llfloor R_{10}^\uns\rrfloor_{3,2,\ost}^\uns.
$$
\end{proof}

\subsection{Suitable complex parameterizations. Theorem~\ref{thmoutloc}}\label{secproofthmoutloc}

In this section we shall prove Theorem~\ref{thmoutloc} concerning the functions $r_1^\uns$ and $r_1^\sta$.
The fact that $R_1^{\uns}$ and $R_1^\sta$ satisfy different equations is not adequate for our purposes of comparing them.
We will now proceed to obtain new parameterizations $r^{\uns,\sta}=\hetr+r^{\uns,\sta}_1$ of the invariant manifolds which
will be solutions of the same functional equation. To obtain such a parameterizations we i) undo the changes~\eqref{changeC1C2} until we get a parameterization of
system~\eqref{initsys-outer2D} and ii) perform the symplectic polar change of coordinates.

The technical proofs can be encountered in~\cite{CastejonPhDThesis}.

\subsubsection{Setting}
Let $R_1^{\uns,\sta}$ be the functions given
by Proposition~\ref{propout}. We consider
$R^{\uns,\sta}=R_0+R_1^{\uns,\sta}$ and we introduce the parameterizations of the invariant manifolds of the equilibrium points $\hat{S}_\mp=(0,0,\mp1)$
of the vector field $X^{\uns,\sta}$ in~\eqref{initsys_usX}:
\begin{equation}\label{defXYuns}
\hat \zeta^{\uns,\sta}(\vv,\theta) := \big (\sqrt{2R^{\uns,\sta}(\vv,\theta)}\cos\theta, \sqrt{2R^{\uns,\sta}(\vv,\theta)}\sin\theta,\hetz(\vv)\big ).
\end{equation}
We define:
\begin{equation}\label{defxyzunvv}
{\zeta}^{\uns,\sta}(\vv,\theta):=(C^{\uns,\sta})^{-1}\hat\zeta^{\uns,\sta}(\vv,\theta)= (x^{\uns,\sta}(\vv,\theta),y^{\uns,\sta}(\vv,\theta),z^{\uns,\sta}(\vv,\theta))
\end{equation}
where $C^{\uns,\sta}$ are given in~\eqref{changeC1C2}.
These are parameterizations of the two dimensional unstable (respectively stable) manifold associated to the equilibrium points $S_\mp(\delta,\param)$ of the
original system~\eqref{initsys-outer2D}.

To compare $(x^\uns(\vv,\theta),y^\uns(\vv,\theta))$ and $(x^\sta(\vv,\theta),y^\sta(\vv,\theta))$ on the $z-$plane (or equivalently in the $\vu-$plane given
by $z=\hetz(\vu)$), we
implicitly define the functions $\vv^{\uns,\sta}$ as:
$$
\hetz(\vu)=z^\uns(\vv^{\uns}(\vu,\theta),\theta),\qquad \hetz(\vu)=z^\sta(\vv^{\sta}(\vu,\theta),\theta).
$$
The result about the existence of functions $\vv^{\uns,\sta}$ is given below. Its proof is an elementary application of the
fixed point theorem.
\begin{lemma}\label{lemfunv}
Let $\dist=\dist(\delta)$ given in Proposition~\ref{propout}.
Fix $m>0$ a constant independent of $\delta$ and $\param$. Let $\bar\dist=\bar\dist(\delta)$
satisfying condition~\eqref{conddist} and such that $\bar\dist>\dist+m$, and let $z^{\uns,\sta}(\vv,\theta)$ be the functions defined in~\eqref{defxyzunvv}
for $(v,\theta) \in \DoutT[\dist][T]{\uns,\sta}\times \Tout$. Let $\bar T$ be a constant such that $0<\bar T <T$. Then,
if $\delta$ is sufficiently small, the functions $\vv^{\uns,\sta}$ defined implicitly by:
$$
\hetz(\vu)=z^{\uns,\sta}(\vv^{\uns,\sta}(\vu,\theta),\theta)
$$
are well defined for all $\vu\in\DoutT[\bar\dist][\bar T]{\uns,\sta}$ and $\theta\in\Tout$, and there exists a constant $M$ such that:
$$
|\vv^{\uns,\sta}(\vu,\theta)-\vu|\leq M\delta^{p+4}|\cosh(\coef\vu)|^2.
$$
\end{lemma}

We take $m>0$ fixed and $\dist, \bar{\dist}, T, \bar{T}$ and ${\beta}$ as in Lemma~\ref{lemfunv}.

Note that, as $\zeta^{\uns,\sta}(\vv,\theta)$ in~\eqref{defxyzunvv}, the functions $\zeta^{\uns,\sta}(\vv^{\uns,\sta}(\vu,\theta),\theta)$
are other parameterizations of the unstable and stable manifolds of $S_\mp(\delta,\param)$ respectively.
We define $r^{\uns,\sta}(\vu,\theta)$ as:
\begin{equation}\label{defruslocal}
r^{\uns,\sta}(\vu,\theta)=\frac{1}{2}\left[(x^{\uns,\sta}(\vv^{\uns,\sta}(\vu,\theta),\theta))^2+(y^{\uns,\sta}(\vv^{\uns,\sta}(\vu,\theta),\theta))^2\right].
\end{equation}

We claim that there exists $K>0$ such that for all $(\vu,\theta) \in \DoutT[\bar\dist][\bar T]{\uns,\sta}\times \Tout$,
\begin{equation}\label{exprrnew}
r^{\uns,\sta}(\vu,\theta) =
\hetr(\vu) + R^{\uns,\sta}_1(\vu,\theta) + r_{2}^{\uns,\sta}(\vu,\theta),\qquad |\cosh (\coef \vu) r_2^{\uns,\sta}(\vu,\theta)|\leq K\delta^{p+4},
\end{equation}
where $R^{\uns,\sta}_1$ are given in Proposition~\ref{propout}.
We deal only with the unstable case being the stable one analogous.
We first begin by study the difference between $\pi_{x,y}\hat \zeta^{\uns}$ and $\pi_{x,y}{\zeta}^{\uns}$ defined respectively
in~\eqref{defXYuns} and~\eqref{defxyzunvv}.
We note that, from Lemma~\ref{lemchangCuCs}, $\pi_{x,y}S_{-}(\delta,\param) = \mathcal{O}(\delta^{p+5})$
and the matrices $M_{-}(\delta,\param)$ in the same lemma have inverse of the form
$M_{-}^{-1}(\delta,\param)=\text{Id} + \delta^{p+5} \hat{M}_{-}^{-1}(\delta,\param)$ having $\hat{M}_-^{-1}(\delta,\param)$
bounded entries. We denote by $\pi_{x,y} \hat{M}_-^{-1}$ the two first rows of $\hat{M}_-^{-1}$. Then, using the form of the change $C^{\uns}$ in~\eqref{changeC1C2}:
$$
\pi_{x,y} {\zeta}^{\uns}(\vv,\theta)= \pi_{x,y} \hat \zeta^{\uns}(\vv,\theta) + k \delta^{p+5} + \delta^{p+5}\pi_{x,y} \hat{M}_-^{-1}\hat \zeta^{\uns}(\vv,\theta)
$$
for some bounded coefficients $k:=k(\delta,\param)$. By Corollary~\ref{corpropout}:
\begin{equation}\label{boundsXY}
|\hat \zeta^\uns(\vv,\theta)|\leq \frac{K}{|\cosh(\coef\vv)|}, \qquad (\vv,\theta)\in\DoutT{\uns}\times\Tout
\end{equation}
and therefore, if $(\vv,\theta)\in\DoutT{\uns}\times\Tout$:
\begin{equation}\label{expxunsyuns}
\pi_{x,y}{\zeta}^\uns(\vv,\theta)=\pi_{x,y}\hat \zeta^\uns(\vv,\theta)+\mathcal{O}(\delta^{p+4}).
\end{equation}
By definition~\eqref{defruslocal} of $r^{\uns,\sta}$ we are interested in computing ${\zeta}^{\uns}(\vv^{\uns}(u,\theta),\theta)$. For that
we also need to study the difference between $\hat \zeta^{\uns}(u,\theta)$ and $\hat \zeta^{\uns}(\vv^{\uns}(u,\theta),\theta)$.
We emphasize that by Lemma~\ref{lemfunv}, taking $\delta$ sufficiently small we can ensure
that
$\vv_{\lambda}^{\uns}(\vu,\theta):=\vv^{\uns}(\vu,\theta)+\lambda(\vv^{\uns}(\vu,\theta)-\vu)\in\DoutT{\uns}$
for $\vu\in\DoutT[\bar\dist][\bar T]{\uns}$ and $\lambda\in[0,1]$.
Then using Proposition~\ref{propout}, Lemma~\ref{lemfunv} and the mean value theorem:
\begin{align*}
|\pi_{x,y}\hat \zeta^\uns(\vv^{\uns}(\vu,\theta),\theta)&-\pi_{x,y}\hat \zeta^\uns(\vu,\theta)|\leq
\sup_{\lambda\in[0,1]}|\partial_\vv \pi_{x,y}\hat\zeta (\vv_{\lambda}^{\uns}(u,\theta))||\vv^{\uns}(\vu,\theta)-\vu| \\
& \leq K\delta^{p+4}\sup_{\lambda\in[0,1]}\frac{|\cosh(\coef\vu)|^2}{|\cosh(\coef(\vv_{\lambda}^{\uns}(\vu,\theta)))|^2}
\leq K\delta^{p+4}.
\end{align*}
Using this expression in~\eqref{expxunsyuns} one obtains:
$$
\pi_{x,y}{\zeta}^\uns(\vv^{\uns}(\vu,\theta),\theta)=\pi_{x,y}\hat \zeta^\uns(\vu,\theta)+\mathcal{O}(\delta^{p+4}).
$$
Then, recalling that $\pi_{x,y}{\zeta}^{\uns} = (x^{\uns},y^{\uns})$ and using~\eqref{boundsXY}, we obtain:
\begin{eqnarray*}
r^\uns(\vu,\theta)&=&\frac{1}{2}\left[(\pi_x \hat \zeta^\uns(\vu,\theta))^2+(\pi_y\hat \zeta^\uns(\vu,\theta))^2\right]+
\mathcal{O}\left(\frac{\delta^{p+4}}{\cosh(\coef\vu)}\right)\\
&=&R_0(\vu)+R_1^\uns(\vu,\theta)+\mathcal{O}\left(\frac{\delta^{p+4}}{\cosh(\coef\vu)}\right)
\end{eqnarray*}
and~\eqref{exprrnew} is proven.
We introduce $r_{1}^{\uns} = R_1^{\uns} + r_2^{\uns}$ and therefore, $r^{\uns}$ is of the form
$r^{\uns}=\hetr + \r^{\uns}_1$.
Note that, by construction, $r_1^\uns$ satisfies the partial differential equation~\eqref{PDEequal}.

In addition, by the compact expression of $\Gout^{\uns}$ in~\eqref{defGuscompact}, the dominant part, $r_{10}^{\uns}$, of $r_{1}^{\uns}$, given
in Theorem~\ref{thmoutloc} is $r_{10}^{\uns}=\Gout^{\uns}\circ\Fout(0)$,
where $\Fout$ is the operator defined in~\eqref{defopF}. Then, by using the expression for $R_1^{\uns}$ in Proposition~\ref{propout}, we obtain
the decomposition:
\begin{equation}\label{exprr1}
r_1^\uns=r_{10}^{\uns}+r_{11}^{\uns}, \qquad r_{10}^{\uns} = \Gout^{\uns}\circ\Fout(0),\qquad r_{11}^{\uns}=\Gout^\uns\circ(\Fout^\uns(0)-\Fout(0))+R_{11}^\uns+r_2^\uns,
\end{equation}
where we have used that the operator $\Gout^{\uns}$ is linear.

\subsubsection{End of the proof of Theorem~\ref{thmoutloc}}
It remains to prove the bounds for $r_{10}^{\uns}$ and $r_{11}^{\uns}$ in Theorem~\ref{thmoutloc}
on $\DoutT[\bar\dist][\bar T] [\bar \beta]{\uns,\sta}\times \Tout[\bar \ost]$. For this, it is convenient to define the auxiliary norms for functions $\phi:\DoutT{\uns,\sta}\times \Tout \to \mathbb{C}$:
$$
\|\phi\|_{n,\ost}^{\dist,\beta,T}=\sup_{\substack{\theta\in\Tout\\\vv\in\DoutT{\uns}}}\left|\cosh^n(\coef\vv)\phi(\vv,\theta)\right|
$$
which satisfies $\|\phi \|_{n,\ost}^{\dist,\beta,T} \leq  \| \phi \|_{n,0,\ost}^{\uns}$ if $\phi \in \Bsout{\uns}{n,0}$. Moreover, if $m>0$,
then
\begin{equation}\label{boundpartialuphi}
\|\partial_\vu\phi\|_{n+1,\bar\ost}^{\bar \dist, \bar T, \bar \beta}\leq M\|\phi\|_{n,\ost}^{\dist, T, \beta}
\end{equation}
with $\bar \dist >\dist+m$ satisfying condition~\eqref{conddist}, $0<\bar \beta <\beta$, $0<\bar T <T$ and $0<\bar \ost<\ost$. This fact can be checked as
Lemma 4.3 in~\cite{BaldomaInner}.

Along this proof we will use that, if $\vu \in \Dout{\uns}$ (see~\eqref{defdoutuns}), then  $|\cosh (\coef \vu)| \geq K \delta$ for some constant $K$.
We also use decomposition~\eqref{exprr1} without mention.

It can be straightforwardly proven (see~\cite{CastejonPhDThesis}) that there exists $M>0$, such that
$$
\left\|\Fout(0)\right\|_{4,0,\ost}^{\uns,\sta}\leq M\delta^{p+3},\qquad
\left\|\Fout(0)-\Fout^{\uns,\sta}(0)\right\|_{2,0,\ost}^{\uns,\sta}\leq M\delta^{p+4}.
$$
Therefore, using the properties of $\Gout^\uns$
in Lemma~\ref{lempropsopG} we obtain
$$
\llfloor \Gout^\uns\circ\Fout(0)\rrfloor_{3,0,\ost}^\uns\leq K\delta^{p+3},\qquad
\llfloor \Gout (\Fout(0)-\Fout^{\uns,\sta}(0))\rrfloor_{1,0,\ost}^{\uns}\leq K\delta^{p+4},
$$
so that the bounds for $r_{10}^{\uns}$ are done using that
$\| \cdot \|_{n,\ost}^{\dist,\beta,T} \leq  \|\cdot \|_{n,0,\ost}^{\uns}$. In addition,
with respect to $R_{11}^{\uns}$, using Proposition~\ref{propout}, we have
$ \llfloor R_{11}^\uns\rrfloor_{4,2,\ost}^\uns\leq K\delta^{2p+6}$ and with respect to
$r_2^{\uns}$ defined in~\eqref{exprrnew},
$\|r_2^\uns\|_{1,\ost}^{\bar\dist, \beta, \bar T}\leq K\delta^{p+4}$.
Therefore, using also property~\eqref{boundpartialuphi}, we have that
$$
\|R_{11}\|_{4,\ost}^{\dist,T,\beta} , \|\partial_\vu R_{11}\|_{5,\ost}^{\dist, T, \beta}\leq  K \delta^{2p+6}\qquad \|r_2\|_{1,\ost}^{\bar\dist ,\beta, \bar T} ,
\|\partial_u r_2\|_{1,\ost}^{\bar\dist ,\bar \beta, \bar T} \leq K \delta^{p+4}.
$$
We point out that we have abused notation, using the same $\bar\dist$ and $\bar T$ although they are different from the previous ones.
However, they still satisfy $\bar\dist-\dist>m$ and $0<\bar T<T$.
Now we are almost done. Notice that by definition of $r_{11}^{\uns}$
$$
|r_{11}^{\uns}(\vu,\theta)| \leq
|\cosh (\coef \vu) | \big (\llfloor \Gout (\Fout(0)-\Fout^{\uns,\sta}(0))\rrfloor_{1,0,\ost}^{\uns}+\|r_2\|_{1,\ost}^{\bar\dist ,\beta, \bar T} \big ) +
|\cosh (\coef \vu) |^4 \llfloor R_{11}^{\uns}\rrfloor_{4,2,\ost}^\uns
$$
and then, using the above bounds, the statement for $r_{11}^{\uns}$ in Theorem~\ref{thmoutloc} is checked.
The bound  $\|\partial_u r_{1}^{\uns}\|_{4,\bar \ost}^{\bar \dist,\bar \beta,\bar T} \leq K\delta^{p+3}$
is straightforward from the above bounds for $\partial_u r_{10}^{\uns}= \partial_u \Gout^\uns\circ\Fout(0)$ and $\partial_u r_{11}^{\uns}$ and from
definition $r^{\uns}_1=r_{10}^{\uns}+r_{11}^{\uns}$.
Finally, using that $r_1^\uns$ satisfies equation~\eqref{PDEequal}, we easily obtain
$\|\partial_\theta r_1\|_{4,\bar\ost}^{\bar \dist,\bar \beta, \bar T}\leq K\delta^{p+4}$.

\section{The difference $\Delta(u,\theta)$. Proof of Theorem~\ref{thmdifpartsolinjective}}\label{secproofDif}
In this section we will prove Theorem~\ref{thmdifpartsolinjective}.
It is clear that the difference $\Delta(\vu,\theta)=r_1^\uns(\vu,\theta)-r_1^\sta(\vu,\theta)$ is defined on
$(\vu,\theta)\in\Doutinter\times\Tout$, where
$\Doutinter=\DoutT{\uns}\cap\DoutT{\sta}$ (see Figure~\ref{figureDoutinter})so this will be our domain from now on.

\subsection{Preliminary considerations}
As we explained in Section~\ref{secDifference}, the difference $\Delta=r_1^{\sta}-r_{1}^{\uns}$ satisfies the linear PDE~\eqref{PDE_difference_intro}
and can be expressed of the form~\eqref{defDelta}: $\Delta(\vu,\theta)=P(\vu,\theta)\tilde{k}(\xi(\vu,\theta))$, with
\begin{equation}\label{xi}
\xi(\vu,\theta)=\theta+\delta^{-1}\alpha\vu+c\coef^{-1}\log\cosh(\coef\vu)+C(\vu,\theta),
\end{equation}
a solution of~\eqref{PDE_k} and
\begin{equation}\label{formaP}
P(\vu,\theta)=\cosh^{2/\coef}(\coef\vu)(1+P_1(\vu,\theta))
\end{equation}
a solution of~\eqref{PDE_difference_intro}.
An straightforward computation shows that if $C$ is a solution of
\begin{align}\label{PDE_C}
 (-\delta^{-1}\alpha-c\hetz(\vu))\partial_\theta C+\partial_\vu C=&l_2(\vu,\theta)(\delta^{-1}\alpha+c\hetz(\vu)+\partial_\vu C)\notag \\
&+l_3(\vu,\theta)(1+\partial_\theta C),
\end{align}
then, $\xi$ is a solution of~\eqref{PDE_k}. Conversely, if $P_1$ is a solution of
\begin{align}\label{PDEpartsol1}
 \left(-\delta^{-1}\alpha-c\hetz(\vu)\right)\partial_\theta P_1+\partial_\vu P_1=&(2\param+l_1(\vu,\theta)+2\hetz(\vu)l_2(\vu,\theta))(1+ P_1)\nonumber\\
&+l_2(\vu,\theta)\partial_\vu P_1+l_3(\vu,\theta)\partial_\theta P_1,
\end{align}
then $P$ as in~\eqref{formaP} is  a solution of~\eqref{PDE_difference_intro}.

Therefore, we focus to prove the existence and the properties stated in
Theorem~\ref{thmdifpartsolinjective} of the functions $C$ and $P_1$.
To this aim, first we point out that the linear operator on the left hand side of equations~\eqref{PDE_C} and~\eqref{PDEpartsol1} is in both cases:
\begin{equation}\label{defLpart}
\hat{\mathcal{L}}(\phi)=(-\delta^{-1}\alpha-c\hetz(\vu))\partial_\theta\phi+\partial_\vu\phi.
\end{equation}
In order to prove the existence of $C$ and $P_1$, we will use a right inverse
of the operator $\Lpart$, which we will call $\Gpart$.

As we did with $\Gout^{\uns,\sta}$, we look for an expression of $\Gpart$ by solving the ordinary differential equations that its
Fourier coefficients satisfy. Proceeding in this way and taking into account that our functions are
defined in $\Doutinter\times\Tout$,
we define $\Gpart$ as the operator acting on functions $\phi$ defined in $\Doutinter\times\Tout$ as:
\begin{equation}\label{defGpart}
\Gpart(\phi)(\vu,\theta)=\sum_{l\in\mathbb{Z}}\Gpart^{[l]}(\phi)(\vu)e^{il\theta},
\end{equation}
where:
\begin{eqnarray}\label{defGpartcoefs}
 \Gpart^{[l]}(\phi)(\vu)&=&\int_{\vu_+}^\vu e^{-il\delta^{-1}\alpha(w-\vu)-ilc\coef^{-1}\log\left(\frac{\cosh(\coef w)}{\cosh(\coef\vu)}\right)}\phi^{[l]}(w)dw, \qquad \textrm{if } l<0,\bigskip\nonumber\\
 \Gpart^{[0]}(\phi)(\vu)&=&\int_{\vu_\textup{R}}^\vu \phi^{[0]}(w)dw,\bigskip\\
 \Gpart^{[l]}(\phi)(\vu)&=&\int_{\vu_-}^\vu e^{-il\delta^{-1}\alpha(w-\vu)-ilc\coef^{-1}\log\left(\frac{\cosh(\coef w)}{\cosh(\coef\vu)}\right)}\phi^{[l]}(w)dw, \qquad \textrm{if } l>0,\nonumber
\end{eqnarray}
and $\vu_\pm=\pm i(\pi/(2\coef)-\delta\dist)$ and $\vu_\textup{R}\in\mathbb{R}$ is the point of $\Doutinter$ with largest real part
(see Figure~\ref{figureDoutinter} in Section~\ref{secDifference}).

In the next section, we introduce the Banach spaces we will work with, give some bounds of the functions $l_i$
and finally check that the operator $\Gpart$ is well defined and has appropriate properties.

\subsection{Banach spaces and properties of $\Gpart$}
We will consider functions $\phi:\Dout{}\times\Tout\rightarrow\mathbb{C}$. Again, they can be written in their Fourier series
$\phi(\vv,\theta)=\sum_{l\in\mathbb{Z}}\phi^{[l]}(\vv)e^{il\theta}$.
In a similar way as we did in Section~\ref{Banachout} we define the norms:
\begin{align*}
\left\|\phi^{[l]}\right\|_{n}&=\sup_{\vv\in\Dout{}}\left|\cosh^n(\coef\vv)\phi^{[l]}(\vv)\right|, \qquad
\|\phi\|_{n,\ost}=\sum_{l\in\mathbb{Z}}\left\|\phi^{[l]}\right\|_{n}e^{|l|\ost},\\
\llfloor \phi\rrfloor_{n,\ost}&=\|\phi\|_{n,\ost}+\|\partial_\vv \phi\|_{n+1,\ost}+\delta^{-1}\|\partial_\theta \phi\|_{n+1,\ost}
\end{align*}
and we consider the Banach spaces endowed with these norms:
\begin{align*}
\Bsout{}{n}&:=\left\{\phi:\Dout{}\times\Tout\rightarrow\mathbb{C} \,:\, \phi\textrm{ is analytic, such that } \|\phi\|_{n,\ost}<+\infty\right\},\\
\Bsoutdiff{}{n}&:=\left\{\phi:\Dout{}\times\Tout\rightarrow\mathbb{C} \,:\, \phi\textrm{ is analytic, such that }
\llfloor \phi\rrfloor_{n,\ost}<+\infty\right\}.
\end{align*}

\begin{remark}
From Theorem~\ref{thmoutloc}, $r_1^{\uns,\sta}\in\Bsoutdiff{}{3}$ and that there exists a constant $M$ such that
$\llfloor r_1^{\uns,\sta}\rrfloor_{3,\ost}\leq M\delta^{p+3}$.
\end{remark}

Next lemma provides bounds for $l_1,l_2$ and $l_3$. Its proof is given in~\cite{CastejonPhDThesis}.
\begin{lemma}\label{lemnormint}
Let $l_i(\vu,\theta)$, $i=1,2,3$, be the functions defined in~\eqref{defl1}--\eqref{defl3}. There exists a constant $M$ such that:
$$\|l_1\|_{2,\ost}\leq M\delta^{p+3},\qquad \|l_2\|_{1,\ost}\leq M\delta^{p+3},\qquad\|l_3\|_{2,\ost}\leq M\delta^{p+3}.$$
\end{lemma}

To finish this section we enunciate the properties of the linear operator $\Gpart$ which turns out to be very similar to the ones
of $\Gout$ in Lemma~\ref{lempropsopG}. Its proof involves several technicalities and can be found in~\cite{CastejonPhDThesis}.
\begin{lemma}\label{lemGpart}
 Let $l\in\mathbb{Z}$, $n\geq1$ and $\phi\in\Bsout{}{n}$. There exists $M>0$ such that:
\begin{enumerate}
\item \label{lemfitacoefGpartpol} If $n>1$, then $\|{\Gpart}^{[l]}(\phi)\|_{n-1}\leq M\left\|\phi^{[l]}\right\|_{n}.$
\item\label{lemfitacoefGpartdelta}  If $l\neq0$, then $\displaystyle\|{\Gpart}^{[l]}(\phi)\|_{n}\leq \frac{\delta M\left\|\phi^{[l]}\right\|_{n}}{|l|}.$
\item \label{corfitaGpart} As a consequence, if $n>1$, $\|\Gpart(\phi)\|_{n-1,\ost}\leq M\|\phi\|_{n,\ost}$.
Moreover, if $\phi^{[0]}(\vv)=0$, then for all $n\geq1$, $\|\Gpart(\phi)\|_{n,\ost}\leq \delta M\|\phi\|_{n,\ost}$.
\item \label{lemfitacoefGpartialtheta}$\|\partial_\theta\Gpart(\phi)\|_{n,\ost}\leq \delta M\|\phi\|_{n,\ost}.$
\item \label{lemfitacoefGpartialu}$\|\partial_\vu\Gpart(\phi)\|_{n,\ost}\leq M\|\phi\|_{n,\ost}.$
\item \label{corfitaGdiffpart} In conclusion, if $n>1$ and $\phi\in\Bsout{}{n}$, then $\Gpart(\phi)\in\Bsoutdiff{}{n-1}$ and there exists $M>0$ such that:
$$\llfloor \Gpart(\phi)\rrfloor_{n-1,\ost}\leq M\|\phi\|_{n,\ost}.$$
\end{enumerate}
\end{lemma}

In the following two sections we will prove the results related to the functions $C$ and $P_1$.
\subsection{Existence and properties of $C$}\label{secpropconst}
We enunciate the results about the function $C$ we will prove. They give a more precise information than
the ones in Theorem~\ref{thmdifpartsolinjective}.
\begin{proposition}\label{propconst}
There exists a particular solution $C$ of~\eqref{PDE_C} of the form:
\begin{equation}\label{defCl2}
C(\vu,\theta)=\delta^{-1}\alpha\int_0^\vu l_2^{[0]}(w)dw+C_1(\vu,\theta),
\end{equation}
with $l_2^{[0]}(\vu)$ the average of the function $l_2$ defined in~\eqref{defl2}.
\begin{equation}\label{boundC1-const}
\|C_1\|_{1,\ost}\leq M\delta^{p+3}, \qquad \|\partial_\vu C\|_{1,\ost}\leq M\delta^{p+2},\qquad \|\partial_\theta C\|_{1,\ost}\leq M\delta^{p+3}.
\end{equation}
Finally $(\xi(\vu,\theta),\theta)$, with $\xi$ given by~\eqref{xi}, is injective in $\Doutinter\times\Tout$.
\end{proposition}
\begin{remark} Assuming that $C$ actually exists and recalling that $\Delta$ has the form~\eqref{Deltaforma2}, the function
$$
k(\vu,\theta):= \tilde k(\theta+\delta^{-1}\alpha\vu+c\coef^{-1}\log\cosh(\coef\vu)+C(\vu,\theta)),
$$
has to be $2\pi-$periodic in $\theta$, which implies that $\tilde k(\tau)$ is $2\pi-$periodic in $\tau$.
\end{remark}

Now we make some further considerations on the integral $\int_0^\vu l_2^{[0]}(\vw)d\vw$.
First of all, we point out that using Lemma~\ref{lemnormint} and the fact that for $w\in\Doutinter$ one has
$|\cosh(\coef w)|\geq K|w^2-\pi^2/(2\coef)^2|$, one obtains:
$$
\left|\delta^{-1}\alpha\int_0^\vu l_2^{[0]}(w)dw\right|\leq K\delta^{p+2}\int_0^\vu|\cosh(\coef w)|^{-1}dw\leq K\delta^{p+2}|\log(\delta\dist)|.
$$
Hence, in the regular case $p>-2$ this integral is small, even for complex values of $\vu\in \Doutinter$, and one can avoid to take into account its
contribution to the function $C(\vu,\theta)$ defined in~\eqref{defCl2}. Notice that when $\vu\in \mathbb{R}$, this
integral is $\mathcal{O}(\delta^{p+2})$. However, in the singular case $p=-2$,
one needs to have some more precise knowledge of its behavior.

The following result deals with that integral. Its proof is given with detail in~\cite{CastejonPhDThesis}.
\begin{lemma}
Define $L_0$ as the following limit, that is well defined:
$$
L_0=\lim_{\vu\to i\frac{\pi}{2\coef}}\lim_{\delta\to0}\delta^{-p-3}l_2^{[0]}(\vu)\tanh^{-1}(\coef\vu).
$$
Then, there exist functions $L(\vu)$ and $\Lambda(\vu)$ such that for all $\vu\in\Doutinter$:
$$\int_0^\vu l_2^{[0]}(\vw)dw=\delta^{p+3}\coef^{-1}L_0\log\cosh(\coef\vu)+\delta L(\vu)+\delta\Lambda(\vu).$$
Moreover, $L_0\in\mathbb{R}$, $L(0)=0$ and $L(\vu)$ is defined on the limit $\vu\to i\pi/(2\coef)$
and
$\|L\|_0\leq M\delta^{p+2}$, $\|L'\|_0\leq M\delta^{p+2}$ and $\|\Lambda\|_1\leq M\delta^{p+3}$, for some $M>0$.
\end{lemma}
\begin{remark}\label{rmkL0L} One can obtain explicit expressions for $L_0$, $L(\vu)$ and $\Lambda(\vu)$ depending only on $\Fout(0)$ or equivalently, on
$F(0),G(0)$ and $H(0)$ defined in~\eqref{notationFGHbis}. See~\cite{CastejonPhDThesis} for details.

We write the formulae for $L_0$ and $L(\vu)$. We define the constants
\begin{align*}
\rho_0 = &\frac{(\coef+1)}{2b\coef(3\coef+2)}\left[\frac{(\coef+1)}{4b}\left(f_{3120}+g_{3210}+3f_{3300}+3g_{3030}\right)-\left(f_{3102}+g_{3012}\right)\right.\\
&\left.-\frac{\coef+1}{b}\left(h_{3201}+h_{3021}\right)+2h_{3003}\right] \\
H_0=&-h_{3003}+\frac{\coef+1}{2b}\left(h_{3021}+h_{3201}\right),
\end{align*}
where the coefficients $f_{qkmn},g_{qkmn}$ and $h_{qkmn}$ were defined as in~\eqref{taylorf}.

In the conservative case, $L_0=-h_{3003}$ meanwhile in the dissipative one:
$$
L_0=-\frac{2b}{\coef} \rho_0 -\frac{1}{\coef}H_0.
$$
With respect to $L(\vu)$, we define
\begin{align*}
H_1(\vu)=&\frac{\cosh^3(\coef \vu) \lim_{\delta\to 0} \delta^{-3} \big (H(0)\big )^{[0]}-H_0\sinh(\coef\vu)}{\cosh(\coef\vu)},\\
\rho_1^{\uns,\sta}(\vu)=&\delta^{-(p+3)}\param\int_{\mp\infty}^\vu \frac{R_0(w)}{\cosh^{4+\frac{2}{\coef}}(\coef w)}dw \\
&+\cosh^{\frac{2}{\coef}}(\coef\vu)\int_{\mp\infty}^\vu\frac{\mathcal{F}_1(w)}{\cosh^{4+\frac{2}{\coef}}(\coef w)}dw-\rho_0\frac{\tanh(\coef\vu)}{\cosh^2(\coef\vu)},\\
\mathcal{F}_1(\vu)=&\lim_{\delta \to 0} \delta^{-p-3} \cosh^4 (\vu) \big[\Fout^{[0]}(\vu) - 2\param \hetr(\vu)\big],
\end{align*}
where we take $-$ for the unstable case and $+$ for the stable one. Then
$$
L(\vu)=-\delta^{p+2}\int_0^\vu \frac{b}{\coef}\cosh^2(\coef w)(\rho_1^\uns(\vu)+\rho_1^\sta(w))+\frac{1}{\coef}H_1(w)dw.
$$
\end{remark}

\subsubsection{Proof of Proposition~\ref{propconst}}
Let us define:
$$\hat l_2(\vu,\theta)=l_2(\vu,\theta)-l_2^{[0]}(\vu).$$
It is easy to see that in order that $C$ defined in~\eqref{defCl2} satisfies~\eqref{PDE_C} it is enough that $C_1$  satisfies the following equation:
\begin{align}\label{PDE_C1}
 (-\delta^{-1}\alpha-c\hetz(\vu))\partial_\theta C_1+\partial_\vu C_1=&\delta^{-1}\alpha \hat l_2(\vu,\theta)+l_3(\vu,\theta)(1+\partial_\theta C_1)\\
&
+l_2(\vu,\theta)(c\hetz(\vu)+\delta^{-1}\alpha l_2^{[0]}(\vu)+\partial_\vu C_1).\nonumber
\end{align}
We define the operator $\Nconstnew$ as:
\begin{equation}\label{defNconstnew}
 \Nconstnew(\phi)=\frac{\alpha}{\delta} \hat l_2(\vu,\theta)+l_2(\vu,\theta)(c\hetz(\vu)+\frac{\alpha}{\delta} l_2^{[0]}(\vu)+\partial_\vu \phi)+l_3(\vu,\theta)(1+\partial_\theta \phi).
\end{equation}
Then equation~\eqref{PDE_C1} can be rewritten as $\Lpart(C_1)=\Nconstnew(C_1)$,
where $\Lpart$ was defined in~\eqref{defLpart}. It is enough then to solve the fixed point equation:
$$
C_1=\tilde{\Nconstnew}(C_1),
$$
where $\tilde{\Nconstnew}=\Gpart\circ\Nconstnew$, and $\Gpart$ is the operator defined in~\eqref{defGpart}.

\begin{lemma}\label{lemfpfunC1}
For $\dist$ big enough and $p\geq-2$, the operator $\tilde{\Nconstnew}:\Bsoutdiff{}{0}\rightarrow\Bsoutdiff{}{0}$.
Moreover, there exists a constant $M$ such that $\llfloor\tilde{\Nconstnew}(0)\rrfloor_{0,\ost}\leq M\delta^{p+2}$,
and $\tilde{\Nconstnew}$ has a unique fixed point in the ball $B\left(2\llfloor\tilde{\Nconstnew}(0)\rrfloor_{0,\ost}\right)\subset\Bsoutdiff{}{0}$.
\end{lemma}
\begin{proof}
First of all we shall prove that:
\begin{equation}\label{boundfirstitNconstnew}
\llfloor\tilde{\Nconstnew}(0)\rrfloor_{0,\ost}\leq K\delta^{p+2}.
\end{equation}
We have:
$$\Nconstnew(0)=\delta^{-1}\alpha\hat l_2(\vu,\theta)+l_2(\vu,\theta)(c\hetz(\vu)+\delta^{-1}\alpha l_2^{[0]}(\vu))+l_3(\vu,\theta).$$
To prove~\eqref{boundfirstitNconstnew} we shall bound the Fourier coefficients of $\Nconstnew(0)$ and then use Lemma~\ref{lemGpart}.
On the one hand, since by definition $\hat l_2$ has zero average, one has:
$$\Nconstnew^{[0]}(0)=l_2^{[0]}(\vu)(c\hetz(\vu)+\delta^{-1}\alpha l_2^{[0]}(\vu))+l_3^{[0]}(\vu).$$
Using Lemma~\ref{lemnormint} and the properties of the norm:
\begin{equation}\label{boundaverageNconstnew}
 \|\Nconstnew^{[0]}(0)\|_2\leq \|l_2^{[0]}\|_1(c\|\hetz\|_1+\delta^{-1}\alpha\|l_2^{[0]}\|_1)+\|l_3^{[0]}\|_2\leq K\delta^{p+3}.
\end{equation}
Then, by item~\ref{lemfitacoefGpartpol} of Lemma~\ref{lemGpart} one has:
\begin{equation}\label{B00}
 \|\Gpart^{[0]}(\Nconstnew(0))\|_1\leq K\|\Nconstnew^{[0]}(0)\|_2\leq K\delta^{p+3}.
\end{equation}
On the other hand, for the remaining Fourier coefficients one has:
$$\Nconstnew^{[l]}(0)=l_2^{[l]}(\vu)(\delta^{-1}\alpha+c\hetz(\vu)+\delta^{-1}\alpha l_2^{[0]}(\vu))+l_3^{[l]}(\vu)\qquad \qquad l\neq0.$$
Again, using Lemma~\ref{lemnormint} and the properties of the norm, we obtain:
\begin{equation}\label{boundlNconstnew}
 \|\Nconstnew^{[l]}(0)\|_1\leq \|l_2^{[l]}\|_1(\delta^{-1}\alpha+c\|\hetz\|_0+\delta^{-1}\alpha\|l_2^{[0]}\|_0)+\|l_3^{[l]}\|_1\leq K\delta^{p+2}.
\end{equation}
Then by item~\ref{lemfitacoefGpartdelta} of Lemma~\ref{lemGpart} and taking into account that $l\neq0$, we have:
\begin{equation}\label{B0l}
 \|\Gpart^{[l]}(\Nconstnew(0))\|_1\leq \frac{K\delta \|\Nconstnew^{[l]}(0)\|_1}{|l|}\leq \frac{K\delta^{p+3}}{|l|}.
\end{equation}
From~\eqref{B00} and~\eqref{B0l} we obtain:
\begin{equation}\label{boundB0-norm1}
 \|\tilde\Nconstnew(0)\|_{1,\ost}\leq K\delta^{p+3},
\end{equation}
and as a consequence:
\begin{equation}\label{boundB0}
 \|\tilde\Nconstnew(0)\|_{0,\ost}\leq K\frac{\delta^{p+2}}{\dist}\leq K\delta^{p+2}.
\end{equation}
We note that from bounds~\eqref{boundaverageNconstnew} and~\eqref{boundlNconstnew} we have
$ \|\Nconstnew(0)\|_{1,\ost}\leq K\delta^{p+2}$,
and then from items~\ref{lemfitacoefGpartialtheta} and~\ref{lemfitacoefGpartialu} of Lemma~\ref{lemGpart} we obtain directly:
\begin{equation}\label{boundpartialuB0}
 \|\partial_\vu\tilde\Nconstnew(0)\|_{1,\ost}\leq K\|\Nconstnew(0)\|_{1,\ost}\leq K\delta^{p+2},
\end{equation}
and:
\begin{equation}\label{boundpartialthetaB0}
 \|\partial_\theta\tilde\Nconstnew(0)\|_{1,\ost}\leq K\delta \|\Nconstnew(0)\|_{1,\ost}\leq K\delta^{p+3}.
\end{equation}
From bounds~\eqref{boundB0},~\eqref{boundpartialuB0} and~\eqref{boundpartialthetaB0} one obtains bound~\eqref{boundfirstitNconstnew}.

It is not difficult to check that given two functions $\phi_1,\phi_2\in\Bsoutdiff{}{0}$:
\begin{equation}\label{lipconstN}
\llfloor\tilde\Nconstnew(\phi_1)-\tilde\Nconstnew(\phi_2)\rrfloor_{0,\ost} \leq\frac{K}{\dist}\delta^{p+2}\llfloor\phi_1-\phi_2\rrfloor_{0,\ost}.
\end{equation}

To finish the proof, we take $\dist$ sufficiently large such that the Lipschitz constant in~\eqref{lipconstN} is smaller than 1. Then $\tilde\Nconstnew$ sends $B\left(2\llfloor\tilde\Nconstnew(0)\rrfloor_{0,\ost}\right)$ to itself and since it is contractive, it has a unique fixed point in this ball.
\end{proof}

\begin{proof}[End of the proof of Proposition~\ref{propconst}]
Let us define $C_1$ as the unique fixed point of the operator $\tilde{\Nconstnew}$ in the ball $B\left(2\llfloor\tilde{\Nconstnew}(0)\rrfloor_{0,\ost}\right)$,
whose existence is proven by Lemma~\ref{lemfpfunC1}. Let $C$ be defined as in~\eqref{defCl2} and
$\xi$ the function defined as~\eqref{xi}. It remains to check that bounds~\eqref{boundC1-const} hold and that $(\xi(\theta,\vu),\theta)$ is injective.

First we shall see that $C_1$ satisfies the bound in~\eqref{boundC1-const}.
We point out that this is not given directly by Lemma~\ref{lemfpfunC1}, but it can be obtained \textit{a posteriori}. Indeed, by definition $C_1$ satisfies:
$C_1=\Gpart(\Nconstnew(C_1))$.
By the definition~\eqref{defNconstnew} of the operator $\Nconstnew$, and since $\Gpart$ is linear, we can write:
\begin{equation}\label{rewrite-C1}
C_1=\Gpart(\Nconstnew(0))+\Gpart(l_2(\vu,\theta)\partial_\vu C_1+l_3(\vu,\theta)\partial_\theta C_1).
\end{equation}
On the one hand, we recall bound~\eqref{boundB0-norm1} which stated:
\begin{equation}\label{boundB0-norm1-recall}
\|\Gpart(\Nconstnew(0))\|_{1,\ost}\leq K\delta^{p+3}.
\end{equation}
On the other hand, since $C_1\in B\left(2\llfloor\tilde{\Nconstnew}(0)\rrfloor_{0,\ost}\right)$, by the definition of the norm $\llfloor.\rrfloor_{0,\ost}$ and the bound of $\llfloor\tilde{\Nconstnew}(0)\rrfloor_{0,\ost}$ provided by Lemma~\ref{lemfpfunC1}, one has:
\begin{equation}\label{bounds-C1-deriv-norm1}
 \|\partial_\vu C_1\|_{1,\ost}\leq K\delta^{p+2},\qquad \|\partial_\theta C_1\|_{1,\ost}\leq K\delta^{p+3}.
\end{equation}
Then, using Lemma~\ref{lemnormint} and bounds~\eqref{bounds-C1-deriv-norm1} it is easy to see that:
$$\|l_2(\vu,\theta)\partial_\vu C_1+l_3(\vu,\theta)\partial_\theta C_1\|_{2,\ost}\leq K\delta^{2p+5},$$
so that by item~\ref{corfitaGpart} of Lemma~\ref{lemGpart} we obtain:
\begin{equation}\label{bound-partialsC1-ls}
\|\Gpart(l_2(\vu,\theta)\partial_\vu C_1+l_3(\vu,\theta)\partial_\theta C_1)\|_{1,\ost}\leq K\delta^{2p+5}.
\end{equation}
Using that $p\geq-2$ and bounds~\eqref{boundB0-norm1-recall} and~\eqref{bound-partialsC1-ls} in equation~\eqref{rewrite-C1}, we obtain
$\|C_1\|_{1,\ost}\leq K\delta^{p+3}$, and then bound~\eqref{boundC1-const} is obtained.

The bounds in~\eqref{boundC1-const} for $C$ are consequence of~\eqref{bounds-C1-deriv-norm1} and Lemma~\ref{lemnormint}.
It only remains to prove that $(\xi(\theta,\vu),\theta)$ is injective. Let us assume $\xi(\vu_1,\theta)=\xi(\vu_2,\theta)$. This means:
$$
 \vu_1-\vu_2=\delta\coef^{-1}\alpha^{-1}c(\log\cosh(\coef\vu_1)-\log\cosh(\coef\vu_2))+\delta\alpha^{-1}(C(\vu_1,\theta)-C(\vu_2,\theta)).
$$
On the one hand, for $\vu_1, \vu_2\in\Doutinter$ we have:
$$
 \delta\coef^{-1}\alpha^{-1}c|\log\cosh(\coef\vu_1)-\log\cosh(\coef\vu_2)|\leq \frac{K}{\dist}|\vu_1-\vu_2|.
$$
On the other hand, using the mean value theorem and bound~\eqref{boundC1-const}:
$$
 \delta|C(\vu_1,\theta)-C(\vu_2,\theta)|\leq \frac{K\delta^{p+2}}{\dist}|\vu_1-\vu_2|,
$$
Thus, since $p\geq-2$, we know that there exists a constant $K$ such that:
$|\vu_1-\vu_2|\leq\dist^{-1} K|u_1-u_2|$.
Taking $\dist$ sufficiently large yields $\vu_1=\vu_2$.
\end{proof}

\subsection{Existence and properties of $P_1$}
Our goal is to find a particular solution of the equation~\eqref{PDEpartsol1} satisfying the properties stated in
Theorem~\ref{thmdifpartsolinjective}. We introduce the operator
\begin{equation}\label{defopA}
 \Apart(\phi)=(2\param+l_1(\vu,\theta)+2\hetz(\vu)l_2(\vu,\theta))(1+ \phi)+l_2(\vu,\theta)\partial_\vu \phi+l_3(\vu,\theta)\partial_\theta \phi,
\end{equation}
in such a way that equation~\eqref{PDEpartsol1} can be written as:
\begin{equation}\label{PDEpartsol1short}
 \Lpart(P_1)=\Apart(P_1).
\end{equation}

We distinguish between the dissipative and the conservative case, since in the first case $P_1$ will be found using a fixed point equation, while in the latter it will be defined in terms of the function $C$ of Theorem~\ref{thmdifpartsolinjective}.

\subsubsection{The dissipative case}
In this subsection we will follow the same steps as in Section~\ref{secpropconst} to prove the existence of $C_1$.
Since $P_1$ has to be a solution of~\eqref{PDEpartsol1short}, we shall do it by solving the fixed point equation:
\begin{equation}\label{fppartsol}
 P_1=\tilde{\Apart}(P_1),
\end{equation}
where $\tilde{\Apart}=\Gpart\circ\Apart$, and $\Gpart$ is the operator defined by~\eqref{defGpart} and~\eqref{defGpartcoefs}, and $\Apart$ is defined in~\eqref{defopA}.

\begin{lemma}\label{lemfppartsol}
 For $\dist$ big enough and $p\geq-2$, the operator $\tilde{\Apart}:\Bsoutdiff{}{1}\rightarrow\Bsoutdiff{}{1}$, and it has a unique fixed point in the ball $B\left(2\llfloor\tilde\Apart(0)\rrfloor_{1,\ost}\right)\subset\Bsoutdiff{}{1}$. Moreover, there exists a constant $M$ such that $\llfloor\tilde{\Apart}(0)\rrfloor_{1,\ost}\leq K\delta^{p+3}$.
\end{lemma}
\begin{proof}
First we deal with $\tilde{\Apart}(0)$.
Indeed, using Lemma~\ref{lemnormint} and that $|\param|\leq \param^*\delta^{p+3}$, it is straightforward to prove that there exists a constant $K$ such that:
$\|\Apart(0)\|_{2,\ost}\leq K\delta^{p+3}$.
Then item~\ref{corfitaGdiffpart} of Lemma~\ref{lemGpart} yields
$\llfloor\tilde{\Apart}(0)\rrfloor_{1,\ost}\leq K\delta^{p+3}$.

Next step is to find the Lipschitz constant of the operator $\tilde\Apart$. To do so, let us fix $\phi_1,\phi_2\in B\left(2\llfloor\tilde\Apart(0)\rrfloor_{1,\ost}\right)$.
We have:
\begin{align}\label{apart1-2sep}
 \|\Apart(\phi_1)-&\Apart(\phi_2)\|_{2,\ost}\leq 2|\param|\|\phi_1-\phi_2\|_{2,\ost}+\|l_1\|_{1,\ost}\|\phi_1-\phi_2\|_{1,\ost}\nonumber\\
&+2\|\hetz\|_{1,\ost}\|l_2\|_{0,\ost}\|\phi_1-\phi_2\|_{1,\ost}+\|l_2\|_{0,\ost}\|\partial_\vu(\phi_1-\phi_2)\|_{2,\ost}\nonumber\\
&+\|l_3\|_{1,\ost}\|\partial_\theta(\phi_1-\phi_2)\|_{1,\ost}.
\end{align}
First we note that, since $\param=\mathcal{O}(\delta^{p+3})$:
\begin{equation}\label{boundparamp12}
 |\param|\|\phi_1-\phi_2\|_{2,\ost}\leq K\frac{\delta^{p+2}}{\dist}\|\phi_1-\phi_2\|_{1,\ost}\leq K\frac{\delta^{p+2}}{\dist}\llfloor \phi_1-\phi_2\rrfloor_{1,\ost}.
\end{equation}
Similarly, by Lemma~\ref{lemnormint}:
\begin{equation}\label{boundsl123}
 \|l_1\|_{1,\ost}\leq K\frac{\delta^{p+2}}{\dist},\qquad \|l_2\|_{0,\ost}\leq K\frac{\delta^{p+2}}{\dist},\qquad\|l_3\|_{1,\ost}\leq K\frac{\delta^{p+2}}{\dist}.
\end{equation}
Finally, we just need to note that:
\begin{equation}\label{boundsp12}
 \|\partial_\theta(\phi_1-\phi_2)\|_{1,\ost}\leq \frac{K}{\delta\dist}\|\partial_\theta(\phi_1-\phi_2)\|_{2,\ost}\leq K\llfloor \phi_1-\phi_1)\rrfloor_{1,\ost}.
\end{equation}
Using the definition of the norm $\llfloor .\rrfloor_{1,\ost}$, the fact that $\|\hetz\|_{1,\ost}\leq K$ and the previous
bounds~\eqref{boundparamp12},~\eqref{boundsl123} and~\eqref{boundsp12} in equation~\eqref{apart1-2sep} we immediately obtain
$$
\|\Apart(\phi_1)-\Apart(\phi_2)\|_{2,\ost}\leq K\frac{\delta^{p+2}}{\dist}\llfloor \phi_1-\phi_2\rrfloor_{1,\ost}.
$$
Again,  by item~\ref{corfitaGdiffpart} of Lemma~\ref{lemGpart}
\begin{equation}\label{LipConstA}
 \llfloor \tilde{\Apart}(\phi_1)-\tilde{\Apart}(\phi_2)\rrfloor_{1,\ost}\leq K\frac{\delta^{p+2}}{\dist}\llfloor \phi_1-\phi_2\rrfloor_{1,\ost} .
\end{equation}
To finish the proof, we take $\dist$ large enough such that the Lipschitz constant in~\eqref{LipConstA} is smaller than 1. Then the fixed point theorem
yields the result.
\end{proof}

The fact $P_1$ satisfies equation~\eqref{PDEpartsol1short} (and consequently~\eqref{PDEpartsol1}) is clear since it is a solution of equation~\eqref{fppartsol}.
Clearly, using Lemma~\ref{lemfppartsol}, one has,
$\llfloor P_1\rrfloor_{1,\ost}\leq2\llfloor\tilde\Apart(0)\rrfloor_{1,\ost}\leq K\delta^{p+3}$.
Since:
$$\sup_{(\vu,\theta)\in\Doutinter\times\Tout}|P_1(\vu,\theta)|\leq\llfloor P_1\rrfloor_{1,\ost}\sup_{(\vu,\theta)\in\Doutinter\times\Tout}|\cosh^{-1}(\coef\vu)|\leq K\frac{\delta^{p+2}}{\dist},$$
taking $\dist$ sufficiently large we obtain $|1+P_1(\vu,\theta)|\geq 1-K\dist^{-1} \delta^{p+2}\neq0$. Finally,
since $\cosh^{2/\coef}(\coef\vu)\neq0$ for $\vu\in\Doutinter$ we can ensure that, for $(\vu,\theta)\in\Doutinter\times\Tout$, $P(\vu,\theta)=\cosh^{2/\coef}(\coef\vu)(1+P_1(\vu,\theta))\neq0$.

\subsubsection{The conservative case}
We recall that in the conservative case we have $\coef=1$ and $\param=0$. Let:
\begin{equation}\label{defP1conservative}
P_1(\vu,\theta)=\frac{\partial_\vu C(\vu,\theta)-l_3(\vu,\theta)}{\delta^{-1}\alpha+c\hetz(\vu)+l_3(\vu,\theta)},
\end{equation}
where $C(\vu,\theta)$ is the function given by Proposition~\ref{propconst} and $l_3(\vu,\theta)$ is defined in~\eqref{defl3}. First let us check that it satisfies
bound~\eqref{boundP1-prop}, that is:
\begin{equation}\label{boundP1_prop-recall}
|P_1(\vu,\theta)|\leq \frac{K}{\dist}\delta^{p+2},
\end{equation}
for all $(\vu,\theta)\in\Doutinter\times\Tout$. On the one hand, note that by Proposition~\ref{propconst} and Lemma~\ref{lemnormint} we have :
\begin{equation}\label{boundpartialuC-l3}
|\partial_\vu C(\vu,\theta)|\leq \frac{K}{\dist}\delta^{p+1},\qquad  |l_3(\vu,\theta)|\leq\frac{K}{\dist^2}\delta^{p+1}.
\end{equation}
On the other hand, taking $\dist$ sufficiently large, we also have:
\begin{equation}\label{boundfrac}
 \left|\frac{1}{\delta^{-1}\alpha+c\hetz(\vu)+l_3(\vu,\theta)}\right|\leq K\delta.
\end{equation}
Then~\eqref{boundP1_prop-recall} follows directly from using~\eqref{boundpartialuC-l3} and~\eqref{boundfrac} in~\eqref{defP1conservative}.

It only remains to prove that $P_1$ defined in~\eqref{defP1conservative} satisfies equation~\eqref{PDEpartsol1}:
\begin{align}\label{PDEpartsol1-recall}
  \left(-\delta^{-1}\alpha-c\hetz(\vu)\right)\partial_\theta P_1+\partial_\vu P_1=&(l_1(\vu,\theta)+2\hetz(\vu)l_2(\vu,\theta))(1+ P_1)\nonumber\\
&+l_2(\vu,\theta)\partial_\vu P_1+l_3(\vu,\theta)\partial_\theta P_1.
\end{align}
Tedious but standard computations yields that $P_1$ is a solution of
\begin{align}\label{PDE_P1conservative_simple}
 (-\delta^{-1}\alpha-&c\hetz(\vu)-l_3(\vu,\theta))\partial_\theta P_1+\partial_\vu P_1\nonumber\\
 =&(\partial_\vu l_2(\vu,\theta)+\partial_\theta l_3(\vu,\theta))(1+P_1)+l_2(\vu,\theta)\partial_\vu P_1.
\end{align}
Therefore, equations~\eqref{PDE_P1conservative_simple} and~\eqref{PDEpartsol1-recall} are the same, if and only if:
$$
l_1(\vu,\theta)+2\hetz(\vu)l_2(\vu,\theta)=\partial_\vu l_2(\vu,\theta)+\partial_\theta l_3(\vu,\theta).
$$
This equality can be checked using the definitions of $l_2$ and $l_3$ as well as the fact that the vector field is
divergence free, that is:
$$\partial_zH(r)+\partial_\theta G(r)=-\partial_rF(r).$$

\section{Acknowledgments}
The authors are in debt with Jordi Villanueva for his help in the proof of Theorem~\ref{thmktilde0}.

The authors have been partially supported by the Spanish
MINECO-FEDER Grant MTM2015-65715-P and the Catalan Grant 2014SGR504.
Tere M-Seara is also supported by the Russian Scientific Foundation grant 14-41-00044. and European  Marie Curie Action FP7-PEOPLE-2012-IRSES: BREUDS.

\bibliography{references}

\def\cprime{$'$}
\begin{thebibliography}{BFGS12}

\bibitem[AS72]{AS}
M.~Abramowitz and I.~A. Stegun.
\newblock {\em {Handbook of mathematical functions: with formulas, graphs, and
  mathematical tables}}.
\newblock Number~55. Courier Dover Publications, 1972.

\bibitem[Bal06]{BaldomaInner}
I.~Baldom{\'a}.
\newblock {The inner equation for one and a half degrees of freedom rapidly
  forced Hamiltonian systems}.
\newblock {\em Nonlinearity}, 19(6):1415--1446, 2006.

\bibitem[BCS13]{BCS13}
I.~Baldom{\'a}, O.~Castej{\'o}n, and T.~M. Seara.
\newblock {Exponentially small heteroclinic breakdown in the generic Hopf-Zero
  singularity}.
\newblock {\em Journal of Dynamics and Differential Equations}, 25(2):335--392,
  2013.

\bibitem[BCS16]{BCS16b}
I.~Baldom{\'a}, O.~Castej{\'o}n, and T.~M. Seara.
\newblock {Breakdown of a 2D heteroclinic connection in the Hopf-zero
  singularity (II). The generic case}.
\newblock {\em preprint}, 2016.

\bibitem[BF04]{BF04}
I.~Baldom{\'a} and E.~Fontich.
\newblock {Exponentially small splitting of invariant manifolds of parabolic
  points}.
\newblock {\em Mem. Amer. Math. Soc.}, 167(792):x--83, 2004.

\bibitem[BF05]{BF05}
I.~Baldom{\'a} and E.~Fontich.
\newblock {Exponentially small splitting of separatrices in a weakly hyperbolic
  case}.
\newblock {\em J. Differential Equations}, 210(1):106--134, 2005.

\bibitem[BFGS12]{BFGS12}
I.~Baldom{\'a}, E.~Fontich, M.~Guardia, and T.~M. Seara.
\newblock {Exponentially small splitting of separatrices beyond {M}elnikov
  analysis: rigorous results}.
\newblock {\em J. Differential Equations}, 253(12):3304--3439, 2012.

\bibitem[BM12]{BM12}
I.~Baldom{\'a} and P.~Mart{\'i}n.
\newblock {The inner equation for generalized standard maps}.
\newblock {\em SIAM J. Appl. Dyn. Syst.}, 11(3):1062--1097, 2012.

\bibitem[BO93]{BO93}
A.~Benseny and C.~Oliv{\'e}.
\newblock {High precision angles between invariant manifolds for radpidly
  forced hamiltonian systems}.
\newblock {\em Proceedings Equadiff91}, pages 315--319, 1993.

\bibitem[BS06]{BaSe06}
I.~Baldom{\'a} and T.~M. Seara.
\newblock {Breakdown of heteroclinic orbits for some analytic unfoldings of the
  Hopf-zero singularity.}
\newblock {\em J. Nonlinear Sci.}, 16(6):543--582, 2006.

\bibitem[BS08]{BaSe08}
I.~Baldom{\'a} and T.~M. Seara.
\newblock {The inner equation for generic analytic unfoldings of the Hopf-zero
  singularity}.
\newblock {\em Discrete Contin. Dyn. Syst., Ser. B}, 10(2-3):323--347, 2008.

\bibitem[BT86]{BT86}
H.~W. Broer and F.~M. Tangerman.
\newblock {From a differentiable to a real analytic perturbation theory,
  applications to the {K}upka {S}male theorems}.
\newblock {\em Ergodic Theory Dynam. Systems}, 6(3):345--362, 1986.

\bibitem[BT89]{BT89}
H.~W. Broer and F.~Takens.
\newblock {Formally symmetric normal forms and genericity}.
\newblock In {\em {Dynamics reported, {V}ol.\ 2}}, volume~2 of {\em {Dynam.
  Report. Ser. Dynam. Systems Appl.}}, pages 39--59. Wiley, Chichester, 1989.

\bibitem[BV84]{BV84}
H.~W. Broer and G.~Vegter.
\newblock {Subordinate \v{S}il'nikov bifurcations near some singularities of
  vector fields having low codimension.}
\newblock {\em Ergodic Theory Dyn. Syst.}, 4:509--525, 1984.

\bibitem[Cas15]{CastejonPhDThesis}
O.~Castej{\'o}n.
\newblock {\em {Study of invariant manifolds in two different problems: the
  Hopf-zero singularity and neural synchrony}}.
\newblock PhD thesis, 2015.

\bibitem[DIKS13]{diks}
F.~Dumortier, S.~Ib{\'a}{\~n}ez, H.~Kokubu, and C.~Sim{\'o}.
\newblock {About the unfolding of a {H}opf-zero singularity}.
\newblock {\em Discrete Contin. Dyn. Syst.}, 33(10):4435--4471, 2013.

\bibitem[DRR99]{DR-R99}
A.~Delshams and R.~Ram{\'i}rez-Ros.
\newblock {Singular separatrix splitting and the {M}elnikov method: an
  experimental study}.
\newblock {\em Experiment. Math.}, 8(1):29--48, 1999.

\bibitem[DS92]{DS92}
A.~Delshams and T.~M. Seara.
\newblock {An asymptotic expression for the splitting of separatrices of the
  rapidly forced pendulum}.
\newblock {\em Comm. Math. Phys.}, 150(3):433--463, 1992.

\bibitem[DS97]{DS97}
A.~Delshams and T.~M. Seara.
\newblock {Splitting of separatrices in {H}amiltonian systems with one and a
  half degrees of freedom}.
\newblock {\em Math. Phys. Electron. J.}, 3:Paper 4, 40 pp. (electronic), 1997.

\bibitem[Fon95]{Fo95}
E.~Fontich.
\newblock {Rapidly forced planar vector fields and splitting of separatrices}.
\newblock {\em J. Differential Equations}, 119(2):310--335, 1995.

\bibitem[FS90a]{FS90-2}
E.~Fontich and C.~Sim{\'o}.
\newblock {Invariant manifolds for near identity differentiable maps and
  splitting of separatrices}.
\newblock {\em Ergodic Theory Dynam. Systems}, 10(2):319--346, 1990.

\bibitem[FS90b]{FS90}
E.~Fontich and C.~Sim{\'o}.
\newblock {The splitting of separatrices for analytic diffeomorphisms}.
\newblock {\em Ergodic Theory Dynam. Systems}, 10(2):295--318, 1990.

\bibitem[Gav78]{G78}
N.~K. Gavrilov.
\newblock {On some bifurcations of an equilibrium with one zero and a pair of
  pure imaginary eigenvalues}.
\newblock In E.~A. Leontovich-Andronova, editor, {\em {Methods of the
  qualitative theory of differential equations ({R}ussian)}}, pages 33--40.
  Gorky State University, Gorky, 1978.

\bibitem[Gav85]{G85}
N.~K. Gavrilov.
\newblock {On bifurcations of codimension two equilibria of divergence-free
  vector fields}.
\newblock In E.~A. Leontovich-Andronova, editor, {\em {Methods of the
  qualitative theory of differential equations ({R}ussian)}}, pages 46--54.
  Gorky State University, Gorky, 1985.

\bibitem[GB08]{GB08}
V.~G. Gelfreich and N.~Br{\"a}nnstrom.
\newblock {Asymptotic series for the splitting of separatrices near a
  Hamiltonian bifurcation}.
\newblock {\em arXiv preprint arXiv:0806.2403}, 2008.

\bibitem[Gel94]{Gel94}
V.~G. Gelfreich.
\newblock {Separatrices splitting for the rapidly forced pendulum}.
\newblock In {\em {Seminar on Dynamical Systems (St.\ Petersburg, 1991)}},
  volume~12 of {\em {Progr. Nonlinear Differential Equations Appl.}}, pages
  47--67. Birkh{\"a}user, Basel, 1994.

\bibitem[Gel97a]{Gel97}
V.~G. Gelfreich.
\newblock {Melnikov method and exponentially small splitting of separatrices}.
\newblock {\em Phys. D}, 101(3-4):227--248, 1997.

\bibitem[Gel97b]{Gel97-2}
V.~G. Gelfreich.
\newblock {Reference systems for splittings of separatrices}.
\newblock {\em Nonlinearity}, 10(1):175--193, 1997.

\bibitem[Gel99]{Gel99}
V.~G. Gelfreich.
\newblock {A proof of the exponentially small transversality of the
  separatrices for the standard map}.
\newblock {\em Comm. Math. Phys.}, 201(1):155--216, 1999.

\bibitem[Gel00]{Gel00}
V.~G. Gelfreich.
\newblock {Separatrix splitting for a high-frequency perturbation of the
  pendulum}.
\newblock {\em Russ. J. Math. Phys.}, 7(1):48--71, 2000.

\bibitem[GG11]{GG11}
J.~P. Gaiv{\~a}o and V.~G. Gelfreich.
\newblock {Splitting of separatrices for the {H}amiltonian-{H}opf bifurcation
  with the {S}wift-{H}ohenberg equation as an example}.
\newblock {\em Nonlinearity}, 24(3):677--698, 2011.

\bibitem[GH90]{GH90}
J.~Guckenheimer and P.~Holmes.
\newblock {\em {Nonlinear oscillations, dynamical systems, and bifurcations of
  vector fields}}, volume~42.
\newblock Springer Verlag, 1990.

\bibitem[GOS10]{GOS10}
M.~Guardia, C.~Oliv{\'e}, and T.~M. Seara.
\newblock {Exponentially small splitting for the pendulum: a classical problem
  revisited}.
\newblock {\em J. Nonlinear Sci.}, 20(5):595--685, 2010.

\bibitem[GR83]{GR83}
N.~K. Gavrilov and N.~V. Roshchin.
\newblock {On stability of an equilibrium with one zero and a pair of pure
  imaginary eigenvalues}.
\newblock In E.~A. Leontovich-Andronova, editor, {\em {Methods of the
  qualitative theory of differential equations}}, pages 41--49. Gorky State
  University, Gorky, 1983.

\bibitem[GS01]{GS01}
V.~G. Gelfreich and D.~Sauzin.
\newblock {Borel summation and splitting of separatrices for the {H}{\'e}non
  map}.
\newblock {\em Ann. Inst. Fourier (Grenoble)}, 51(2):513--567, 2001.

\bibitem[GS08]{GS08}
V.~G. Gelfreich and C.~Sim{\'o}.
\newblock {High-precision computations of divergent asymptotic series and
  homoclinic phenomena}.
\newblock {\em Discrete Contin. Dyn. Syst. Ser. B}, 10(2-3):511--536, 2008.

\bibitem[GSV13]{GSV13}
V.~Gelfreich, C.~Simó, and A.~Vieiro.
\newblock Dynamics of symplectic maps near a double resonance.
\newblock {\em Physica D: Nonlinear Phenomena}, 243(1):92 -- 110, 2013.

\bibitem[Gua13]{Guar13}
M.~Guardia.
\newblock {Splitting of separatrices in the resonances of nearly integrable
  {H}amiltonian systems of one and a half degrees of freedom}.
\newblock {\em Discrete Contin. Dyn. Syst.}, 33(7):2829--2859, 2013.

\bibitem[Guc81]{Guc81}
J.~Guckenheimer.
\newblock {On a codimension two bifurcation}.
\newblock {\em Dynamical Systems and Turbulence, Warwick 1980}, pages 99--142,
  1981.

\bibitem[HMS88]{HMS88}
P.~Holmes, J.~Marsden, and J.~Scheurle.
\newblock {Exponentially small splittings of separatrices with applications to
  {KAM} theory and degenerate bifurcations}.
\newblock In {\em {Hamiltonian dynamical systems}}, volume~81 of {\em {Contemp.
  Math.}} 1988.

\bibitem[KS91]{KS91}
M.~D. Kruskal and H.~Segur.
\newblock {Asymptotics beyond all orders in a model of crystal growth}.
\newblock {\em Stud. Appl. Math.}, 85(2):129--181, 1991.

\bibitem[Laz]{Lazutkin}
V.~F. Lazutkin.
\newblock {Splitting of separatrices for the Chirikov standard map. VINITI
  6372/82, 1984}.
\newblock {\em Preprint (Russian)}.

\bibitem[{Laz}03]{La03}
J.~T. {Lazaro Ochoa}.
\newblock {\em {On Normal Forms and Splitting of Separatrices in Reversible
  Systems}}.
\newblock PhD thesis, Universitat Polit{\`e}cnica de Catalunya, 2003.

\bibitem[LMS03]{LMS03}
P.~Lochak, J.-P. Marco, and D.~Sauzin.
\newblock {On the splitting of invariant manifolds in multidimensional
  near-integrable {H}amiltonian systems}.
\newblock {\em Mem. Amer. Math. Soc.}, 163(775):viii+145, 2003.

\bibitem[Lom00]{Lom00}
E.~Lombardi.
\newblock {\em {Oscillatory integrals and phenomena beyond all algebraic
  orders}}, volume 1741 of {\em {Lecture Notes in Mathematics}}.
\newblock Springer-Verlag, Berlin, 2000.
\newblock With applications to homoclinic orbits in reversible systems.

\bibitem[Mel63]{Mel63}
V.~K. Mel{\cprime}nikov.
\newblock {On the stability of a center for time-periodic perturbations}.
\newblock {\em Trudy Moskov. Mat. Ob\v s\v c.}, 12:3--52, 1963.

\bibitem[MSS11]{MSS11}
Pau Mart{\'i}n, David Sauzin, and Tere~M. Seara.
\newblock {Exponentially small splitting of separatrices in the perturbed
  {M}c{M}illan map}.
\newblock {\em Discrete Contin. Dyn. Syst.}, 31(2):301--372, 2011.

\bibitem[MSV13]{MSV13}
Narc{\'i}s Miguel, Carles Sim{\'o}, and Arturo Vieiro.
\newblock From the h{\'e}non conservative map to the chirikov standard map for
  large parameter values.
\newblock {\em Regular and Chaotic Dynamics}, 18(5):469--489, 2013.

\bibitem[Ne{\u\i}84]{Neishtadt}
A.~I. Ne{\u\i}shtadt.
\newblock {The separation of motions in systems with rapidly rotating phase}.
\newblock {\em Prikl. Mat. Mekh.}, 48(2):197--204, 1984.

\bibitem[OSS03]{OSS03}
C.~Oliv{\'e}, D.~Sauzin, and T.~M. Seara.
\newblock {Resurgence in a {H}amilton-{J}acobi equation}.
\newblock In {\em {Proceedings of the {I}nternational {C}onference in {H}onor
  of {F}r{\'e}d{\'e}ric {P}ham ({N}ice, 2002)}}, volume~53, pages 1185--1235,
  2003.

\bibitem[Poi90]{Poincare}
H.~Poincar{\'e}.
\newblock {Sur le probl{\`e}me des trois corps et les {\'e}quations de la
  dynamique}.
\newblock {\em Acta mathematica}, 13(1):A3--A270, 1890.

\bibitem[Sau01]{Sau01}
D.~Sauzin.
\newblock {A new method for measuring the splitting of invariant manifolds}.
\newblock {\em Ann. Sci. {\'E}cole Norm. Sup. (4)}, 34(2):159--221, 2001.

\bibitem[SV09]{SV09}
C~Simó and A~Vieiro.
\newblock Resonant zones, inner and outer splittings in generic and low order
  resonances of area preserving maps.
\newblock {\em Nonlinearity}, 22(5):1191, 2009.

\bibitem[Tak73]{Tak73}
F.~Takens.
\newblock {A nonstabilizable jet of a singularity of a vector field}.
\newblock In {\em {Dynamical systems ({P}roc. {S}ympos., {U}niv. {B}ahia,
  {S}alvador, 1971)}}, pages 583--597. Academic Press, New York, 1973.

\bibitem[Tak74]{Tak74}
F.~Takens.
\newblock {Singularities of vector fields}.
\newblock {\em Publications Math{\'e}matiques de l'IHES}, 43(1):47--100, 1974.

\bibitem[Tre97]{Tre97}
D.~V. Treschev.
\newblock {Splitting of separatrices for a pendulum with rapidly oscillating
  suspension point}.
\newblock {\em Russian J. Math. Phys.}, 5(1):63--98 (1998), 1997.

\bibitem[\v{S}65]{Shil65}
L.~P. \v{S}il{\cprime}nikov.
\newblock {A case of the existence of a denumerable set of periodic motions}.
\newblock {\em Dokl. Akad. Nauk SSSR}, 160:558--561, 1965.

\end{thebibliography}
\bibliographystyle{alpha}

\end{document}